\documentclass[preprint]{elsarticle}
\usepackage{amssymb,enumerate, amsmath}
\usepackage{cancel}
\usepackage{graphicx}
\usepackage{multirow}
\usepackage{longtable}
\usepackage{setspace}
\usepackage{epsfig,latexsym,graphics,bm,pifont,subfigure}
\usepackage{dcolumn}% Align table columns on decimal point
\usepackage{afterpage}
\usepackage{color}

        \newtheorem{thm}{Theorem}[section]
	\newtheorem{prop}[thm]{Proposition}
	\newtheorem{cor}[thm]{Corollary}
	
	\newtheorem{lem}[thm]{Lemma}
	\newtheorem{rem}[thm]{Remark}
	\newtheorem{mydef}[thm]{Definition}
	\newproof{proof}{Proof}
         \newproof{pot}{Proof of Theorem \ref{thm2}}
	
	\newcommand{\vectornorm}[1]{\parallel#1\parallel}

\newcommand{\Grad}{\nabla}

\newcommand{\Div}{\nabla \cdot}
\newcommand{\Gradh}{\nabla_{h}}
\newcommand{\Divh}{\nabla_{h} \cdot}

\newcommand{\bu}{{\bf u}}

\newcommand{\bv}{{\bf v}}
  
\newcommand{\br}{{\bf r}}
\newcommand{\bg}{{\bf g}}

    \newcommand{\cip}[2]{\left( #1 \middle| #2 \right)}
    \newcommand{\eipns}[2]{\left[ #1 \middle\| #2 \right]_{\rm ns}}
    \newcommand{\eipew}[2]{\left[ #1 \middle\| #2 \right]_{\rm ew}}
    
    \newcommand{\ciptwo}[2]{\left( #1 \middle\| #2 \right)}
     \newcommand{\hf}{\frac{1}{2}}
      \newcommand{\Dh}{\Delta_h}

\begin{document}

\begin{frontmatter}
\title{ Energy Stable Multigrid Method for Local and Non-local Hydrodynamic Models for Freezing}
          
\author[uci]{Arvind Baskaran}
\ead{baskaran@math.uci.edu}

\author[uci]{Zhen Guan}
\ead{guanz2@math.uci.edu}

\author[uci]{John Lowengrub}
\ead{lowengrb@math.uci.edu}

\address[uci]{Department of Mathematics, University of California Irvine, Irvine, CA 92697 USA}

         %{Put the URL for your home page here if you have one}

          %Use \thanks statements for acknowledgements of grants and
          %support. They will appear below all the authors' addresses, so be
          %specific about which author is thanking whom:

          %\thanks{}

%\pagestyle{myheadings} \markboth{Energy Stable Method for Hydrodynamic Models for Freezing}{Arvind Baskaran, Zhen Guan \& John Lowengrub}
%\maketitle

\begin{abstract}
In this paper we present a numerical method for hydrodynamic models that arise from time dependent density functional theories of freezing.
The models take the form of compressible Navier-Stokes equations whose pressure is determined by the variational derivative 
of a free energy, which is a functional of the density field.
We present unconditionally energy stable and mass conserving implicit finite difference methods for the models.
The methods are based on a convex splitting of the free energy and that ensures that a discrete energy is non-increasing
for any choice of time and space step. The methods are applicable to a large class of models, including
both local and non-local free energy functionals. The theoretical basis for the numerical method is presented in a general context. 
The method is applied to problems using two specific free energy functionals: one local and one non-local functional.
A nonlinear multigrid method is used to solve the numerical method, which is nonlinear at the implicit time step.
The non-local functional, which is a convolution operator, is approximated using the Discrete Fourier Transform.
Numerical simulations that confirm the stability  and accuracy of the numerical method are presented. 
\end{abstract}
\begin{keyword}
Classical Density Functional Theory \sep Phase Field Crystal \sep Compressible Navier-Stokes \sep Convex Splitting, Finite Difference Methods \sep Energy stability
\end{keyword}

\end{frontmatter}

%%%%%%%%%%%%%%%%%%%%%%%%%%%%%%%%%%%%%%%%%%%%%%%

%%%%%%%%%%%%%%%%%%%%%%%%%%%%%%%%%%%%%%%%%%%%%%%
\section{Introduction}
Solid liquid phase transitions are of great scientific interest.
The equilibrium properties of this physical process are fairly well understood in
the context of classical density functional theory (CDFT) of freezing \cite{Hansen2006,Lutsko2010}.
This theory characterizes the equilibrium state of a pairwise interacting set of particles 
in terms of the one particle density field.
The density function $\rho$ represents the spatial distribution of particles, i.e., the probability of finding a particle
at some point in space.
This function at equilibrium is represented as a minimizer of a free energy, which in turn is a functional of the density.
At equilibrium this function admits two forms of solutions corresponding to a homogeneous distribution known
as the liquid phase and the inhomogeneous distribution known as the solid phase.
The inhomogeneous distribution typically consists of peaks located on an ordered Bravais lattice
representing the probable locations of atoms.
This approach is attractive for the reason that the equilibrium solid phase carries the information
about the lattice symmetries.
A non-equilibrium time dependent theory that characterizes the non-equilibrium distribution has the 
potential to capture lattice-dependent anisotropic effects such as defect formation and microstructure evolution.
Such models have many potential applications in the area of materials modeling.
Consequently, the development of time dependent models has been the focus of recent research \cite{Marconi1999,2012Kaliadasis,Espanol2009,Yoshimori2005,Archer2009, Lutsko2012,Chavanis2011,Baskaran2014}.

Given a free energy functional $\mathcal{F}[\rho]$ the simplest dynamical equations for the evolution of the density field
 are to use gradient dynamics on the free energy surface \cite{Marconi1999,2012Kaliadasis,Espanol2009,Yoshimori2005}.
While this approach is appealing it has several shortcomings including the inability to capture the inertial effects and the effect of flow on the 
phase transition. Hydrodynamic coupling has been incorporated in models involving colloidal suspensions \cite{Archer2009, Lutsko2012} and in models that describe the
freezing of a dense hard sphere gas \cite{Baskaran2014}. 
% 
%
%%FINDME
%
%To remedy this a hydrodynamic model for isothermal solid liquid phase transitions was derived from a well known kinetic theory in \cite{Baskaran2014}.
%The hydrodynamic theory can be derived through simpler local equilibrium approximations \cite{Archer2009, Lutsko2012}
%for systems like colloidal suspension where the system is typically over-damped.
%The authors in \cite{Baskaran2014} used a modified Chapmann-Enskog procedure to derive the hydrodynamic model from the revised Enskog kinetic theory \cite{VanBeijeren1973}.
In all these approaches the hydrodynamic model takes the form of a compressible Navier-Stokes-like equations where the 
pressure ($p$) is defined by the equation of state $\Grad p := \rho \Grad \frac{\delta \mathcal{F}}{\delta \rho }$.
This can be easily seen from a thermodynamic point of view through the Gibbs Duhem relation $dp = \rho d \mu$ where
$\mu$ is the chemical potential taken to be the variational derivative of the free energy \cite{Lutsko2012}.
Models for hydrodynamic coupling with solid-liquid phase transitions thus take the form
\begin{equation}
\label{Eqn:Hydro_damp}
\begin{array}{ll}
\displaystyle\partial _{t}\rho +\nabla \cdot \left(\rho \mathbf{u}\right)
& =0, \\
\displaystyle\partial _{t}(\rho \mathbf{u})+\nabla \cdot
(\rho \mathbf{u}\otimes \mathbf{u}) & =\displaystyle-\rho \nabla  \left( \frac{\delta \mathcal{F}}{\delta \rho }\right) +   \gamma \Grad^2 \bu ,
\end{array}
\end{equation}%
where $\rho(\br,t) : \Omega \times [0,\infty) \to \mathbb{R}^d$, $\bu (\br,t) : \Omega\times [0,\infty) \to \mathbb{R}^d$,  $\Omega \subset \mathbb{R}^d $, $\mathcal{F}[\rho]$ is a functional of $\rho$  which may be local or non-local in nature, $\gamma  >0$ is  a viscosity coefficient and $d$ is the dimension of the system.
%The dissipation tensor $\mathcal{D}$ is defined as
%\begin{equation}
%\mathcal{D} := \gamma_1 \left( \Grad \bu + \Grad \bu^{T} \right) - \gamma_2( \Div \bu ) \mathbf{I},
%\end{equation}
 %where $\mathbf{I}$ is the identity tensor of rank $d$ and  $\gamma_1 >0$ and $\gamma_2 > 0$ are nondimensional the shear and bulk coefficients of viscosity.
It is easy to see that the model in Eq. (\ref{Eqn:Hydro_damp}) admits an energy 
\begin{equation}
\label{total_energy}
\mathcal{E} [\rho, \bu] = \frac{1}{2} \int_{\Omega} \rho \mid \bu \mid^2 d \br + \mathcal{F}[\rho],
\end{equation}
such that
\begin{equation}
\label{energy_derivative}
\frac{d \mathcal{E}}{dt} \leq 0.
\end{equation}
 The proof of Eq. (\ref{energy_derivative})  is presented in \ref{general_energy}.
%In the above equation $ \mathcal{D}:\mathcal{B} := \mathcal{D}_{ij}\mathcal{B}_{ij}$ where the repeated indices are summed.
The long time equilibrium for the above system corresponds to $\delta_{\rho} \mathcal{F}[\rho] =0, \bu =0$.
Thus the equilibrium properties of the system are determined by the functional $\mathcal{F}[\rho]$ which represents the free energy functional.
The specific functional form of the free energy $\mathcal{F}[\rho]$ determines the theory and the phase diagram 
of the model. One must appeal to CDFT or a similar approximation for the specific form of the free energy.

It is worth noting that the time dynamics of the system however is much richer than a standard gradient descent model \cite{Baskaran2014}.
The first set of simulations using this type of model was presented in \cite{Baskaran2014}. It was observed that while the total energy $\mathcal{E}[\rho,\bu]$
and the free energy $\mathcal{F}[\rho]$ were non-increasing functions of time, the kinetic energy and the individual components
of the free energy were not. However no details about the numerical method were given in \cite{Baskaran2014}. Here, we present the first
energy-stable method to solve such coupled systems.
%To this date a reliable numerical method has not been proposed that will allow one to study this model in detail. 
%It is our aim to do that.

The physical properties and relevance of the model lies in the free energy functional $\mathcal{F}[\rho]$ used.
A large number of models for the free energy exist in literature.
These models include CDFT \cite{Hansen2006,Lutsko2010}
derived in the framework of equilibrium statistical mechanics, as well as 
phase field crystal (PFC) models \cite{PFC_book},
which are derived using phenomenological theories motivated by CDFT.
The energy in CDFT is typically a non-local functional of the density while the energy in the PFC is a weakly non-local 
functional that depends only on gradients of the density.  In addition, in CDFT the ideal gas part of
the free energy contains a logarithm. In CDFT, the combination of the logarithmic part of the free energy with the nonlocal interaction energy induces much sharper peaks in the density field than the peaks obtained using the PFC model where the logarithmic term is replaced by a polynomial and a gradient approximation is used for the interaction energy (as described in Appendix B). Thus, the CDFT system requires a finer grid to resolve the dynamics.

These free energy models have been very popular and capture phase transitions at  the atomic length scales,
but on diffusion time scales, thus enabling the simulation of microstructure evolution at much longer time scales
than can be captured using molecular dynamics or other stochastic approaches \cite{PFC_book}.
%
% while the time dependent hydrodynamic model captures the diffusion time scales (see \cite{Baskaran2014}).
%Thus allowing one to capture atomistic and micro structural detail on diffusion time scales.
This makes the hydrodynamic theory considered here an extremely valuable tool in investigating non-equilibrium
effects such as the effect of flow on the phase transition.
Thus developing an efficient and stable numerical method for this class of models is of critical importance.

In this work we use the convex splitting framework introduced  by Eyre \cite{Eyre} to develop stable numerical schemes.
The convex splitting framework has been successfully applied to the PFC-based models \cite {WISE2009,WISE2009_2,MPFC2011,MPFC2013,MPFC2013_2}
and to non-local CDFT-like models \cite{Zhen1,Zhen2}.
However these approaches consider conserved gradient descent models without hydrodynamics.
The current literature on hydrodynamic models driven by free energy gradients
mainly considers incompressible multiphase fluids using the Cahn-Hilliard free energy (see \cite{Kim2012} for a comprehensive review).
We note that in \cite{Axel2015}, the incompressible Navier-Stokes system was coupled with the PFC model.
A similar approach using a phenomenological PFC-like free energy driven hydrodynamic model was presented in \cite{Laszlo2014}
however no numerical methods were presented. While there has been much work in the literature on the development of numerical methods for compressible fluids driven by van der Waals and Cahn-Hilliard-like free energies, to the best of our knowledge the development of accurate and energy-stable numerical methods for the case of compressible fluids driven by  gradients of PFC and CDFT free energies, which are higher order and nonlocal (CDFT),  has not been addressed.

The paper is structured as follows. In Section \ref{section:free_energy} we present the description of the model 
and the associated free energies. In Section \ref{section:CS} we present the basic ideas of convex splitting 
and outline the need for a comprehensive treatment for the various free energy models discussed in this work. This section
also presents the basic energy estimates required to develop an unconditionally energy stable method for 
the hydrodynamic theory of freezing. An unconditionally energy stable conservative discretization of time 
is presented in Section \ref{section:method_semi},  where the properties of scheme are proven. The corresponding 
fully discrete numerical method is presented in Section \ref{section:method_full}.
Finally the numerical results are presented in Section \ref{section:numerical}, where the properties of the numerical method 
are explored numerically. In Section \ref{section:conclusion}, conclusions are drawn and future work is discussed.
Details of the model derivation, definitions of the discrete operators, and the nonlinear multigrid method used
in the simulations are presented in the
Appendices.

 %%%%%%%%%%%%%%%%%%%%%%%%%%%%%%%%%%%%%%%%%%%%%%%%%

%%%%%%%%%%% Model Equations %%%%%%%%%%%%%%%%%%%%%%%%%%%%%

%%%%%%%%%%%%%%%%%%%%%%%%%%%%%%%%%%%%%%%%%%%%%%%%%

\section{Hydrodynamic Model and the Associated Free Energies} \label{section:free_energy}
In the reminder of the paper we will restrict ourselves to 2 dimensions.
However we note that the ideas, numerical methods, theorems and proofs extend to 3 dimensions in a straightforward manner.
We will consider a spatial domain $\Omega : = [0,L_x) \times [0, L_y)$ with periodic boundary conditions.
The results however extend to other boundary conditions such as homogenous Dirichlet (no-slip) solid wall boundary conditions and Neumann boundary conditions
 (see \ref{BC}).
The hydrodynamic equations Eq. (\ref{Eqn:Hydro_damp}) are first rewritten in the primitive variable form for convenience as follows:
\begin{eqnarray}
\label{hydro_prim}
\displaystyle\partial _{t}\rho +\Grad \cdot \left(\rho \mathbf{u}\right) =0, \\
\displaystyle \rho \left( \partial _{t} \mathbf{u} + \bu \cdot \Grad \bu \right) =\displaystyle-\rho 
\Grad \left( \frac{\delta \mathcal{F}}{\delta \rho }\right) + \gamma \Div \mathcal{D}.
\end{eqnarray}%
 The dissipation tensor $\mathcal{D}$ is defined as
\begin{equation}
\mathcal{D} := \left( \Grad \bu + \Grad \bu^{T} \right) - ( \Div \bu ) \mathbf{I},
\end{equation}
 where $\mathbf{I}$ is the identity tensor of rank $2$. 
%where we have assumed for simplicity that $\gamma_1 = \gamma_2 = 1$. 
In this work we will consider the following two specific models for the free energy.
\subsection{Classical Density Functional Theory}
Classical density functional theory typically relies on a free energy that is a non-local functional of 
the density $\rho$.
A typical form for the free energy is given by
\begin{equation}
\mathcal{F}_{CDFT} [ \rho ] = \mathcal{F}_{id} [ \rho] + \mathcal{F}_{ex} [\rho],
\label{energy_DFT}
\end{equation}
where
\begin{equation}
\mathcal{F}_{id} [ \rho ] := \int_{\Omega} \rho (\ln ( \rho) - 1) d \br,
\end{equation}
is the ideal gas part of the free energy (e.g., describes mixing)
and 
\begin{equation}
\mathcal{F}_{ex} [ \rho ] := -\frac{1}{2} \int_{\Omega} \rho ( J * \rho) d \br,
\end{equation}
which is the excess free energy functional (e.g., describes interactions), with
\begin{equation}
(J * \rho) := \int_{\Omega} J({\bf x} - {\bf y}) \rho( {\bf y} ) d {\bf y}.
\end{equation}
The convolution kernel $J$ represents different quantities in different approximations of 
the density functional theory. For example, in the context of the local density approximation 
 $J(\br - \br') = \tilde{J}( \mid \br - \br' \mid)$, represents the pairwise inter-atomic  potential energy associated with
two particles that located at $\br$ and $\br'$ respectively (see \cite{Hansen2006}). 
% Note that the $J$ on the
%left hand side of the equality is a function of a vector while that on the right hand side is a function of a scalar. This
%follows standard notation in the field. 
In the context of the Ramakrishnan-Yousseff density functional theory the quantity $J$ represents the two particle
direct correlation function \cite{Ramakrishnan1979}. 
While the convolution kernel has different physical meanings for different physical approximations the form of the 
excess free energy is consistent across a large class of density functional theories and    
is worth addressing in general.

For the reminder of this paper we will assume that 
the convolution kernel $J$ is a sufficiently regular and  $\Omega$-periodic function such that
\begin{itemize}
\item $J = J_c - J_e$, where $J_c, J_e $ are sufficiently regular, $\Omega$-periodic and  pointwise non-negative.
\item $J_c$ and $J_e$ are even, i.e.,  $J_{\alpha} (- \br ) = J_{\alpha} (\br)$ where $\br \in \mathbb{R}^2$, $\alpha = c,e$.
\item $\int_{\Omega} J(\br) d \br <0$.
\end{itemize}
It is worth noting that the above assumptions are in general true. For example, the nontrivial assumption that the kernel is even 
follows from radial symmetry of the interaction potentials for both the local density approximation and the Ramakrishnan-Yousseff formalism \cite{Hansen2006}.

This free energy satisfies a convex splitting of the form  $\mathcal{F}_{CDFT} : = \mathcal{F}_{CDFT,c} - \mathcal{F}_{CDFT,e}$
where $\mathcal{F}_{CDFT,c}$ and $\mathcal{F}_{CDFT,e}$ are convex functionals defined as:

\begin{equation}
\begin{array}{c}
\displaystyle \mathcal{F}_{CDFT,c} [ \rho ] := \int_{\Omega}  \left\{ \rho ( \ln \rho -1) + (J_e * 1) \rho^2  \right\} d\br  ;\\
\displaystyle \mathcal{F}_{CDFT,e} [ \rho ] := \int_{\Omega}  \left\{   \frac{1}{2} \rho ( J * \rho)  + (J_e * 1) \rho^2 \right\}d \br .
\end{array}
\label{DDFT_CS}
\end{equation}  
Here we say a functional $\mathcal{F}[\rho]$  is convex in a Hilbert space $H$ (consisting of sufficiently regular functions) if for any
$\rho \in H$ that $\lim_{\epsilon \to 0} \frac{d^2 }{d \epsilon^2} \mathcal{F}[\rho +\epsilon v] \geq 0$ for all $v \in H$.
This convex splitting is not unique and other forms of convex splitting can be obtained (see \cite{Zhen1,Zhen2}).
This specific convex splitting is similar to the one proposed in \cite{Zhen1} as it has certain advantages, 
which will become apparent in the later sections.
It is worth noting that in \cite{Zhen1} the authors considered the case where $(J*1)>0$.
However in CDFT $(J *1)$ is usually negative as assumed here. %in \cite{Zhen1}.
For example in the Ramakrishnan-Youssef \cite{Ramakrishnan1979} CDFT model
$(J*1)$ is related to the isothermal compressibility of the liquid and is negative \cite{Hansen2006}.

%%%%%%%%%%%%%%%%%%%%%%%%%%%%%%%%%%%%%%%%%%%%%%%

%%%%%%%%%    %%%%%%%%%%%%%%%%

%%%%%%%%%%%%%%%%%%%%%%%%%%%%%%%%%%%%%%%%%%%%%%%

\subsection{Phase Field Crystal Model}
The phase field crystal (PFC) approximation was first proposed by Elder et al \cite{Elder2004}. While the PFC was first derived using a phenomenological approach, the PFC has since been 
re-interpreted as an approximation to dynamic density functional theory (DDFT), which is a time-dependent counterpart
to CDFT \cite{Elder2004,VanTeeffelen2009,Akusti,PFC_book}. 
Typically, PFC formulations use a scaled version of the deviation of the density from a reference density rather than the density field itself (see \ref{PFC_derivation}).  As a result, the density is not always non-negative in PFC models. Here, we use a formulation of the PFC that uses the particle density field directly and maintains its non-negativity.

%
%
%
%
%The relationship of the PFC to CDFT has been explored in the literature, e.g. \cite{Elder2004,VanTeeffelen2009}  .
%The physical implications of the PFC approximation on the equilibrium properties are quantified and discussed in Van Teeffelen et al \cite{VanTeeffelen2009}.
%The PFC model proposed in the literature \cite{PFC_book} is derived from a phenomenological approach
%and cannot be directly used in this work.
%This is because the free energy is not represented as a functional of the density field.
%The modified PFC approach used in this work appeals to the ideas due 
% to Van Teeffelen et al \cite{VanTeeffelen2009} and Jaatinen et al \cite{Akusti} to interpret PFC as an approximation to 
%the density functional theory approach.
%The derivation is presented in the Appendix \ref{PFC_derivation}.
The PFC energy functional we consider is given by
\begin{equation}
\mathcal{F}_{PFC} ( \rho ) = \int_{\Omega} \left \{ \frac{1}{12} \left( \rho-\frac{3}{2} \right)^4 + \frac{\alpha}{2} \left( \rho-\frac{3}{2} \right)^2 - |\nabla \rho |^2 + \frac{1}{2} (\Delta \rho)^2 \right \} d\mathbf{r} ,
\label{energy}
\end{equation}
with
\begin{equation}
\frac{\delta \mathcal{F}_{PFC}}{\delta \rho}    =  \frac{1}{3}\left( \rho-\frac{3}{2} \right)^3 + \alpha \left( \rho-\frac{3}{2} \right) + 2 \Delta \rho + \Delta^2 \rho ,
\label{energy_PFC}
\end{equation}
where $\alpha$ is a free parameter that determines the strength of interactions between the particles and hence the elastic constants of the solid phase. 
We refer the reader to \ref{PFC_derivation} for details regarding the relation of the above free energy 
and more standard formulations of the PFC energy.
We note that this free energy admits a convex splitting of the form $\mathcal{F}_{PFC} : = \mathcal{F}_{PFC,c} - \mathcal{F}_{PFC,e}$,
where $\mathcal{F}_{PFC,c}$ and $\mathcal{F}_{PFC,e}$ are convex functional defined as:
\begin{equation}
\mathcal{F}_{PFC,c} [ \rho ] := \int_{\Omega} \left \{ \frac{1}{12} \left( \rho-\frac{3}{2} \right)^4 + \frac{\alpha}{2} \left( \rho-\frac{3}{2} \right)^2 + \frac{1}{2} (\Delta \rho)^2 \right \} d\mathbf{r},
\label{CS_PFC1}
\end{equation}
 and
\begin{equation}
\mathcal{F}_{PFC,e} [ \rho ] := \int_{\Omega}   \mid \nabla \rho \mid^2 d\mathbf{r}.
\label{CS_PFC2}
\end{equation}

 %%%%%%%%%%%%%%%%%%%%%%%%%%%%%%%%%%%%%%%%%%%%%%%%%%%%%%%%
 
 %%%%%%%%%% General Framework of CS

 %%%%%%%%%%%%%%%%%%%%%%%%%%%%%%%%%%%%%%%%%%%%%%%%%%%%%%%%%
 
 \section{General Framework of the Convex Splitting Schemes}\label{section:CS}
In this section we discuss in considerable generality the ideas involved in convex splitting schemes for the class of models 
described so far.
%We are interested in studying energy driven systems that are modeled by free energy gradients like  the models mentioned in
%the previous section.
%
Previously, Wise et al in \cite{WISE2009} considered  a class of local free energies like the PFC model that satisfies the form 
$\mathcal{F} : = \mathcal{F}_c - \mathcal{F}_e$, where $\mathcal{F}_e [\rho]$ and $\mathcal{F}_c [\rho]$ are convex functionals of $\rho$ with a specific 
functional form.
It was assumed that  $\mathcal{F}_e [\rho]$ and $\mathcal{F}_c [\rho]$ are integrals of energy densities $f_{c}(\rho, \partial_x \rho, \partial_y \rho, \Delta \rho)$ and $f_{e}(\rho, \partial_x \rho, \partial_y \rho, \Delta \rho)$
respectively. It was further assumed in \cite{WISE2009} that the functions $f_c$ and $f_e$ are convex in their arguments respectively.
The authors derived an energy estimate for this special case as summarized in the following theorem.
	\begin{thm}
	\label{cor_estimate3}
	Given a free energy $\mathcal{F}[\rho]$ that admits a convex splitting of the form  
	$\mathcal{F} : = \mathcal{F}_c - \mathcal{F}_e$ where $\mathcal{F}_e [\rho]$ and $\mathcal{F}_c [\rho]$ are convex functionals of $\rho$
	of the form	
         \begin{equation}
	\mathcal{F}_c[\rho] = \int_{\Omega} f_{c}(\rho, \partial_x \rho, \partial_y \rho, \Delta \rho)
	\label{energydensity_form1}
	\end{equation}
	and 
	\begin{equation}
	\mathcal{F}_e[\rho] = \int_{\Omega} f_{e}(\rho, \partial_x \rho, \partial_y \rho, \Delta \rho)
	\label{energydensity_form2}
	\end{equation}
	such that $f_c$ and $f_e$ are convex in their arguments.
	 Further if $\rho$ is sufficiently  regular and periodic in $\Omega$ and $\partial_x \rho, \partial_y \rho$ are also periodic in $\Omega$, then
	\begin{equation}
	\mathcal{F}[\phi] - \mathcal{F}[\psi] \leq ( \delta_{\phi} \mathcal{F}_c[\phi] - \delta_{\psi} \mathcal{F}_e[\psi],\phi - \psi)_2   ,       
	\label{ineq_CS2}
	\end{equation}
	where $(\cdot,\cdot)_2$ is the usual $L_2$ inner product and $\delta_{\rho}$ represents the variational derivative with respect to $\rho$.
	\end{thm}
	
\begin{rem} \label{time_estimate}
Note that the estimate Eq. (\ref{ineq_CS2}) can be immediately translated into an energy estimate. 
Consider a discretization of time $t^k \in [0,T]$ such that $t^k := 0 + k s$ where $s>0$, 
$s \in \mathbb{R}$ and $k \in \mathbb{Z}^+$.
Then we have the energy decay estimate 
\begin{equation}
\mathcal{F}[\rho^{k+1}] - \mathcal{F}[\rho^k] \leq ( \delta_{\rho} \mathcal{F}_c[\rho^{k+1}] - \delta_{\rho} \mathcal{F}_e[\rho^k],\rho^{k+1} - \rho^k)_2,
\end{equation}
where $\rho^k : = \rho({\bf x},t^k)$.
This energy decay estimate can be used to show energy stability for convex splitting schemes.

\end{rem}	
The energy estimate however is limited to the class of free energy functionals considered in Theorem \ref{cor_estimate3}.
The CDFT free energy presented in Eq. (\ref{energy_DFT}), which is a nonlocal functional of the density field, is an example
of a free energy functional out of the scope of Theorem \ref{cor_estimate3}.
Nonlocal free energies were considered in \cite{Zhen1,Zhen2}, where a convex splitting scheme was presented taking advantage 
of the specific form of the CDFT free energy functional.

In what follows we consider a general approach to convex splitting schemes and the associated energy decay estimates
that seamlessly include a much larger class of free energies including the PFC and CDFT.
Before we proceed further we define the notion of a proper convex splitting:
\begin{mydef}
\label{PCS}
Consider a functional $\mathcal{F}[\rho] : H \subset L_{per}^2(\Omega) \to \mathbb{R}$, where $H \subset L_{per}^2(\Omega)$ is a Hilbert space of sufficiently regular periodic functions 
$\rho: \Omega \to \mathbb{R}$ with $\Omega = [0,L_x) \times [0,L_y)$.
$F[\rho]$ is said to admit a proper convex splitting  if $\mathcal{F} : = \mathcal{F}_c - \mathcal{F}_e$ such that
\begin{equation}
\frac{d^2}{d \epsilon^2} \mathcal{F}_\alpha [ \rho +\epsilon v] \geq 0 \qquad \forall \rho, v \in H, \epsilon \in \mathbb{R}, \alpha \in \{c,e\}.
\label{condition2}
\end{equation}
\end{mydef}
Note that the definition above is stronger than a simple convex splitting, where $\mathcal{F}_c$ and $\mathcal{F}_e$ are convex functionals, i.e.,
\begin{equation}
\lim_{\epsilon \to 0}\frac{d^2}{d \epsilon^2} \mathcal{F}_\alpha [ \rho +\epsilon v] \geq 0 \qquad \forall \rho, v \in H, \epsilon \in \mathbb{R}, \alpha \in \{c,e\}.
\label{condition22a}
\end{equation}
If $\mathcal{F} : = \mathcal{F}_c - \mathcal{F}_e$ is a proper convex splitting of $\mathcal{F}$ then the $\mathcal{F}_c$ and $\mathcal{F}_e$ are indeed convex functionals.
However the converse is not necessarily true.
Taking advantage of the proper convex splitting of $\mathcal{F}$ we can obtain the energy decay estimate presented in Theorem \ref{cor_estimate3}.

	\begin{thm}
	\label{thm_estimate2}
	Given a free energy $\mathcal{F}[\rho]$ that admits a proper convex splitting of the form  
	$\mathcal{F} : = \mathcal{F}_c - \mathcal{F}_e$, where $\mathcal{F}_e [\rho]$ and $\mathcal{F}_c [\rho]$ are  functionals of $\rho$
         such that 	
         \begin{equation}
	\frac{d^2}{d \epsilon^2} \mathcal{F}_\alpha [ \rho +\epsilon v] \geq 0 \qquad \forall \rho, v \in H, \epsilon \in \mathbb{R}, \alpha \in \{ c,e \},
	\end{equation}
	 then
	\begin{equation}
	\mathcal{F}[\phi] - \mathcal{F}[\psi] \leq ( \delta_{\phi} \mathcal{F}_c[\phi] - \delta_{\psi} \mathcal{F}_e[\psi],\phi - \psi)_2,          
	\label{ineq_CS}
	\end{equation}
	where $(\cdot,\cdot)_2$ is the usual $L_2$ inner product and $\delta_{\rho}$ represents the variational derivative with respect to $\rho$.
	\end{thm}
	\begin{proof}
	The proof is presented in \ref{Proof_Estimate}.
	\end{proof}

In order to demonstrate the power of this decay estimate let us first consider the simple gradient model :
\begin{equation}
\partial_t \rho = \Grad^2 \delta_{\rho} \mathcal{F}.
\end{equation}
Using the energy estimate in Remark. \ref{time_estimate} we can construct an unconditionally energy stable
conservative scheme for gradient dynamics with free energies that satisfy a proper convex splitting.
This is summarized in the following theorem.

\begin{thm}
\label{thm_grad}
Suppose that $\Omega = [0,L_x)\times[0,L_y)$ and $\rho^k : \Omega \to \mathbb{R}^2$ and $\bu^k : \Omega \to \mathbb{R}^2$ for all $k\geq 0, k\in \mathbb{Z}^+$ are
periodic and sufficiently regular. Then
the solution to the scheme 
\begin{equation}
\rho^{k+1} - \rho^k = s \Grad^2 (\delta_{\rho} \mathcal{F}_c[\rho^{k+1}] - \delta_{\rho} \mathcal{F}_e[\rho^{k}] ) ,
\label{gradient_scheme}
\end{equation} 
conserves mass, i.e.,
\begin{equation}
\left( \rho^{k+1},1\right)_2 = \left( \rho^{k},1\right)_2, \qquad \forall k\geq 0.
\end{equation}
If in addition $\mathcal{F}$ admits a proper convex splitting  such that
\begin{equation}
\mathcal{F}[\phi] - \mathcal{F}[\psi] \leq ( \delta_{\phi} \mathcal{F}_c[\phi] - \delta_{\psi} \mathcal{F}_e[\psi],\phi - \psi)_2,
\label{ineq1b}
\end{equation}
then the scheme is unconditionally energy stable and
\begin{equation}
\mathcal{F}[\rho^{k+1}] \leq \mathcal{F}[\rho^k] \quad  \forall k\geq 0, 
\end{equation}  
for any $s>0$.
\end{thm}
\begin{proof}
Conservation follows by integrating Eq. (\ref{gradient_scheme}) over $\Omega$ to obtain 
\begin{equation}
(\rho^{k+1},1)_2= (\rho^k,1)_2.
\end{equation}
Choosing $\phi =\rho^{k+1}$ and $\psi = \rho^{k}$ in Eq. (\ref{ineq1b}) we have the  estimate :
\begin{equation}
\mathcal{F}[\rho^{k+1}] - \mathcal{F}[\rho^k] \leq ( \delta_{\rho} \mathcal{F}_c[\rho^{k+1}] - \delta_{\rho} \mathcal{F}_e[\rho^k],\rho^{k+1} - \rho^k)_2.
\label{ineq2}
\end{equation}
Now using Eq. (\ref{gradient_scheme}) to replace $\rho^{k+1} - \rho^k$ we have
\begin{equation}
\mathcal{F}[\rho^{k+1}] - \mathcal{F}[\rho^k] \leq -s \parallel \Grad \left( \delta_{\rho} \mathcal{F}_c[\rho^{k+1}] - \delta_{\rho} \mathcal{F}_e[\rho^k] \right)\parallel_2^2 \leq 0.
\end{equation}
In the above expression $\parallel \cdot \parallel_2^2 := (\cdot, \cdot)_2$.
\end{proof}

\begin{rem}
The above theorem seamlessly includes both the local and non-local free energy functionals discussed in Section  \ref{section:free_energy}, i.e., 
the cases of PFC and CDFT,
thus serving as a stronger result than the one presented in Wise et al \cite{WISE2009}.
Further Theorem \ref{cor_estimate3} follows immediately from Theorem \ref{thm_estimate2}.
To see this note that convexity of $f_c$ and $f_e$ in their arguments implies $\mathcal{F}_c$ and $\mathcal{F}_e$ satisfy 
Eq. (\ref{condition2}). Hence $\mathcal{F} = \mathcal{F}_c- \mathcal{F}_e$ forms a proper convex splitting. 
The proof of Theorem \ref{cor_estimate3} now follows from Theorem \ref{thm_estimate2}
and can be treated as a corollary.
\end{rem}

\begin{rem}
D. Eyre \cite{Eyre} is credited widely for the convex splitting approach which was originally applied to 
the Cahn-Hilliard equations.
%The authors S. Wise et al in \cite{WISE2009} presented the proof of Theorem \ref{thm_grad} for the specific 
%functional form in Eq. (\ref{energydensity_form1}) and Eq. (\ref{energydensity_form2}) in the manner outlined here.
This is the first instance to the best of our knowledge where the energy decay was characterized solely in terms of 
a convexity property without specific functional forms, which enables 
convex splitting arguments to apply to a wider class of free energies,
i.e., the ones that satisfy what we call a proper convex splitting.
In a similar manner one can construct convex splitting schemes for the hydrodynamic models that form the subject of this paper.
This is presented in the following section.
\end{rem}
%It is easy to see that the free energies in Eq. (\ref{energy_DFT}) and Eq. (\ref{energy_PFC}) satisfy conditions in the above theorem and
%satisfy the inequality Eq. (\ref{ineq_CS}).

%%%%%%%%%%%%%%%%%%%%%%%%%%%%%%%%%%%%%%%%%%%%%%%%%%%%%%%%%%%%%%

%%%%%%%%%%%%%%%%%%%%%%%%%%%%%%%%%%%%%%%%%%%%%%%%%%%%%%%%%%%%%%

\section{Semi-discrete Energy Stable discretization of the Hydrodynamic Model} \label{section:method_semi}
An unconditionally energy stable time discretization of the model in Eq. (\ref{hydro_prim}) takes the form (space remains continuous):
\begin{eqnarray}
\rho^{k+1} - \rho^{k} = -s \Div ( \rho^k \bu^{k+\frac{1}{2}} ), \label{semisicrete_1_a} \\
\begin{array}{r} \displaystyle \rho^k ( \bu^{k+1} - \bu^{k} ) = s \left[ -\rho^k \omega^k \times \bu^{k+\frac{1}{2}} - \frac{\rho^k}{2} \Grad \mid \bu^{k+1} \mid^2 - \rho^k \Grad \mu^{k+1} \right. \\
\displaystyle \left. + \gamma \Div \mathcal{D}^{k+\frac{1}{2}} \right],
\end{array} \label{semisicrete_1_b} \\
\mu^{k+1} = \delta_{\rho} \mathcal{F}_c[\rho^{k+1}] - \delta_{\rho} \mathcal{F}_e[\rho^{k}]\label{semisicrete_1_c} , 
\end{eqnarray}
where we have used the relation $\bu \cdot \Grad \bu = \omega \times \bu + \frac{1}{2} \Grad \mid  \bu \mid^2 $ and $\omega = \Grad \times \bu $.
In the above equations the superscripts ($k$ and $k+1$) stand for the discretization of time as in the case of Remark \ref{time_estimate} and 
\begin{equation}
\label{u_half}
\bu^{k+\frac{1}{2}} := \frac{1}{2} (\bu^{k+1} + \bu^k), 
\end{equation}
and 
\begin{equation}
\label{D_half}
\mathcal{D}^{k+\frac{1}{2}} :=  \left(\Grad \bu^{k+\frac{1}{2}} + ({\Grad \bu^{k+\frac{1}{2}}})^T\right) - \Div \bu^{k+\frac{1}{2}} \mathbf{I}.
\end{equation}
It is easy to see that $\Div \mathcal{D}^{k+\frac{1}{2}} = \Grad^2 \bu$.
Note that this discretization respects the proper convex splitting of the free energies, i.e., the 
convex part of the free energy is treated as an implicit term and the concave parts are treated as explicit terms.
Now we have the following theorem.

%%%%%%%%%%%%%%%%%%%Thrm %%%%%%%%%%%%%
\begin{thm}
\label{discrete_time_thm}
Suppose that $\Omega = [0,L_x)\times [0,L_y)$ and $\rho^k : \Omega \to \mathbb{R}^2$ and $\bu^k : \Omega \to \mathbb{R}^2$ for all $k\geq 0, k\in \mathbb{Z}^+$ are
periodic and sufficiently regular. Then
the solution to the scheme Eq. (\ref{semisicrete_1_a}) - (\ref{semisicrete_1_c}) conserves mass, i.e.,
\begin{equation}
\left( \rho^{k+1},1\right)_{2} = \left( \rho^{k},1\right)_{2}, \qquad \forall k\geq 0.
\end{equation}
If in addition $\mathcal{F}$ admits a proper convex splitting  such that
\begin{equation}
\mathcal{F}[\phi] - \mathcal{F}[\psi] \leq ( \delta_{\phi} \mathcal{F}_c[\phi] - \delta_{\psi} \mathcal{F}_e[\psi],\phi - \psi)_2,
\label{ineq1}
\end{equation}
then the scheme is unconditionally energy stable with respect to the total energy $\mathcal{E}[\rho, \bu]$ (defined in Eq. \ref{total_energy}) and
\begin{equation}
\mathcal{E}[\rho^{k+1},\bu^{k+1}] \leq \mathcal{E}[\rho^k,\bu^k] \quad  \forall k\geq 0, 
\end{equation}  
for any $s>0$.
\end{thm}

\begin{proof}
Integrating Eq. (\ref{semisicrete_1_a})  over $\Omega$ we immediately have
mass conservation $\left( \rho^{k+1},1\right)_{2} = \left( \rho^{k},1\right)_{2}$.
Choosing $\phi =\rho^{k+1}$ and $\psi = \rho^{k}$ in Eq. (\ref{ineq1}) we have
the following estimate on the free energy :
\begin{equation}
\mathcal{F}[\rho^{k+1}] - \mathcal{F}[\rho^k] \leq ( \delta_{\rho} \mathcal{F}_c[\rho^{k+1}] - \delta_{\rho} \mathcal{F}_e[\rho^k],\rho^{k+1} - \rho^k)_2.
\label{ineq1_use}
\end{equation}
This estimate immediately translates to an estimate for the total energy (see Eq. (\ref{total_energy})) as follows :
 \begin{equation}
 \begin{array}{rl}
\displaystyle \mathcal{E} [ \rho^{k+1},\bu^{k+1}] - \mathcal{E}[\rho^k,\bu^k] & \displaystyle = \frac{1}{2}(\rho^{k+1} \bu^{k+1}, \bu^{k+1} ) - \frac{1}{2}\left( \rho^{k} \bu^{k}, \bu^{k} \right)  \\
& \displaystyle  \quad + \mathcal{F}[\rho^{k+1}] - \mathcal{F}[\rho^k]  \\
& \displaystyle \leq \frac{1}{2}(\rho^{k+1} \bu^{k+1}, \bu^{k+1} ) - \frac{1}{2}\left( \rho^{k} \bu^{k}, \bu^{k} \right)  \\
& \displaystyle  \quad +  (\rho^{k+1} - \rho^k,  \mu^{k+1} ), \\
\end{array}
\label{total_estimate}
\end{equation}
where we have used Eq. (\ref{ineq1_use}) in in step 2 along with the definition of $\mu^{k+1}$ (Eq. (\ref{semisicrete_1_c})).
Multiplying Eq. (\ref{semisicrete_1_b}) by $\bu^{k+\frac{1}{2}}$ and integrating over $\Omega$ we have
 \begin{equation}
 \label{eq_s0}
 \begin{array}{rl}
\displaystyle \frac{1}{2} \int_{\Omega}\rho^{k+1} \mid\bu^{k+1}\mid^2 d\br  - \frac{1}{2}\int_{\Omega} \rho^{k} \mid \bu^{k}\mid^2 d\br  
& \displaystyle - \frac{1}{2}\int_{\Omega} \left(\rho^{k+1} - \rho^k \right) \mid \bu^{k+1}\mid^2 d\br \\
 = & \displaystyle - s \frac{1}{2} \int_{\Omega} \rho^k \bu^{k+\frac{1}{2}} \cdot \Grad \mid \bu^{k+1} \mid^2 d\br \\
& \displaystyle  - s \int_{\Omega} \rho^k \bu^{k+\frac{1}{2}} \cdot \Grad \mu^{k+1} d \br  \\
& \displaystyle+ s \gamma 
\int_{\Omega} \bu^{k+\frac{1}{2}} \cdot  \left( \Div \mathcal{D}^{k+\frac{1}{2}} \right) d\br . 
\end{array}
\end{equation}
Now we note that
\begin{equation}
\label{eq_s1}
\begin{array}{rl}
\displaystyle - \frac{1}{2}\int_{\Omega} \left(\rho^{k+1} - \rho^k \right) \mid \bu^{k+1}\mid^2 d\br  & \displaystyle =  \frac{1}{2}\int_{\Omega} s \Div (\rho^k \bu^{k+\frac{1}{2}}) \mid \bu^{k+1}\mid^2 d\br,    \\
	 					& \displaystyle =  - s \frac{1}{2}\int_{\Omega} \rho^k \bu^{k+\frac{1}{2}} \cdot \Grad \mid \bu^{k+1}\mid^2 d\br,
\end{array}
\end{equation}
where we have used the continuity equation Eq. (\ref{semisicrete_1_a}) in the first step and integration by parts in the second step.
Further we have
\begin{equation}
\label{eq_s2}
\begin{array}{rl}
\displaystyle - s \int_{\Omega} \rho^k \bu^{k+\frac{1}{2}} \cdot \Grad \mu^{k+1} d \br &\displaystyle =  s \int_{\Omega} \Div \left( \rho^k \bu^{k+\frac{1}{2}} \right) \mu^{k+1} d \br,  \\
& \displaystyle =  - \int_{\Omega} \left(\rho^{k+1} - \rho^k \right) \mu^{k+1} d \br,  
\end{array}
\end{equation}
where we have used integration by parts in the first step and the continuity equation Eq. (\ref{semisicrete_1_a}) in the second step.
Now combining the relations Eq. (\ref{eq_s1}) and Eq. (\ref{eq_s2}) with Eq. (\ref{eq_s0}) and simplifying the expression we get

\begin{equation}
 \begin{array}{rl}
\displaystyle  \frac{1}{2}(\rho^{k+1} \bu^{k+1}, \bu^{k+1} ) - \frac{1}{2}\left( \rho^{k} \bu^{k}, \bu^{k} \right)  \\
\displaystyle \quad+  (\rho^{k+1} - \rho^k,  \mu^{k+1} ) 
& \displaystyle= s \gamma 
\int_{\Omega} \bu^{k+\frac{1}{2}} \cdot \Div \mathcal{D}^{k+\frac{1}{2}} d\br,  \\
& \displaystyle=  - \frac{s \gamma}{2} 
\int_{\Omega} \mathcal{D}^{k+\frac{1}{2}} :\mathcal{D}^{k+\frac{1}{2}} d\br,
\end{array}
\label{parts_eq}
\end{equation}
where we have used the identity
\begin{equation}
\int_{\Omega} \bu^{k+\frac{1}{2}} \cdot \Div \mathcal{D}^{k+\frac{1}{2}} d\br =  - \frac{1}{2} \int_{\Omega} \mathcal{D}^{k+\frac{1}{2}} : \mathcal{D}^{k+\frac{1}{2}} d \br,
\end{equation}
in the second step. This expression is proven in  \ref{Appendix_Tensor}.
Finally using Eq. (\ref{parts_eq}) to simplify the right hand side of Eq. (\ref{total_estimate})  we have
 \begin{equation}
 \begin{array}{rl}
\displaystyle \mathcal{E} [ \rho^{k+1},\bu^{k+1}] - \mathcal{E}[\rho^k,\bu^k] & \displaystyle \leq \frac{1}{2}(\rho^{k+1} \bu^{k+1}, \bu^{k+1} ) - \frac{1}{2}\left( \rho^{k} \bu^{k}, \bu^{k} \right)  \\
& \displaystyle  \quad +  (\rho^{k+1} - \rho^k,  \mu^{k+1} ), \\
& \displaystyle= -\frac{s \gamma}{2} \int_{\Omega} \mathcal{D}^{k+\frac{1}{2}} : \mathcal{D}^{k+\frac{1}{2}} d \br, \\
&\displaystyle \leq 0.
\end{array}
\end{equation}

Hence the total energy of the system is non-increasing, regardless of the time step $s>0$ i.e., $\mathcal{E}[\rho^{k+1},\bu^{k+1}] \leq \mathcal{E}[\rho^k,\bu^k]$. Such schemes are called unconditionally energy stable.
\end{proof}

The above theorem is one of the main results in this paper.
The discrete time continuous space method is energy stable for the 
CDFT and the PFC free energy models presented in the previous section.
It is worth noting that the proof of the above theorem relies entirely on the free energy 
satisfying a proper convex splitting. Thus the theorem holds for a large class of continuum free energy models 
where the free energy admits a proper convex splitting. In what follows we will present the fully discrete version of the method.

%%%%%%%%%%%%%%%%%%%%%%%%%%%%%%%%%%%%%%%%%%%%%%%%%%%%%%%%%

%%%% Numerical Methdod Fully Discrete

%%%%%%%%%%%%%%%%%%%%%%%%%%%%%%%%%%%%%%%%%%%%%%%%%%%%%%%%%

\section{Fully Discrete Numerical method for Hydrodynamic Models} \label{section:method_full}

A fully discrete energy stable discretization of Eq. (\ref{Eqn:Hydro_damp}) is written as :
\begin{eqnarray}
\displaystyle \rho^{k+1}_{ij} - \rho^{k}_{ij} = -s \Divh ( \rho^k_{ij} \bu^{k+\frac{1}{2}}_{ij} )\label{discrete_a},\\
\displaystyle \begin{array}{r}\displaystyle  \rho^k_{ij} ( \bu^{k+1}_{ij} - \bu^{k}_{ij} ) = 
 s \left[ - \rho^k_{ij} \omega_{ij}^k \times \bu^{k+\frac{1}{2}}_{ij} - \frac{\rho^k}{2} \Gradh \mid \bu^{k+1}_{ij} \mid^2 - \rho^k_{ij} \Gradh \mu^{k+1}_{ij} \right. \\ 
\displaystyle  \left. + \gamma \Delta_h \bu^{k+\frac{1}{2}}_{ij} \right],\end{array}
\label{discrete_b}\\
\displaystyle \mu_{ij}^{k+1} = \delta_{\rho} \mathcal{F}_c[\rho_{ij}^{k+1}] - \delta_{\rho} \mathcal{F}_e[\rho_{ij}^{k}]\label{discrete_c},\\
\displaystyle  \omega^k_{ij} = \left( \Gradh \times \bu^k_{ij} \right)\label{discrete_d}, 
\end{eqnarray}
where $\rho \in \mathcal{C}_{m\times n}$ periodic, $\mathcal{F},\mathcal{F}_c, \mathcal{F}_e : \mathcal{C}_{m \times n} \longrightarrow \mathbb{R}$ are fully discrete free energy functionals such that $\mathcal{F}[\rho] = \mathcal{F}_c[\rho] - \mathcal{F}_e[\rho]$ is a proper convex splitting of $\mathcal{F}[\rho]$. 
The discrete function spaces and the corresponding discrete differential operators 
and norms are defined in \ref{appendix_norms}.

%%%%%%%%%%%%%%%%%%%thrm %%%%%%%%%%%%%%%%%%%%%%
%%%%%%%%%%%%%%%%%%%%%%%%%%%%%%%%%%%%%%%%%%%%%%%

%%%%%%%%%    %%%%%%%%%%%%%%%%

%%%%%%%%%%%%%%%%%%%%%%%%%%%%%%%%%%%%%%%%%%%%%%%
\subsection{Unconditional Energy Stability and Conservation}
In order to show that the fully discrete scheme presented in  Eqs. (\ref{discrete_a}) - (\ref{discrete_d})
is unconditionally energy stable and mass conserving
we present the following general theorem which is independent of the form 
of the free energy.
The result however takes advantage of the energy estimate from the proper convex splitting to show unconditional stability.
\begin{thm}
Suppose that $\rho^{k} \in \mathcal{C}_{\bar{m} \times \bar{n}}$ and $\bu^k \in \mathcal{C}_{\bar{m} \times \bar{n}} \times \mathcal{C}_{\bar{m} \times \bar{n}} $ for all $k\geq 0$ are
periodic , then
the solution to the scheme Eq. (\ref{discrete_a}) - (\ref{discrete_d}), $\rho^{k+1}, \bu^{k+1}$ conserves mass, i.e.,
\begin{equation}
\left( \rho^{k+1} \parallel 1\right) = \left( \rho^{k} \parallel 1\right), \qquad \forall k\geq 0.
\end{equation}
If in addition $\mathcal{F}:= \mathcal{F}_c - \mathcal{F}_e$  satisfies the inequality
\begin{equation}
\mathcal{F}[\phi] - \mathcal{F}[\psi] \leq ( \delta_{\phi} \mathcal{F}_c[\phi] - \delta_{\psi} \mathcal{F}_e[\psi] \parallel \phi - \psi),
\label{ineq2c}
\end{equation}
then the scheme is also unconditionally energy stable, with discrete energy 
\begin{equation}
\mathcal{E}[\rho^k, \bu^k] = \frac{1}{2} (\rho^k \bu^k \parallel \bu^k ) + \mathcal{F}[\rho^k].
\end{equation}
Thus we have
\begin{equation}
\mathcal{E}[\rho^{k+1},\bu^{k+1}] \leq \mathcal{E}[\rho^k,\bu^k] \quad  \forall k\geq 0, 
\end{equation}  
for any $s>0$.
\label{thm_energy_stability}
\end{thm}
\begin{proof} The proof of the theorem is similar to Theorem \ref{discrete_time_thm} and is presented in \ref{proof_thm_energy_stability}. Note that $\left( \cdot \parallel \cdot \right)$ is an appropriately defined discrete inner product (see \ref{appendix_norms}).
\end{proof}
 
Now in order to construct a fully discrete convex splitting scheme that is unconditionally energy stable 
and mass conserving all we need is a fully discrete proper convex splitting of the free energy.

%%%%%%%%%%%%%%%%%%%%%%%%%%%%%%%%%%%%%%%%%%%%%%%

%%%%%%%%%    %%%%%%%%%%%%%%%%

%%%%%%%%%%%%%%%%%%%%%%%%%%%%%%%%%%%%%%%%%%%%%%%

\subsection{Proper Convex-Splitting for Fully Discrete CDFT Free Energy}
The fully discrete free energy functional the CDFT model $\mathcal{F}_{CDFT}:\mathcal{C}_{m \times n} \longrightarrow \mathbb{R}$ 
is given by :
\begin{equation}
\mathcal{F}_{CDFT}[\rho] = \left( \rho \parallel \ln( \rho) -1 \right) -  \frac{1}{2} (\rho \parallel  J * \rho). 
\label{CS_DFT_discrete}
\end{equation}
where $J,J_c, J_e \in \mathcal{C}_{m \times n}$ are the restrictions of the convolution kernel given by $J_{i,j} := J( x_i, x_j)$.
In the above equation the periodic discrete convolution is defined as:
\begin{equation}
\label{discrete_convolution}
(J*\rho)_{i,j} := h^2 \sum_{k}^m \sum_l^n J_{k,l} \rho_{i-k.j-l},
\end{equation}
where $h$ is the uniform grid spacing for $\mathcal{C}_{m \times n}$ (see \ref{appendix_norms} for definitions).
We now have the following energy estimate.
\begin{lem} \label{lem:CDFT_ineq}
Suppose $\rho \in \mathcal{C}_{m \times n}$ is periodic then the energies
\begin{equation}
\displaystyle \mathcal{F}_{CDFT,c} [ \rho ] :=  \left( \rho \parallel \ln( \rho) -1 \right) + (J_e * 1)  \parallel \rho \parallel_2^2 ,
\label{CS_DFT_discrete1}
\end{equation} 
and
\begin{equation}
\displaystyle \mathcal{F}_{CDFT,e} [ \rho ] :=   (J_e  * 1)  \parallel \rho \parallel_2^2 +  \frac{1}{2} (\rho \parallel  J * \rho) ,
\label{CS_DFT_discrete2}
\end{equation}
are convex in $\rho$ and  $\mathcal{F}_{CDFT}:=\mathcal{F}_{CDFT,c} -\mathcal{F}_{CDFT,e}  $ is a proper convex splitting of the $\mathcal{F}_{CDFT}$.
The gradients of the energies are
 \begin{equation}
 \delta_{\rho} \mathcal{F}_{CDFT,c}[\rho] = \ln ( \rho ) + 2 (J_e * 1) \rho,
 \end{equation} 
  \begin{equation}
 \delta_{\rho} \mathcal{F}_{CDFT,e}[\rho] =  2 (J_e  * 1) \rho + J * \rho,
 \end{equation}
 and the energy satisfies 
\begin{equation}
\mathcal{F}[\phi] - \mathcal{F}[\psi] \leq ( \delta_{\rho} \mathcal{F}_c[\phi] - \delta_{\rho} \mathcal{F}_e[\psi] \parallel \phi - \psi).
\label{CDFT_ineq}
\end{equation}
\label{CDFT_estimate}
\end{lem} 
\begin{proof}
It is easy to verify that \[  \frac{d^2}{d q^2} F_{CDFT,c}[\rho + q \phi] \geq0  \qquad \text{and} \qquad \frac{d^2}{d q^2} F_{CDFT,e}[\rho + q \phi] \geq0 \] 
for any periodic $\phi \in \mathcal{C}_{m \times n}, q \in \mathbb{R}$.
The inequality in Eq. (\ref{CDFT_ineq}) follows from the proper convex splitting of $\mathcal{F}_{CDFT}$.
\end{proof}

\begin{cor}
Suppose that $\rho^{k} \in \mathcal{C}_{\bar{m} \times \bar{n}}$ and $\bu^k \in \mathcal{C}_{\bar{m} \times \bar{n}} \times \mathcal{C}_{\bar{m} \times \bar{n}} $ for all $k\geq 0$ are periodic and
the discrete free energy is defined by Eqs. (\ref{CS_DFT_discrete})-(\ref{CS_DFT_discrete2}), then the scheme in Eqs. (\ref{discrete_a}) - (\ref{discrete_d}) is mass conserving and unconditionally energy stable for $s>0$.
\end{cor}
\begin{proof}
The proof follows from Theorem \ref{thm_energy_stability} and Lemma \ref{lem:CDFT_ineq}.
\end{proof}

%%%%%%%%%%%%%%%%%%%%%%%%%%%%%%%%%%%%%%%%%%%%%%%

%%%%%%%%%   Proper Convex-Splitting Scheme for Fully Discrete PFC  Free Energy %%%%%%%%%%%%%%%%

%%%%%%%%%%%%%%%%%%%%%%%%%%%%%%%%%%%%%%%%%%%%%%%

\subsection{Proper Convex-Splitting Scheme for Fully Discrete PFC  Free Energy}
The fully discrete free energy functional the PFC model $\mathcal{F}_{PFC}:\mathcal{C}_{m \times n} \longrightarrow \mathbb{R}$ 
is given by :
\begin{equation}
\mathcal{F}_{PFC} [ \rho ] :=\frac{1}{12} \parallel \rho-\frac{3}{2} \parallel_4 + \frac{\alpha}{2} \parallel \rho-\frac{3}{2} \parallel_2 + \frac{1}{2} \parallel \Delta_h \rho \parallel_2 +  \left( \rho \parallel \Delta_h \rho\right),
\label{CS_PFC_discrete}
\end{equation}
where the last term is introduced to approximate $-\parallel \Grad \rho \parallel^2$ in Eq. (\ref{energy})  and the discrete norm $\parallel \cdot \parallel_4$ is defined in
\ref{appendix_norms}.
The corresponding energy estimate is now given by the following lemma.
\begin{lem} \label{lem:PFC_ineq}
Suppose $\rho \in \mathcal{C}_{m \times n}$ is periodic and $\Delta_h \rho$ is also periodic then the energies
\begin{equation}
\displaystyle \mathcal{F}_{PFC,c} [ \rho ] := \frac{1}{12} \parallel \rho-\frac{3}{2} \parallel_4 + \frac{\alpha}{2} \parallel \rho-\frac{3}{2} \parallel_2 + \frac{1}{2} \parallel \Delta_h \rho \parallel_2 ,
\label{CS_PFC_discrete1}
\end{equation} 
and
\begin{equation}
\displaystyle \mathcal{F}_{PFC,e} [ \rho ] :=   - \left(\rho \parallel \Delta_h \rho \right) ,
\label{CS_PFC_discrete2}
\end{equation}
are convex in $\rho$ and  $\mathcal{F}_{PFC}:=\mathcal{F}_{PFC,c} -\mathcal{F}_{PFC,e}  $ is a convex splitting of the $\mathcal{F}_{PFC}$.
The gradients of the energies are
 \begin{equation}
 \delta_{\rho} \mathcal{F}_{PFC,c}[\rho] = \frac{1}{3}\left( \rho-\frac{3}{2} \right)^3 + \alpha \left( \rho-\frac{3}{2} \right) + \Delta_h^2 \rho,
 \end{equation} 
  \begin{equation}
 \delta_{\rho} \mathcal{F}_{PFCT,e}[\rho] = - 2 \Delta_h \rho, 
 \end{equation}
 and the energy satisfies 
\begin{equation}
\mathcal{F}[\phi] - \mathcal{F}[\psi] \leq ( \delta_{\rho} \mathcal{F}_c[\phi] - \delta_{\rho} \mathcal{F}_e[\psi] \parallel \phi - \psi).
\label{PFC_ineq}
\end{equation}
\end{lem} 
\begin{proof}
The proof is similar to that of Lemma \ref{CDFT_estimate} and we omit the details for brevity.
\end{proof}

\begin{cor}
Suppose that $\rho^{k} \in \mathcal{C}_{\bar{m} \times \bar{n}}$ and $\bu^k \in \mathcal{C}_{\bar{m} \times \bar{n}} \times \mathcal{C}_{\bar{m} \times \bar{n}} $ for all $k\geq 0$ are periodic and
the discrete free energy is defined by Eqs. (\ref{CS_PFC_discrete})- (\ref{CS_PFC_discrete2}), then the scheme in Eqs. (\ref{discrete_a}) - (\ref{discrete_d}) is mass conserving and unconditionally energy stable for $s>0$.
\end{cor}
\begin{proof}
The proof follows from Theorem \ref{thm_energy_stability} and Lemma \ref{lem:PFC_ineq}.
\end{proof}

%%%%%%%%%%%%%%%%%%%%%%%%%%%%%%%%%%%%%%%%%%%%%%%%%%%%%%%%%%

%%%%%%%%%%%%% Numerical Results %%%%%%%%%%%%%%%%%%%%%%%%%%%%

%%%%%%%%%%%%%%%%%%%%%%%%%%%%%%%%%%%%%%%%%%%%%%%%%%%%%%%%%%

\section{Numerical Results} \label{section:numerical}
To solve the nonlinear system at the implicit time level, we use a standard non-linear Full Approximation Scheme (FAS) (see \cite{Trottenberg}).
The implementation of the prolongation and restriction operations for the multigrid solver are presented in detail in \cite{Baskaran2010}
in the context of a linear problem. 
The Gauss-Seidel smoothing scheme used in the FAS method is presented in \ref{appendix_MGDFT} and \ref{appendix_MGPFC} respectively for the DFT and PFC models.
Further we note that we use one smoothing iteration per level in all the simulations presented in this paper.
Because the density field takes the form
of sharp Gaussian like peaks on a lattice with near vacuum in the interstitial spaces as observed in \cite{Baskaran2014} in the case of the CDFT models, two modifications to the method are introduced:
\begin{enumerate}
\item The first is to introduce a projection step for the density field.
We use a projection step where we replace $\rho$ by its positive projection $ \sqrt{ \rho^2 + \delta^2}$. 
Here $\delta$ is a regularization parameter chosen to be $10^{-10}$ in the case of the CDFT model.
This allows us to handle the vacuum seamlessly and ensure that the $\log(\rho)$ term does not receive a negative input in the vacuum region.
The parameter $\delta$ is set to 0 in the PFC case.
\item The second modification is to damp the Gauss-Seidel iterations and slowing down the approach 
to the solution (see \ref{DAMP}).
This prevents the method from overshooting the vacuum and producing negative densities.
Although preliminary simulations indicate that damping alone might be sufficient to maintain positivity of the density
the combination of damping and projection allows us to use minimal damping and faster convergence.
\end{enumerate}  
In the reminder of this section we present numerical evidence for convergence and energy-stability of the schemes.
We also present some simulations illustrating the mechanisms of the solid-liquid phase transition in the model and 
the ability of driven flows to steer the system away from equilibrium. 
We use the following model convolution kernel:
\begin{equation}
J({\bf x}) = \sqrt{2} \exp \left( - \frac{\pi^2}{2\nu^2}{\bf x}^2 \right) -  \exp \left( - \frac{\pi^2}{4\nu^2}{\bf x}^2 \right), 
\end{equation}
where $\nu$ is a constant spatial scaling.
The specific form is chosen to model a two particle direct correlation function \cite{Hansen2006} with a single peak in Fourier space.
The parameter $\nu$ is chosen to recover a phase transition qualitatively similar to the one presented in \cite{Baskaran2014}.
%The functions $J_c$ and $J_e$ are chosen to be the first and second terms in the expression above.
The function $J_e$ is chosen to be the function :
\begin{equation}
J_e({\bf x}) = \left\{ \begin{array}{lr} -J({\bf x}) & \text{if} \quad J({\bf x}) \leq 0\\ 0 & \text{otherwise} \end{array} \right. .
\end{equation}
$J_c$ is now defined by the relation $J_c = J + J_e$.
All simulations presented in this work use a residual tolerance of $10^{-12}$ for the PFC model and
$10^{-14}$ for the CDFT model. 
The viscosity coefficient $\gamma$ is set to $\gamma = 2$ in all simulations. 
%%%%%%%%%%%%%%%%%%%%%%%%%%%%%%%%%%%%%%%%%%%%%%%

%%%%%%%%%  Convergence Test  %%%%%%%%%%%%%%%%

%%%%%%%%%%%%%%%%%%%%%%%%%%%%%%%%%%%%%%%%%%%%%%%
\subsection{Convergence in time}
In order to estimate the convergence rate we evolve the system using same initial data with increasingly finer grid resolutions.
The time step refinement path is chosen to be $s = 0.025h^2$. 
This choice is not motivated by a time step restriction as the scheme is unconditionally energy stable.
Since the scheme is first order accurate in time and second order accurate in space the refinement path of
$s = 0.025h^2$ should predict a global error of order $O(h^2)$.
This is verified numerically below.

%%%%%%%%%%%%%%%%%%%%%%%%%%%%%%%%%%%%%%%%%%%%%%%

%%%%%%%%%  Hydrodynamic DFT Model Convergence Test  %%%%%%%%%%%%%%%%

%%%%%%%%%%%%%%%%%%%%%%%%%%%%%%%%%%%%%%%%%%%%%%%

\subsubsection{Hydrodynamic CDFT Model Convergence Test}
\label{Convergence_DFT}
For the case of the hydrodynamic CDFT model we use smooth initial conditions analogous to those previously used for the PFC models in \cite{WISE2009_2, MPFC2013}.
The density field used is 
\begin{equation}
\begin{array}{rl}
\displaystyle \rho(x,y,t=0) =& \displaystyle  \rho_o - 0.02 \cos \left( \frac{2 \pi ( x - 12) }{32} \right) \sin\left( \frac{2 \pi ( y-1 ) }{32} \right) \\
						& \displaystyle + 0.02 \cos^2 \left( \frac{\pi(x+10)}{32} \right) \cos^2 \left( \frac{\pi(y +3)}{32} \right) \\
						 &\displaystyle -0.01 \sin^2\left( \frac{4\pi x}{32} \right) \sin^2 \left( \frac{4\pi(y-6)}{32} \right),
\end{array}
\end{equation}
and $\rho_o = \frac{\pi}{6} \times 0.6 $ which is the density corresponding to packing fraction 0.6 (see \cite{Baskaran2014}).
The initial velocity field is set to $\bu =0$, i.e., a stationary field.
The parameter $\nu =1$ in this test case.

To estimate the convergence rate we start with a grid spacing $h$ and time evolve the system to $t_f=10$ using 
three different grid spacings $h/2,h$ and $2h$.
Given solutions $\rho^h$ on a grid of spacing $h$, we define the Cauchy error between two solutions $\rho^h$ and $\rho^{h/2}$ as
\begin{equation}
e^{h; h/2}_{ij} = \rho^h_{i,j} - \rho^{h/2}_{2i,2j}.
\end{equation}
The convergence rate is then defined as 
\begin{equation}
\log_2\left( \frac{ \parallel e^{2h;h}\parallel_2}{ \parallel e^{h;h/2}\parallel_2} \right).
\end{equation}
The results for  grid sizes $16^2, 32^2,64^2,128^2, 256^2$ and $512^2$ are characterized in Tables  \ref{Table_DFT1}-\ref{Table_DFT3} for
the density field and the two components of the velocity field.
The results suggest that the method is in fact first order accurate in time and second order accurate in space.
The energy evolution shown in Figure \ref{fig_DFT_t1} demonstrates that the total energy of the system is in fact non-increasing.
The figure also shows a striking feature of the model observed in \cite{Baskaran2014} that the model does not preserve monotonicity of the kinetic 
energy or the free energy of the system.

%%%%%%%%%%%%%%%%%%%%%%%%%%%%%%%%%%%%%%%%%%%%%%%

%%%%%%%%%  Hydrodynamic PFC Model Convergence Test  %%%%%%%%%%%%%%%%

%%%%%%%%%%%%%%%%%%%%%%%%%%%%%%%%%%%%%%%%%%%%%%%

\subsubsection{Hydrodynamic PFC Model Convergence Test}
\label{Convergence_PFC}
We use a smooth initial condition analogous to the one previously used for the CDFT convergence test given by
\begin{equation}
\begin{array}{rl}
\displaystyle \rho(x,y,t=0) =&\displaystyle \sqrt{3} \left[   -0.07 - 0.02 \cos \left( \frac{2 \pi ( x - 12) }{32} \right) \sin\left( \frac{2 \pi ( y-1 ) }{32} \right)\right. \\
						 &\displaystyle + 0.02 \cos^2 \left( \frac{\pi(x+10)}{32} \right) \cos^2 \left( \frac{\pi(y +3)}{32} \right) \\
						 & \displaystyle \left. -0.01 \sin^2\left( \frac{4\pi x}{32} \right) \sin^2 \left( \frac{4\pi(y-6)}{32} \right) \right] +\frac{1}{2},
\end{array}
\end{equation}
where $(x,y)^T \in \Omega := [0,32) \times [0,32)$.
This corresponds to the initial conditions used in the standard PFC model convergence test case presented  in \cite{WISE2009_2, MPFC2013}.
With this as the density field and stationary initial velocity field we choose parameter $\epsilon = 0.025$ and time evolve the 
system to $t_f = 10$.
These parameters predict a non-uniform steady state corresponding to a solid phase (demonstrated in Section  \ref{freezing}).
The other parameters are chosen to be identical to the CDFT model presented in the previous section.
The results summarized in Tables \ref{Table_PFC1}-\ref{Table_PFC3} suggest that the method is first order accurate in time and second order accurate in space
for the PFC free energy as well.
The total energy is also observed to be non-increasing (see Figure \ref{fig_PFC_t1}).
It can be seen in Figure \ref{fig_PFC_t1} that the hydrodynamic PFC model also does not preserve
the monotonicity of the free energy or the kinetic energy.

%%%%%%%%%%%%%%%%%%%%%%%%%%%%%%%%%%%%%%%%%%%%%%%

%%%%%%%%%  Simulation of Solid Liquid Phase Transition  %%%%%%%%%%%%%%%%

%%%%%%%%%%%%%%%%%%%%%%%%%%%%%%%%%%%%%%%%%%%%%%%

\subsection{Simulation of Solid Liquid Phase Transition}
\label{freezing}
Next we simulate freezing of a super-cooled liquid using the two models.
In both cases we start the simulation using a random perturbation of a constant density field 
 $\rho = \bar{\rho} ( 1 + 0.1\eta )$, where $\eta$ is a random number in $[0,1]$.
The average density $\bar{\rho}$ is chosen to lie in the solid phase.
The  simulation of the CDFT model is performed with $\bar{\rho} = \frac{\pi}{6} \times 0.6$  corresponding to 
packing fraction 0.6.
The scaling parameter is chosen to be $\nu =2.362$.
 The simulation is performed on a domain  $\Omega = [0,17) \times [0,17)$ with time step $s = 0.02$
 with $128 \times 128$ uniform discretization.
The discrete Fourier image $\hat{J}(k_x,k_y)$ is shown in Figure \ref{fig_conv},
 where 
 \begin{equation}
 \hat{J}(k_x,k_y) = \sum_{\ell=1}^{m} \sum_{j=1}^{n} J(x_\ell,x_j) \exp\left( - i \frac{2\pi (x_\ell k_x + y_j k_y)}{L} \right)
 \end{equation}
 and $(x_\ell,y_j) = (\ell h,jh)$ and $(k_x, k_y)$ belong in a $128 \times 128$ discretization of $[-\pi, \pi) \times [-\pi,\pi)$.
 Observe that there exists wave numbers ${\bf k} = (k_x,k_y)^T$ such that $1- \bar{\rho} \hat{J}({\bf k}) < 0$ so that the 
 homogeneous density $\bar{\rho}$ is unstable, as described in \cite{Baskaran2014}. 
The time evolution of the density and velocity fields are shown in Figure \ref{fig_DFT_t2b}.
The figure  shows how the flow field drives the density field to form the peaks located on
a lattice that represent the solid phase.
The corresponding energy evolution is shown in Figure \ref{fig_DFT_t2a}.
It is seen that the total energy is a non-increasing function.
Figure \ref{fig_DFT_t2a} shows that during the liquid to solid phase transition the
ideal gas part of the free energy increases because the system is acquiring order. 
Concomitantly, the excess part of the free energy decreases.
The total Helmholtz free energy and the total energy of the system are observed to be decreasing functions
of time. Once again the kinetic energy is a non-monotone function of time showing a sharp increase
in kinetic energy at approximately time $t \approx 475$, which corresponds to the sharp decrease in the free energy and total energy
of the system. This corresponds to the fast relaxation and pattern formation during the freezing process (see Figure \ref{fig_DFT_t2b}).
Another sharp decay in the free energy is found at around $t = 25000$.
A closer investigation of the density field at $t =20000$ and $t=28000$ (Figure \ref{fig_DFT_t2b}) 
shows that at $t=20000$ the density field has a defect in the crystal lattice pattern which is relaxed by
the removal of a few peaks (see $t=28000$). The rapid decay in energy near $t=25000$ corresponds
to this process. 
This demonstrates the ability of the numerical method to capture the rich behavior of the hydrodynamic DFT model.

%An analogous simulation for  the PFC model is run with $\epsilon = 0.025$, 
%$\bar{\rho} =  -0.07\sqrt{3} + \frac{1}{2}$
%and $\Omega = [0,128) \times [0,128)$. The simulation is performed on a $128 \times 128$ grid  with $s = 0.1$.
%The evolution of the density field during the freezing transition is shown in Figure \ref{fig_PFC_t2b}.
%The snapshots of the density with the flow field superimposed on it is shown on a zoomed region 
% along side the density field on the full domain in Figure \ref{fig_PFC_t2b}.
%The time evolution of the energy and its components are shown in Figure \ref{fig_PFC_t2a}.
%It is once again observed that the total energy and the free energy are decreasing functions 
%of time while the kinetic energy is non-monotone.
%Consistent with the nonlocal model, the kinetic energy again shows sharp increases in the regions corresponding to the rapid decay
%of the free energy and total energy.

An analogous simulation for  the PFC model is run with $\epsilon = 0.025$, 
$\bar{\rho} =  -0.07\sqrt{3} + \frac{1}{2}$
and $\Omega = [0,32) \times [0,32)$. The simulation is performed on a $128 \times 128$ grid  with $s = 0.01$.
The evolution of the density field during the freezing transition is shown in Figure \ref{fig_PFC_t2b}.
The snapshots of the density with the flow field superimposed on it is shown on a zoomed region 
 along side the density field on the full domain in Figure \ref{fig_PFC_t2b}.
The time evolution of the energy and its components are shown in Figure \ref{fig_PFC_t2a}.
It is once again observed that the total energy and the free energy are decreasing functions 
of time while the kinetic energy is non-monotone.
Consistent with the nonlocal model, the kinetic energy again shows sharp increases in the regions corresponding to the rapid decay
of the free energy and total energy.

%%%%%%%%%%%%%%%%%%%%%%%%%%%%%%%%%%%%%%%%%%%%%%%

%%%%%%%%% Simulation of the Effect of Flow on Phase Transition in PFC Model    %%%%%%%%%%%%%%%%

%%%%%%%%%%%%%%%%%%%%%%%%%%%%%%%%%%%%%%%%%%%%%%%

\subsection{Simulation of the Effect of Flow on Phase Transition in PFC Model}
\label{flow_PFC}
In order to explore a practical application we study the effect of flow on a nanocrystal suspended in
its liquid phase in a driven channel.
The system considered is a channel represented by the domain $\Omega =[0,L) \times [0,L)$,
with $L = 256 h$ where $h = \frac{\pi}{2\sqrt{3}}$. The boundary conditions are periodic in the 
horizontal direction and no slip  moving wall boundary in the vertical direction: $\bu_{i,n} = (u_{wall,n},0)^T$ and $\bu_{i,0} = (u_{wall,0},0)^T$ at the boundary, where $u_{wall,\cdot}$ corresponds to the horizontal velocity of the walls at the top ($n$) and bottom ($0$).
No-flux boundary conditions for the density are imposed at the walls: 
$\rho_{i,n} = \rho_{i,n-1}$ and $\rho_{i,0} = \rho_{i,1}$. 
%at  the top and bottom boundary grid points respectively
%to represent normal derivative of the density to be zero across the wall.
%The no slip velocity is set to a Dirichlet values of Note that the energy stability theorems are valid for this boundary condition as long as $u_{wall,n} = u_{wall,0} = 0$.

The crystal sample for the test is prepared by placing a crystal that is in equilibrium with the surrounding liquid
and annealing it.
The parameter $\epsilon = 0.4$ is chosen where the co-existence region corresponds to $\rho_l = 0.72958$  to
$\rho_s = 0.90073$.
%$\psi_l = -0.444797$ to  $\psi_s = -0.345985 $. 
The solid phase under the one mode approximation takes the form 
\begin{equation}
\rho(x,y) := \rho_s + A \left( \cos \left( \frac{\sqrt{3}}{2} x + \pi \right) \cos \left( \frac{1}{2} y \right) -\frac{1}{2} \cos(y) \right),
\label{solid}
\end{equation} 
where
\begin{equation}
A = \frac{4}{5\sqrt{3}} \left(\rho_s - \frac{3}{2} \right) - \frac{4\sqrt{3}}{15} \sqrt{ 15\epsilon - 12\left(\rho_s -\frac{3}{2}\right)^2}.
\end{equation}
The initial sample is prepared by placing a solid phase given by Eq. (\ref{solid}) in a hexagonal region whose 
diagonal has the dimension $13a$ centered around the point $(L/2 , L/2)$.
This hexagonal solid is surrounded by a liquid phase corresponding to homogeneous density $\rho_l$.
Here $a$ represents the size of an atom $a = 4 \pi / \sqrt{3}$ which is the length of one period in $x$ of the 
solid phase, one mode solution (Eq. (\ref{solid})).
Thus the hexagonal nucleate shown in Figure \ref{fig_nucleate} is 13 atoms wide along the main diagonal.
The solid is expected to be thermodynamic equilibrium with the surrounding liquid.
This initial condition is annealed until time reaches $t=20000$, with wall velocity $u_{wall, 0/n} =0$ (see final density field in Figure \ref{fig_nucleate}), where all components of the energy have equilibrated to a tolerance of $10^{-6}$.
All simulations in this test are performed at a step size of $s= 0.02$.
The equilibrium field is not much different from the initial condition due to the specific choice of $\rho$.
During the annealing, however, the interface between the solid and the liquid becomes diffuse as expected.
Thus the equilibrium configuration consists of a stationary crystal surrounded by a quiescent liquid.

In order to study the effect of flow on the equilibrium crystal we now drive the system by moving the walls 
and setting up a parallel flow in the liquid.
 It is worth noting that the case where the wall moves with a non-zero velocity corresponds to a
driven system. 
In this case the driving force from the wall will allow the system to be driven away from equilibrium. 
The system is no longer energy stable in the sense that the system will no longer follow the gradients of the energy to reach
thermodynamic equilibrium.
The results of the evolution at two different wall speeds ($u_{wall,0/n} = \pm 0.1$ and $u_{wall,0/n} = \pm 0.5$) are shown in 
Figures. \ref{fig_shear1} and \ref{fig_shear2} respectively.
The evolution in the low shear case shows that the crystal does not change shape or size but merely
rotates as the fluid flow is set up.
A close-up of the velocity field that shows the rotation is presented in Figure \ref{fig_shear1_v}.
However in the high shear case the crystal first begins to rotate and then change shape.
As time progresses the crystal nucleate shrinks in size and settles into 
a steady smaller size.
A closer investigation of the long time evolution (not shown) does indicate that 
the crystal stops shrinking beyond a critical size and continues to rotate in the fluid.
This indicates that a system can be driven by flows to settle into 
a different phase co-existence than the one predicted by the stationary equilibrium phase diagram.
This mechanism can be of critical importance in the growth of nanocrystals from a solution or liquid phase.

%{\color{blue} Finally we note that first order accurate in time numerical methods are known to be dissipative. 
%We hope to consider fully second order method in future work. 
%However we note that dissipation manifests as increased viscosity in these systems. 
%Hence  qualitative aspects of the effect of flow observed in this work are still very much valid. }

%%%%%%%%%%%%%%%%%%%%%%%%%%%%%%%%%%%%%%%%%%%%%%%

%%%%%%%%%   Conclusion %%%%%%%%%%%%%%%%

%%%%%%%%%%%%%%%%%%%%%%%%%%%%%%%%%%%%%%%%%%%%%%%

\section{Conclusion}
\label{section:conclusion}
In this paper, we developed energy-stable, fully-discrete numerical methods for hydrodynamic-CDFT models in which compressible flow governed by the isothermal Navier-Stokes equations is driven by a free energy gradient.
An efficient nonlinear multigrid method was proposed and implemented to solve the implicit scheme.
Although we did not demonstrate that the implicit discretization is uniquely solvable for any choice of time and space steps, the multigrid method was always able to solve the system without difficulty for a wide range of temporal and spatial grid sizes.

Numerical simulations of both local (PFC) and nonlocal (CDFT) models were presented that demonstrate that the schemes are first order accurate in time and second order accurate in space.
The energy stability and the ability of the methods to capture solid liquid phase transition in the model was verified numerically.
Simulations illustrating the ability of the methods to capture the effect of flow on freezing transition and non-equilibrium
properties of solid-liquid phase co-existence were presented.
These simulations demonstrated the predictive ability of the model and the robustness of the
numerical method in capturing physically relevant solutions. 

 While first-order in time energy-stable methods, such as that presented here, are known to be significantly dissipative, we view this work as a necessary first step towards the development of second-order accurate (or higher) energy stable methods that would be significantly less dissipative. Nevertheless, the results presented here are accurate. The development of higher order accurate methods, as well as a more thorough investigation of the effect of flow using this model, and a Stokes flow counterpart, will be undertaken in future work.

\paragraph{Acknowledgements}

The authors gratefully acknowledge partial support from NSF Grants NSF-CHE 1035218, NSF-DMR 1105409, and NSF-DMS 1217273.
\let\thefigureSAVED\thefigure
\let\thetableSAVED\thetable
\begin{appendix}

\section{Proof of Energy Dissipation in the General Hydrodynamic Model}
\label{general_energy}
%%%%%%%%%%%%%%%%%%%%%%%%%%%%%%%%%%%%%%%%%%%%%%%

%%%%%%%%%   Proof of energy stability of genenral model %%%%%%%%%%%%%%%%

%%%%%%%%%%%%%%%%%%%%%%%%%%%%%%%%%%%%%%%%%%%%%%%

In this appendix we present a proof of Eq. (\ref{energy_derivative}). Note that by taking the time derivative of
 Eq. (\ref{total_energy}) we have
 \begin{equation}
 \label{energy_d1}
 \frac{d \mathcal{E} }{d t} = \frac{1}{2} \int_{\Omega} \partial_t \rho \mid \bu \mid^2 \quad d \br + \int_{\Omega}  \rho \bu \cdot \partial_t \bu  \quad d \br + \int_\Omega \frac{\delta F}{\delta \rho} \cdot \partial_t \rho \quad d \br.
 \end{equation}
Next we note that the hydrodynamic equations Eq. (\ref{Eqn:Hydro_damp}) can be written in primitive variable form  as
\begin{eqnarray}
\label{hydro_prim2}
\displaystyle\partial _{t}\rho +\Grad \cdot \left(\rho \mathbf{u}\right) =0, \\
\label{hydro_prim3}
\displaystyle \rho \left( \partial _{t} \mathbf{u} + \bu \cdot \Grad \bu \right) =\displaystyle-\rho 
\Grad \left( \frac{\delta \mathcal{F}}{\delta \rho }\right) + \gamma \Grad^2 \bu.
\end{eqnarray}%
Using Eqs. (\ref{hydro_prim2}) and (\ref{hydro_prim3}) to eliminate the time derivative in the first and second integral we have
\begin{equation}
\begin{array}{rl}
\displaystyle \frac{d \mathcal{E} }{d t}  =& \displaystyle   \cancel{ \frac{1}{2} \int _\Omega -\Div (\rho \bu) \mid \bu \mid^2 d \br} +  \int_{\Omega}  \rho \bu \cdot \left( \cancel{- \bu \cdot \Grad \bu} -\Grad \left( \frac{\delta F}{\delta \rho} \right) + \frac{\gamma}{\rho} \Grad^2 \bu \right) d \br   \\
& \displaystyle + \int_\Omega \frac{\delta F}{\delta \rho} \cdot \partial_t \rho  \quad  d \br \\
= & \displaystyle \cancel{ \int_\Omega \Div (\rho \bu) \frac{\delta F}{\delta \rho}  \quad  d \br} + \int_\Omega \gamma \bu \cdot \Grad^2 \bu  \quad d \br+ \cancel{ \int_\Omega \frac{\delta F}{\delta \rho} \cdot \partial_t \rho \quad d \br } \\
= & \displaystyle + \int_\Omega \gamma \bu \cdot \Grad^2 \bu \quad d \br \\
\leq & 0.
\end{array}
\end{equation}
In the above calculation we have used integration by parts in steps 1 and step 2.
Further the cancelation in step 2 is achieved by the use of the continuity equation.

%\appendix
%%%%%%%%%%%%%%%%%%%%%%%%%%%%%%%%%%%%%%%%%%%%%%%

%%%%%%%%%   Derivation of Phase Field Approximation %%%%%%%%%%%%%%%%

%%%%%%%%%%%%%%%%%%%%%%%%%%%%%%%%%%%%%%%%%%%%%%%
\section{Derivation of PFC model}
\label{PFC_derivation}
In order to derive an appropriate phase field crystal approximation we follow the work of 
Van Teeffelen et al \cite{VanTeeffelen2009} and Jaatinen et al \cite{Akusti}.
First consider the CDFT free energy given by

\begin{equation}
\mathcal{F}_{CDFT} [ \rho ] = \mathcal{F}_{id} [ \rho] + \mathcal{F}_{ex} [\rho],
\label{energy_DFT2}
\end{equation}
where
\begin{equation}
\mathcal{F}_{id} [ \rho ] := \int_{\Omega} \rho (\ln ( \rho) - 1) d \br,
\end{equation}
and 
\begin{equation}
\mathcal{F}_{ex} [ \rho ] := -\frac{1}{2} \int_{\Omega} \rho ( J * \rho) d \br.
\end{equation}
The above free energy is approximated in two steps.
\subsection{Approximation of the Excess free energy}
Taking advantage of the fact that the dynamics is conserved i.e $\partial_t \int_{\Omega} \rho d \br =0$ we 
can easily rewrite the  excess free energy as

\begin{equation}
\mathcal{F}_{ex} [ \rho ] := \frac{1}{2} \rho_{ref}^2 \int_{\Omega} (\rho - \rho_{ref} ) \left( J * (\rho - \rho_{ref})\right) d \br,
\end{equation}
where $\rho_{ref}$ is scalar value representing 
a homogeneous reference density field.
Next we take advantage of the radial symmetry of the convolution kernel to expand it as a Taylor series in 
Fourier space about the zero mode:
\begin{equation}
\hat{J}(k) = C_o + C_2 k^2 + C_4 k^4 + \ldots,
\end{equation}
where $k$ is the Fourier variable. 
Now inserting the truncated form of the series into the free energy we have 
\begin{equation}
\label{expression1}
\mathcal{F}_{ex} [ \rho ] := \frac{1}{2} \rho_{ref}^2 \int_{\Omega} \phi ( C_o - C_2 \Grad^2 + C_4 \Grad^4) \phi d {\bf r},
\end{equation}
where $\phi : = \frac{\rho - \rho_{ref}}{\rho_{ref}}$.
Note that we have used the the real space representation of the powers of the Fourier variables in terms of the gradients.
Also note that this expression for the excess free energy is no longer non-local.
The excess free energy now only depends on the local gradients of the density field.

\subsection{Approximation of the ideal gas part of free energy}
In a similar manner we will approximate the ideal gas part of the free energy as follows
\begin{equation}
\begin{array}{rl}
\displaystyle \mathcal{F}_{id} [ \rho ] & \displaystyle  = \rho_{ref} \int_{\Omega} (\phi +1)[ \ln ( \phi +1) + \ln( \rho_{ref} ) -1 ] d \br, \\
 & \displaystyle =  \rho_{ref} \int_{\Omega} \left[ \frac{\phi^2}{2} - \frac{\phi^3}{6} + \frac{\phi^4}{12} + \ldots + \ln(\rho_{ref} )(\phi +1) -1 \right] d\br,
 \end{array}
\end{equation}
where we have Taylor expanded the logarithm term about $\phi = 1$.
Note that if $\rho_{ref}$ is chosen to be close to average density there is reasonable expectation for $\phi$ to be small.
Finally we note that the terms corresponding to $\ln(\rho_{ref} )(\phi +1) -1$ do not contribute to the dynamics
 as the dynamical equations only depend on $\Grad \delta_{\rho} \mathcal{F}[\rho]$.
 This term does not alter the equilibrium as long as the mass conservation is enforced.
 Thus we drop these term and truncate the expansion to obtain the approximate expression
 \begin{equation}
 \label{expression2}
 \mathcal{F}_{id} [ \rho ] \approx  \rho_{ref} \int_{\Omega} \left[ \frac{\phi^2}{2} - \frac{\phi^3}{6} + \frac{\phi^4}{12}\right] d {\bf r}.
 \end{equation}
\subsection{Phase Field Crystal Model}
Combining the two expressions (Eq. \ref{expression1} and \ref{expression2})  for the approximate free energies and re-arranging the terms we obtain the following expression
\begin{equation}
\mathcal{F}[\rho] = \rho_{ref} \left[ \int_{\Omega} \left( - \frac{\phi^3}{6} + \frac{\phi^4}{12} \right) d {\bf r} + \frac{1}{2} \int_{\Omega} \phi [ \zeta + \lambda ( q_o^2 + \Grad^2 )^2] \phi d {\bf r}  \right],
\end{equation}
where the free constants $C_o,C_2 $ and $C_4$ have been replaced without loss of generality with new constants $\lambda, q_o$ and $\zeta$.
The above expressions can further be simplified by noting that the dynamical equations are conservative (i.e.,  $\partial_t \int_{\Omega} \phi d {\bf x} =0$)
and that the dynamics only depends on the first variational derivative of the free energy.
Thus  the terms which are multiples of $\int_{\Omega} \phi^2 d {\bf r} $ do not contribute
to the dynamics and can be added and subtracted freely to complete powers.
Now we introduce the change of variables ${\bf r'} = q_o {\bf r} $, $\psi = \sqrt{ \frac{1}{3\lambda q_o^4} }\left( \phi - \frac{1}{2} \right)$ and $ \epsilon = \frac{1}{\lambda q_o^4} \left( \frac{1}{4} - \zeta \right)$
and dropping terms that are linear in $\phi$ without loss of generality we have
\begin{equation}
\tilde{\mathcal{F}}[\psi]  = (3 \lambda^2 q_o^5) \int_{\Omega}  \frac{\psi}{2} \left[ - \epsilon + ( 1 + \Grad^2 )^2 ] \psi + \frac{1}{2} \psi^4 \right] d {\bf r'}.
\end{equation}
Note that the above expression is the standard PFC free energy \cite{Elder2004} or the standard Swift-Hohenberg free energy \cite{PFC_book}.
Through appropriate non-dimensionalization we can choose the constants  $q_o = \lambda = \rho_{ref} =1$ (\cite{Akusti}).
For this choice of $q_o, \lambda$ and $\rho_{ref}$ the phase and density fields are related by:
\begin{equation}
\psi = \frac{1}{\sqrt{3}} \left( \rho - \frac{3}{2} \right).  
\end{equation} 
and the approximate free energy can be written as
\begin{equation}
\begin{array}{rl}
\displaystyle \mathcal{F}_{PFC}[\psi] & \displaystyle  = 3 \int_{\Omega}\frac{\psi}{2} \left[ - \epsilon + ( 1 + \Grad^2 )^2 \right] \psi + \frac{1}{2} \psi^4 d \br, \\
&\displaystyle = \int_{\Omega}   \left \{ \frac{1}{12} \left( \rho-\frac{3}{2} \right)^4 + \frac{\alpha}{2} \left( \rho-\frac{3}{2} \right)^2 - |\nabla \rho |^2 + \frac{1}{2} (\Delta \rho)^2 \right \} d \br .
\end{array}
\end{equation}
where $\alpha: = 1-\epsilon$.
This is the expression used this work.
Also in the above equation the $'$ is dropped from the scaled ${\bf r}'$.

%%%%%%%%%%%%%%%%%%%%%%%%%%%%%%%%%%%%%%%%%%%%%%%

%%%%%%%%%   Proof of Convex Splitting Estimate %%%%%%%%%%%%%%%%%%%%

%%%%%%%%%%%%%%%%%%%%%%%%%%%%%%%%%%%%%%%%%%%%%%%
\section{Proof of Convex Splitting Estimate}
\label{Proof_Estimate}
Before we prove the Theorem \ref{thm_estimate2} we need to obtain the following estimate.
	\begin{thm}
	\label{thm_estimate1}
	Consider a free energy functional $\mathcal{F}[\rho]: H \subset L_{per}^2(\Omega) \to \mathbb{R}$, where $H \subset L_{per}^2(\Omega)$ is a Hilbert space of sufficiently regular periodic functions 
	$\rho: \Omega \to \mathbb{R}$ with $\Omega = [0,L_x) \times [0,L_y)$. 
	Then 
	\begin{equation}
	\mathcal{F}[\phi] - \mathcal{F}[\psi] \geq ( \delta_{\psi} \mathcal{F}[\psi], \phi - \psi)_2  \qquad \forall \phi, \psi \in H, \phi \neq \psi        
	\label{condition1a}
	\end{equation}
	if 
	\begin{equation}
	\frac{d^2}{d \epsilon^2} \mathcal{F} [ \rho +\epsilon v] \geq 0 \qquad \forall \rho, v \in H, \epsilon \in \mathbb{R}, 
	\label{condition2a}
	\end{equation}
	where $(\cdot,\cdot)_2$ is the usual $L_2$ inner product and $\delta_{\rho}$ represents the variational derivative with respect to $\rho$.
	\end{thm}
	\begin{proof}
	First we define a function $G_{v,\rho}(\epsilon) : \mathbb{R} \to \mathbb{R}$ for  a fixed $\rho,v \in H$ as $G_{v,\rho}(\epsilon) := F[\rho + \epsilon v]$.
	Now let us consider the case where Eq. (\ref{condition2a}) is true, i.e, $G_{\psi,v}(\epsilon)$ is a convex function of $\epsilon$ for all $\psi,v \in H$.
	From the convexity of $G_{\psi,v}(\epsilon)$ we have 
	\begin{equation}
	\label{ineq_C1}
	G_{v,\psi}(\delta) - G_{v,\psi}(0) = F[\psi + \delta v] - F[\psi] \geq \delta\lim_{\epsilon \to 0} \frac{d}{d\epsilon} \mathcal{F} [ \psi + \epsilon v]  \qquad  \forall \psi, v \in H, \delta \in \mathbb{R}.  
	\end{equation}
	This gives us the inequality 
	\begin{equation}
	F[\psi+ \delta v] -F[ \psi]  \geq ( \delta_{\psi} \mathcal{F}, \psi + \delta v - \psi )_2  \qquad \forall \psi, v \in H, 
	\label{rel1}
	\end{equation}
	where the defining relation for the variational derivative (in the Frech\'et sense) 
	\begin{equation}
	\lim_{\epsilon \to 0} \frac{d}{d \epsilon} F[\psi + \epsilon v] = \left( \delta_{\psi} F[\psi], v \right)_2,
	\end{equation}
	is used.
	Now identifying $\psi + \delta v = \phi$  in Eq. (\ref{rel1}) we have the relation in Eq. (\ref{condition1a}),
	this completes the proof of the theorem. 
           \end{proof}
          It is worth noting at this point that Eq. (\ref{ineq_C1}) must hold for all $\delta \in \mathbb{R}$. This requires 
           Eq. (\ref{condition2}) to hold and not just  the limit as given in Eq. (\ref{condition22a}).           
           Since we prove Theorem \ref{thm_estimate2} using Theorem \ref{thm_estimate1}, 
           Theorem \ref{thm_estimate2} requires a proper convex splitting (defined in Definition. \ref{PCS}) rather than a simple convex splitting. \\
                      
         \noindent
        Now we are ready to prove Theorem \ref{thm_estimate2}.
	\begin{proof}(Theorem \ref{thm_estimate2}) \\
	Using Theorem \ref{thm_estimate1} we have 
	\begin{equation}
	\mathcal{F}_c[\psi] - \mathcal{F}_c[\phi] \geq ( \delta_{\phi} \mathcal{F}_c[\phi],\psi - \phi)_2          
	\label{ineq_a}
	\end{equation}
	and 
	\begin{equation}
	\mathcal{F}_e[\phi] - \mathcal{F}_e[\psi] \geq ( \delta_{\psi} \mathcal{F}_e[\psi],\phi - \psi)_2.          
	\label{ineq_b}
	\end{equation}
	Adding the two inequalities above we obtain the result.
	\end{proof}

%%%%%%%%%%%%%%%%%%%%%%%%%%%%%%%%%%%%%%%%%%%%%%%

%%%%%%%%%   Calculation of dissipation %%%%%%%%%%%%%%%%%%%%

%%%%%%%%%%%%%%%%%%%%%%%%%%%%%%%%%%%%%%%%%%%%%%%

\section{Simplification of Energy Dissipation rate} \label{Appendix_Tensor}
In this appendix we prove the identity
\begin{equation}
\int_{\Omega} \bu^{k+\frac{1}{2}} \cdot \Div \mathcal{D}^{k+\frac{1}{2}} d\br = - \frac{1}{2} \int_{\Omega} \mathcal{D}^{k+\frac{1}{2}} : \mathcal{D}^{k+\frac{1}{2}} d \br.
% - \frac{1}{2} \int_{\Omega} (\Div  \bu^{k+\frac{1}{2}} {\bf I}) : (\Div  \bu^{k+\frac{1}{2}} {\bf I}) d \br.
\end{equation}
Integrating by parts and using the symmetry of $\mathcal{D}$ we have 
 \begin{equation}
 \begin{array}{rl}
\displaystyle \int_{\Omega} \bu^{k+\frac{1}{2}} \cdot \Div \mathcal{D}^{k+\frac{1}{2}} d\br,  
& \displaystyle=   
- \frac{1}{2} \int_{\Omega}  \left( \Grad \bu^{k+\frac{1}{2}}  + ( \Grad \bu^{k+\frac{1}{2}})^T \right) : \mathcal{D}^{k+\frac{1}{2}} d\br \\
& \displaystyle =  -\frac{1}{2} \int_{\Omega}  \left( \Grad \bu^{k+\frac{1}{2}}  + ( \Grad \bu^{k+\frac{1}{2}})^T \right) : \left( \Grad \bu^{k+\frac{1}{2}}  + ( \Grad \bu^{k+\frac{1}{2}})^T \right) d\br \\
&\displaystyle   \quad + \frac{1}{2} \int_{\Omega}  \left( \Grad \bu^{k+\frac{1}{2}}  + ( \Grad \bu^{k+\frac{1}{2}})^T \right) : \left(\Div \bu^{k+\frac{1}{2}} \mathbf{I} \right)d \br
\end{array}
\label{eq_appendix_cal}
\end{equation}
where we have used 
the definition of $\mathcal{D}$ in Eq. (\ref{D_half}) to obtain the second equality. 
It is easy to see that
\begin{equation}
 \int_{\Omega}  \left( \Grad \bu^{k+\frac{1}{2}}  + ( \Grad \bu^{k+\frac{1}{2}})^T \right) : \left(\Div \bu^{k+\frac{1}{2}} \mathbf{I} \right)d \br =  \int_{\Omega} \left(\Div \bu^{k+\frac{1}{2}} \mathbf{I} \right): \left(\Div \bu^{k+\frac{1}{2}} \mathbf{I} \right)d \br.
 \label{eq_appendix_cal2}
 \end{equation}
 Now adding and subtracting $\displaystyle \frac{1}{2} \int_{\Omega}  \left( \Grad \bu^{k+\frac{1}{2}}  + ( \Grad \bu^{k+\frac{1}{2}})^T \right) : \left(\Div \bu^{k+\frac{1}{2}} \mathbf{I} \right)d \br$ 
 to the right hand side of Eq. (\ref{eq_appendix_cal})  and using the expression in Eq. (\ref{eq_appendix_cal2}) we have
  \begin{equation}
 \begin{array}{rl}
\displaystyle \int_{\Omega} \bu^{k+\frac{1}{2}} \cdot \Div \mathcal{D}^{k+\frac{1}{2}} d\br,  
& \displaystyle =  -\frac{1}{2} \int_{\Omega}  \left( \Grad \bu^{k+\frac{1}{2}}  + ( \Grad \bu^{k+\frac{1}{2}})^T \right) : \left( \Grad \bu^{k+\frac{1}{2}}  + ( \Grad \bu^{k+\frac{1}{2}})^T \right) d\br \\
&\displaystyle   \quad + \int_{\Omega}  \left( \Grad \bu^{k+\frac{1}{2}}  + ( \Grad \bu^{k+\frac{1}{2}})^T \right) : \left(\Div \bu^{k+\frac{1}{2}} \mathbf{I} \right)d \br \\
&\displaystyle   \quad - \frac{1}{2}\int_{\Omega}   \left(\Div \bu^{k+\frac{1}{2}} \mathbf{I} \right): \left(\Div \bu^{k+\frac{1}{2}} \mathbf{I} \right)d \br, \\
& \displaystyle =  -\frac{1}{2} \left[ \int_{\Omega}  \left( \Grad \bu^{k+\frac{1}{2}}  + ( \Grad \bu^{k+\frac{1}{2}})^T \right) : \left( \Grad \bu^{k+\frac{1}{2}}  + ( \Grad \bu^{k+\frac{1}{2}})^T \right) d\br \right. \\
&\displaystyle   \qquad - 2 \int_{\Omega}  \left( \Grad \bu^{k+\frac{1}{2}}  + ( \Grad \bu^{k+\frac{1}{2}})^T \right) : \left(\Div \bu^{k+\frac{1}{2}} \mathbf{I} \right)d \br \\
&\displaystyle   \qquad +\left. \int_{\Omega}   \left(\Div \bu^{k+\frac{1}{2}} \mathbf{I} \right): \left(\Div \bu^{k+\frac{1}{2}} \mathbf{I} \right)d \br \right] \\
%&\displaystyle   \quad - \frac{1}{2} \int_{\Omega}   \left(\Div \bu^{k+\frac{1}{2}} \mathbf{I} \right): \left(\Div \bu^{k+\frac{1}{2}} \mathbf{I} \right)d \br, \\
&\displaystyle =  -\frac{1}{2} \int_{\Omega} \mathcal{D}^{k+\frac{1}{2}} : \mathcal{D}^{k+\frac{1}{2}} d \br  \\
%&\displaystyle \quad - \frac{1}{2} \int_{\Omega} (\Div  \bu^{k+\frac{1}{2}} {\bf I}) : (\Div  \bu^{k+\frac{1}{2}} {\bf I}) d \br.
\end{array}
\label{eq_appendinx_cal3}
\end{equation}
This completes the proof.

%%%%%%%%%%%%%%%%%%%%%%%%%%%%%%%%%%%%%%%%%%%%%%%

\section{Discrete Function Spaces, Operators  and Inner Products}
 \label{appendix_norms} 
 %%%%%%%%%%%%%%%%%%%%%%%%%%%%%%%%%%%%%%%%%%%%%%
% Discrete Functions Spaces, Inner Products and Operators %

\subsection{Discrete Functions Spaces}
In this section we define the discretization of the domain $\Omega := (0,L_x) \times (0,L_y)$. 
We use a nodal discretization.
The development of the discretization is the similar to  the one used in (\cite{MPFC2011,WISE2009})
where a staggered grid representation was used.

Let $h>0$ be the grid spacing such that $L_x = m\cdot h$ and $L_y = n \cdot h$ where $m,n$ are positive integers.
Now we define the sets
\begin{equation}
\begin{array}{l}
C_m = \{  i \cdot h \mid i = 1 , \ldots , m \} \\
C_{\bar{m}} = \{ i \cdot h \mid i = 0, \ldots , m+1 \} \\
E_{m} = \{ i \cdot h \mid i = \frac{1}{2}, \ldots, m + \frac{1}{2} \}
\end{array}
\end{equation}
The set $C_m\times C_n$ partitions $\Omega$ into a uniform rectangular grid with cells of size $h\times h$. The set
$\{ C_{\bar{m}} \times C_{\bar{n}} \} \backslash \{ C_m \times C_n \}$ 
contains the ghost points outside the boundary of $\Omega $, which are mapped back into $\Omega$ in the case
of periodic boundary conditions.

Now we are ready to define the functions spaces as follows
	\begin{eqnarray}
{\mathcal C}_{m\times n} &=& \left\{\phi: C_m\times C_n \rightarrow \mathbb{R} \right\},\  {\mathcal C}_{\overline{m}\times\overline{n}} = \left\{\phi: C_{\overline{m}}\times C_{\overline{n}}\rightarrow \mathbb{R} \right\}  ,
    \\
{\mathcal C}_{\overline{m}\times n} &=& \left\{\phi: C_{\overline{m}}\times C_n \rightarrow \mathbb{R} \right\},\ {\mathcal C}_{m\times\overline{n}} = \left\{\phi: C_m\times C_{\overline{n}} \rightarrow \mathbb{R} \right\}  ,
	\\
{\mathcal E}^{\rm ew}_{m\times n} &=& \left\{u: E_m\times C_n \rightarrow\mathbb{R} \right\},\ {\mathcal E}^{\rm ns}_{m\times n} = \left\{v: C_m\times E_n \rightarrow\mathbb{R}  \right\}	 ,
	\\
{\mathcal E}^{\rm ew}_{m\times \overline{n}} &=& \left\{u: E_m\times C_{\overline{n}} \rightarrow\mathbb{R} \right\},\ {\mathcal E}^{\rm ns}_{\overline{m}\times n} = \left\{v: C_{\overline{m}}\times E_n \rightarrow\mathbb{R}  \right\}	 .
  \end{eqnarray}

The spaces $\mathcal{C}_{m\times n},\mathcal{C}_{\bar{m}\times \bar{n}},\mathcal{C}_{m\times \bar{n} }$ and 
$\mathcal{C}_{\bar{m}\times n}$ contain the grid functions.
A grid function are identified as $\phi_{ij}:=\phi(x_i,y_j)$, where $x_i=i \cdot h$ and $y_j = j \cdot h$ 
and $i,j$ are integers. The east-west edge-centered functions, in spaces ${\mathcal E}^{\rm ew}_{m\times \overline{n}}$ and
${\mathcal E}^{\rm ew}_{m\times n}$ are identified as $u_{i+\frac{1}{2},j} := u ( x_{i+\frac{1}{2}}, y_j)$.
In a similar manner the north-south edge centered functions, in spaces $\ {\mathcal E}^{\rm ns}_{m\times n}$ and
$\ {\mathcal E}^{\rm ns}_{\overline{m}\times n}$ are identified as $v_{i,j+\frac{1}{2}} := u ( x_{i},y_{j+\frac{1}{2}})$

By defining a set of difference operators and equipping the space with a set of discrete norms 
one can write down self consistent, discrete gradient and Laplacian operators.

We start by defining the relevant difference and average operators on the space.
\begin{enumerate}
\item The averaging and difference operators $A_x,D_x : \mathcal{C}_{\bar{m} \times n} \longrightarrow {\mathcal E}^{\rm ew}_{m\times n}$
and $A_y,D_y : \mathcal{C}_{m \times \bar{n}} \longrightarrow \ {\mathcal E}^{\rm ns}_{m\times n}$ are defined as
\begin{equation}
A_x f_{i+\frac{1}{2},j} = \frac{1}{2}( f_{i+1,j} + f_{i,j} ), D_x f_{i+\frac{1}{2},j} = \frac{1}{h} ( f_{i+1,j} - f_{i,j} ) ,
\begin{array}{l}
i = 0,\ldots,m \\
j=1,\ldots,n 
\end{array}
\end{equation}
\begin{equation}
A_y f_{i,j+\frac{1}{2}} = \frac{1}{2} ( f_{i,j+1} + f_{i,j} ), D_y f_{i,j+\frac{1}{2}} = \frac{1}{h} ( f_{i,j+1} - f_{i,j} ),
 \begin{array}{l}
i = 1,\ldots,m \\
j=0,\ldots,n 
\end{array}
\end{equation}
The edge to center difference operators are defined as $d_x : {\mathcal E}^{\rm ew}_{m\times n} \longrightarrow \mathcal{C}_{m \times n}$
and $d_y : \ {\mathcal E}^{\rm ns}_{m\times n} \longrightarrow \mathcal{C}_{m \times n}$  are defined as
\begin{equation}
d_x f_{i,j} = \frac{1}{h}( f_{i+\frac{1}{2},j} + f_{i-\frac{1}{2},j} ), d_y f_{i,j} = \frac{1}{h} ( f_{i,j+\frac{1}{2}} - f_{i,j-\frac{1}{2}} ) ,
\begin{array}{l}
i = 1,\ldots,m \\
j=1,\ldots,n 
\end{array}
\end{equation}

\item The discrete gradient operator $\nabla_h : \mathcal{C}_{\bar{m} \times n} \times \mathcal{C}_{m \times \bar{n}} \longrightarrow \mathcal{C}_{m \times n} \times \mathcal{C}_{m \times n}$ is given by
\begin{equation}\begin{array}{r}
\displaystyle \nabla_h \phi_{ij} = \frac{1}{2} \left( D_x \phi_{i+\frac{1}{2},j} + D_x \phi_{i-\frac{1}{2},j}, D_y \phi_{i,j+\frac{1}{2}} + D_y \phi_{i,j-\frac{1}{2}} \right)^T, \\
\begin{array}{l}
i = 1,\ldots,m \\
j=1,\ldots,n 
\end{array} 
\end{array}
\end{equation}
\item The discrete divergence operator is  defined in a similar manner as 
\begin{equation}
\begin{array}{r}
\displaystyle \nabla_h \cdot \bu_{ij} = 
 \frac{1}{2} \left( D_x u_{i+\frac{1}{2},j} + D_x u_{i-\frac{1}{2},j} + D_y v_{i,j+\frac{1}{2}} + D_y v_{i,j-\frac{1}{2}} \right), \\
\begin{array}{l}
i = 1,\ldots,m \\
j=1,\ldots,n 
\end{array} 
\end{array}
\end{equation}
where we have used the definition $\bu_{ij} := (u_{ij}, v_{ij})^T$.
\item The discrete 2D Laplacian $ \Delta_h : \mathcal{C}_{\bar{m} \times \bar{n}} \longrightarrow \mathcal{C}_{m\times n}$ is defined as
\begin{equation}
\begin{array}{lll} 
\displaystyle \Delta_h \phi_{ij} & \displaystyle = d_x(D_x \phi)_{ij}  + d_y(D_y \phi)_{ij} \\
 & \displaystyle=  \frac{1}{h^2} ( \phi_{i+1,j} +\phi_{i,j+1}+ \phi_{i-1,j}+ \phi_{i-1,j} - 4 \phi_{i,j} ) &
\begin{array}{l}
i = 1,\ldots,m \\
j=1,\ldots,n 
\end{array} 
\end{array}
\end{equation}
\end{enumerate}
We define the following weighted inner products on the function spaces
\begin{equation}
(\phi \parallel \psi )  = h^2 \sum\limits_{i=1}^{m} \sum\limits_{j=1}^{n} \phi_{i,j}\psi_{i,j} , \quad
 \phi,\psi \in \mathcal{C}_{m\times n} \cup \mathcal{C}_{\bar{m}\times \bar{n}} \cup \mathcal{C}_{m\times \bar{n} }  \cup \mathcal{C}_{\bar{m}\times n} 
 \end{equation}
with $\bu = ( u_1,u_2)^T$ and $\bv = (v_1,v_2)^T$, the inner products 
 \begin{equation}
(\bu \parallel \bv )  = (u_1 \parallel v_1 ) + (u_2 \parallel v_2 )  , \quad
 u_1,u_2,v_1,v_2 \in \mathcal{C}_{m\times n} \cup \mathcal{C}_{\bar{m}\times \bar{n}} \cup \mathcal{C}_{m\times \bar{n} }  \cup \mathcal{C}_{\bar{m}\times n} 
 \end{equation}
 \begin{equation}
 \eipew{f}{g} = h^2 \sum\limits_{i=1}^{m} \sum\limits_{j=1}^{n} \left( f_{i+\frac{1}{2},j}  g_{i+\frac{1}{2},j} + f_{i-\frac{1}{2},j}  g_{i-\frac{1}{2},j} \right), \quad  f,g \in  \mathcal{E}_{m\times n}^{\rm ns}
 \end{equation}
 \begin{equation}
 \eipns{f}{g} = h^2 \sum\limits_{i=1}^{m} \sum\limits_{j=1}^{n} \left( f_{i,j+\frac{1}{2}}  g_{i,j+\frac{1}{2}} + f_{i,j-\frac{1}{2}}  g_{i,j-\frac{1}{2}} \right), \quad f,g \in \mathcal{E}_{m\times n}^{\rm ew}
 \end{equation}
and the one dimensional inner products are defined as
\begin{equation}
\begin{array}{l}
(f_{\star,j+\hf} \mid g_{\star,j+\hf})= h \sum\limits_{i=1}^n f_{i,j+\hf} g_{i,j+\hf} , \quad f,g \in  \mathcal{E}_{m\times n}^{\rm ns} \\
(f_{i+\hf,\star} \mid g_{i+\hf,\star})= h \sum\limits_{i=1}^m f_{i+\hf,j} g_{i+\hf,j} ,\quad f,g \in \mathcal{E}_{m\times n}^{\rm ew}\\
(\phi_{\star,j} \mid \psi_{\star,j})= h \sum\limits_{i=1}^n \phi_{i,j} \psi_{i,j} , \quad \phi,\psi \in  \mathcal{C}_{m\times n} \\
(\phi_{i,\star} \mid \psi_{i,\star})= h \sum\limits_{i=1}^m \phi_{i+\hf,j} \psi_{i+\hf,j} ,\quad \psi,\psi \in \mathcal{C}_{m\times n}
\end{array}
\end{equation}

%Using these definitions we are now ready to obtain the following useful results.

Using the definitions given in this Appendix and in~\cite{WISE2009}, we obtain the following summation-by-parts formulae:
	\begin{prop}
	\label{sbp-2D-edge}
{\em (summation-by-parts)} If $\phi\in{\mathcal C}_{\overline{m}\times n} \cup{\mathcal C}_{\overline{m}\times\overline{n}}$ and $f\in{\mathcal E}_{m\times n}^{\rm ew}$ then
	\begin{eqnarray}
   \eipew{D_x \phi}{f} &=&  -\ciptwo{\phi}{d_x f}
	\nonumber
	\\
&& - \cip{A_x\phi_{\hf,\star}}{f_{\hf,\star}}+ \cip{A_x\phi_{m+\hf,\star}}{f_{m+\hf,\star}} \ ,
	\label{sbp-c-ew}
	\end{eqnarray}
and if $\phi\in{\mathcal C}_{m\times\overline{n}} \cup {\mathcal C}_{\overline{m}\times\overline{n}}$ and $f\in{\mathcal E}_{m\times n}^{\rm ns}$ then
	\begin{eqnarray}
  \eipns{D_y\phi}{f} &=&  - \ciptwo{\phi}{d_y f} 
	\nonumber
	\\
&&  -\cip{A_y\phi_{\star,\hf}}{f_{\star,\hf}}+ \cip{A_y\phi_{\star,n+\hf}}{f_{\star,n+\hf}} \ .
	\label{sbp-c-ns}
	\end{eqnarray}
	\end{prop}
	
	\begin{prop}
	\label{divergence_thm}
{\em (divergence theorem)}Let $\phi\in {\mathcal C}_{\overline{m}\times\overline{n}}$ and $\bu =(u,v)^T \in {\mathcal C}_{\overline{m}\times\overline{n}} \times {\mathcal C}_{\overline{m}\times\overline{n}}$    .  Then
	\begin{eqnarray}
	  (\nabla_h \phi  \parallel \bu ) &=& 
  -( \phi \parallel \Divh \bu )
	\nonumber
	\\
 &&- \hf( \phi_{0,\star} \mid u_{1, \star } ) + \hf( \phi_{m,\star} \mid u_{m+1, \star } ) - \hf( \phi_{1,\star} \mid u_{0, \star } ) +  \hf( \phi_{m+1,\star} \mid u_{m, \star } )    \nonumber
        \\
 &&- \hf( \phi_{\star,0} \mid v_{ \star,1 } ) + \hf( \phi_{\star,n} \mid v_{ \star,n+1 } ) - \hf( \phi_{\star,1} \mid v_{ \star,0 } ) +  \hf( \phi_{\star,n+1} \mid v_{ \star,n } )   . \nonumber
	\\
	\end{eqnarray}

	\end{prop}

	\begin{prop}
	\label{green1stthm-2d}
\emph{(discrete Green's first identity)} Let $\phi,\, \psi\in {\mathcal C}_{\overline{m}\times\overline{n}}$.  Then
	\begin{eqnarray}
	 %(\nabla_h \phi  \parallel \nabla_h \psi )  =
   \eipew{D_x\phi}{D_x\psi} + \eipns{D_y\phi}{D_y\psi}
	%\nonumber
	%\\
&=& - \ciptwo{\phi}{\Dh\psi}
	\nonumber
	\\
& & - \, \cip{A_x\phi_{\hf,\star}}{D_x\psi_{  \hf,\star}} +  \cip{A_x\phi_{m+\hf,\star}}{D_x \psi_{m+\hf,\star}}
	\nonumber
	\\
& & -  \cip{A_y\phi_{\star,  \hf}}{D_y\psi_{\star,  \hf}} +  \cip{A_y \phi_{\star,n+\hf}}{D_y\psi_{\star,n+\hf}} \ . 
	\nonumber
	\\
& &
	\end{eqnarray}
	\end{prop}

	\begin{prop}\label{green2ndthm-2d}
	\emph{(discrete Green's second identity)} Let $\phi,\, \psi\in {\mathcal C}_{\overline{m}\times\overline{n}}$.  Then
	\begin{eqnarray}
\, \ciptwo{\phi}{\Dh\psi} &=&  \ciptwo{\Dh\phi}{\psi}
	\nonumber
	\\
& & + \, \cip{A_x \phi_{m+\hf,\star}}{D_x \psi_{m+\hf,\star}} - \, \cip{D_x \phi_{m+\hf,\star}}{A_x \psi_{m+\hf,\star}}
	\nonumber
	\\
& & -  \cip{A_x \phi_{  \hf,\star}}{D_x \psi_{  \hf,\star}} +  \cip{D_x \phi_{  \hf,\star}}{A_x \psi_{  \hf,\star}}
	\nonumber
	\\
& & +  \cip{A_y \phi_{\star,n+\hf}}{D_y \psi_{\star,n+\hf}} -  \cip{D_y \phi_{\star,n+\hf}}{A_y \psi_{\star,n+\hf}}
	\nonumber
	\\
	& & -  \cip{A_y \phi_{\star,  \hf}}{D_y \psi_{\star,  \hf}} +  \cip{D_y \phi_{\star,  \hf}}{A_y \psi_{\star,  \hf}} \ .
	\end{eqnarray} 
	\end{prop}

%%%%%%%%%%%%%%%%%%%%%%%%%%%%%%%%%%%%%%%%%%%%%%
% Periodic Boundary Conditions %

\subsection{Boundary Conditions} \label{BC}
In this paper we work with periodic functions.
We will consider functions $\phi \in \mathcal{C}_{\bar{n} \times \bar{m}}$ which have the form
\begin{equation}
\begin{array}{lll}
\phi_{m+1,j} = \phi_{1,j}, & \phi_{0,j} = \phi_{m,j}, & j = 1,\ldots,n , \\
\phi_{i,n+1} = \phi_{i,1}, & \phi_{i,0} =\phi_{i,n}, & i = 0,\ldots,m. 
\end{array}
\end{equation}
For such functions the center to edge averages and differences are also periodic.
Further at this point we note that the results proven for periodic functions in the following sections will
also hold with slight modifications for homogeneous Neumann boundary conditions:
\begin{equation}
\begin{array}{lll}
\phi_{m+1,j} = \phi_{m,j}, & \phi_{0,j} = \phi_{1,j}, & j = 1,\ldots,n , \\
\phi_{i,n+1} = \phi_{i,n}, & \phi_{i,0} =\phi_{i,1}, & i = 0,\ldots,m, 
\end{array}
\end{equation}
and homogeneous Dirichlet (no-slip) boundary conditions for velocity, e.g. $\mathbf{u}=\left(u_1,u_2\right)=(0,0)$:
\begin{equation}
\begin{array}{lll}
u_{k,m+1,j} = -u_{k,m,j}, & u_{k,0,j} = -u_{k,1,j}, & j = 1,\ldots,n; ~ k=1,~2\\
u_{k,i,n+1} = -u_{k,i,n}, & u_{k,i,0} =-u_{k,i,1}, & i = 0,\ldots,m; ~ k=1,~2
\end{array}
\end{equation}

In the periodic, Neumann boundary and no-slip condition cases, the boundary terms from the integration by parts vanish.

%%%%%%%%%%%%%%%%%%%%%%%%%%%%%%%%%%%%%%%%%%%%%%
% Discrete Norms and their properties %

\subsection{Discrete Norms and Their Properties}
We now define norms for the cell centered periodic functions.
If $\phi \in \mathcal{C}_{m\times n} $ and periodic then
\begin{eqnarray}
\vectornorm{\phi}_2 : = \sqrt{  (\phi \parallel \phi ) },\\
\vectornorm{\phi}_4: =\sqrt{  (\phi^4 \parallel {\bf 1}) }.
\end{eqnarray}

%%%%%%%%%%%%%%%%%%%%%%%%%%%%%%%%%%%%%%%%%%%%%%%%%%%

%%%%%%%%%%%%%% Proof Energy Stability

%%%%%%%%%%%%%%%%%%%%%%%%%%%%%%%%%%%%%%%%%%%%%%%%%%%

\section{Proof of Theorem \ref{thm_energy_stability}}\label{proof_thm_energy_stability}
In this appendix we present the details of the proof of Theorem  \ref{thm_energy_stability}.
\begin{proof}
Summing Eq. (\ref{discrete_a})  over all cell centers and using the discrete Divergence theorem, we immediately have
mass conservation $\left( \rho^{k+1}\parallel 1\right)= \left( \rho^{k}\parallel 1\right)$.
Choosing $\phi =\rho^{k+1}$ and $\psi = \rho^{k}$ in Eq. (\ref{ineq2c}) we have
the following estimate on the free energy
\begin{equation}
\mathcal{F}[\rho^{k+1}] - \mathcal{F}[\rho^k] \leq ( \delta_{\rho} \mathcal{F}_c[\rho^{k+1}] - \delta_{\rho} \mathcal{F}_e[\rho^k],\rho^{k+1} - \rho^k).
\label{ineq2b}
\end{equation}
This estimate immediately translates to an estimate for the discrete total energy as follows :
 \begin{equation}
 \begin{array}{rl}
\displaystyle \mathcal{E} [ \rho^{k+1},\bu^{k+1}] - \mathcal{E}[\rho^k,\bu^k] & \displaystyle = \frac{1}{2}(\rho^{k+1} \bu^{k+1}\parallel \bu^{k+1} ) - \frac{1}{2}\left( \rho^{k} \bu^{k} \parallel \bu^{k} \right)  \\
& \displaystyle  \quad + \mathcal{F}[\rho^{k+1}] - \mathcal{F}[\rho^k]  \\
& \displaystyle \leq \frac{1}{2}(\rho^{k+1} \bu^{k+1} \parallel \bu^{k+1} ) - \frac{1}{2}\left( \rho^{k} \bu^{k} \parallel \bu^{k} \right)  \\
& \displaystyle  \quad +  (\rho^{k+1} - \rho^k \parallel  \mu^{k+1} ), \\
\end{array}
\label{total_estimate2}
\end{equation}
where we have used Eq.(\ref{ineq1_use}) in in step 2 along with the definition of $\mu^{k+1}$ (Eq. \ref{semisicrete_1_c}).

Multiplying Eq. (\ref{discrete_b}) by $\bu^{k+\frac{1}{2}}$ and summing over cell centers we have
 \begin{equation}
 \begin{array}{rl}
\displaystyle \frac{1}{2} (\rho^{k+1} \bu^{k+1} \parallel \bu^{k+1} )  - \frac{1}{2}( \rho^{k} \bu^k \parallel \bu^{k} )  
& \displaystyle -\frac{1}{2} ( \left(\rho^{k+1} - \rho^k \right) \parallel \mid \bu^{k+1}\mid^2 )  \\
  = & \displaystyle - s \frac{1}{2} (\rho^k \bu^{k+\frac{1}{2}} \parallel \Gradh \mid \bu^{k+1} \mid^2 )\\
& \displaystyle  - s ( \rho^k \bu^{k+\frac{1}{2}} \parallel \Gradh \mu^{k+1} ) \\
& \displaystyle + s \gamma 
 ( \bu^{k+\frac{1}{2}} \parallel \Delta_h \bu^{k+\frac{1}{2}} ) , 
\end{array}
\end{equation}
Now simplifying this in a manner analogous to the proof of 
Theorem \ref{discrete_time_thm} we have
 \begin{equation}
 \begin{array}{rl}
\displaystyle  \frac{1}{2}(\rho^{k+1} \bu^{k+1} \parallel \bu^{k+1} ) - \frac{1}{2}\left( \rho^{k} \bu^{k}\parallel \bu^{k} \right)  &  \\ 
\displaystyle  +  (\rho^{k+1} - \rho^k,  \mu^{k+1} ) & \displaystyle = + s \gamma 
 ( \bu^{k+\frac{1}{2}} \parallel \Delta_h \bu^{k+\frac{1}{2}} ) \\
& \displaystyle =- s \gamma \left( \eipew{D_x u}{D_x u} + \eipns{D_y v}{D_y v} \right),
%\left( ( \Gradh u \parallel \Gradh u ) \right. \\
%&\displaystyle \left. \quad + ( \Gradh v \parallel \Gradh v )  \right),
\end{array}
\label{parts_eq2}
\end{equation} 
where $u,v$ are components of $\bu$. Here we have used the discrete Green's identity (see Proposition \ref{green1stthm-2d}) in the last step. 
Finally using the relation Eq. (\ref{parts_eq2}) and Eq. (\ref{total_estimate2}) we have
 \begin{equation}
 \begin{array}{rl}
 \displaystyle \mathcal{E} [ \rho^{k+1},\bu^{k+1}] - \mathcal{E}[\rho^k,\bu^k] & \displaystyle \leq \frac{1}{2}(\rho^{k+1} \bu^{k+1} \parallel \bu^{k+1} )\\
 &\displaystyle \quad - \frac{1}{2}\left( \rho^{k} \bu^{k} \parallel \bu^{k} \right)  +  (\rho^{k+1} - \rho^k \parallel \mu^{k+1} ) \\
 &\displaystyle = - s \gamma \left( \eipew{D_x u}{D_x u} + \eipns{D_y v}{D_y v} \right) \\
& \displaystyle \leq0. 
\end{array}
\end{equation}

Hence we have shown that the total energy of the system is non-increasing, regardless of the time step $s>0$ i.e $\mathcal{E}[\rho^{k+1},\bu^{k+1}] \leq \mathcal{E}[\rho^k,\bu^k]$.
\end{proof}

%%%%%%%%%%%%%%%%%%%%%%%%%%%%%%%%%%%%%%%%%%%%%%%%%%%

%%%%%%%%%%%%%% Multigrid

%%%%%%%%%%%%%%%%%%%%%%%%%%%%%%%%%%%%%%%%%%%%%%%%%%%

\section{Nonlinear Multigrid Solvers}
In this Appendix we present the details of the nonlinear multigrid algorithm used for time stepping the implicit numerical scheme presented in the 
previous sections.

\subsection{Multigrid method for the convex splitting scheme for the CDFT model}
\label{appendix_MGDFT}
In order to time step the implicit numerical scheme presented in Eqs. (\ref{discrete_a}) - (\ref{discrete_d}) with the CDFT free energy
Eq. (\ref{CS_DFT_discrete}), Eq. (\ref{CS_DFT_discrete1}) and Eq. (\ref{CS_DFT_discrete2}) we need to solve for
$\rho^{k+1},u^{k+1}, v^{k+1} \in \mathcal{C}_{\bar{m}\times \bar{n}}$ such that $\bu^{k+1} = (u^{k+1},v^{k+1})^T$ given $\rho^{k},u^{k}, v^{k} \in \mathcal{C}_{\bar{m}\times \bar{n}}$ such that $\bu^{k} = (u^{k},v^{k})^T$,
that satisfy the following equations 
\begin{equation}
\displaystyle \rho^{k+1}_{ij} - \rho^k_{ij} = -\frac{s}{2h}\left( \rho^k_{i+1,j} u^{k+\frac{1}{2}}_{i+1,j} - \rho^k_{i-1,j} u^{k+\frac{1}{2}}_{i-1,j} +\rho^k_{i,j+1} v^{k+\frac{1}{2}}_{i,j+1} - \rho^k_{i,j-1} v^{k+\frac{1}{2}}_{i,j-1}  \right), \label{MG_DFT_a} 
\end{equation}

\begin{equation}
\begin{array}{rl}
\displaystyle \rho^k_{ij} ( u^{k+1}_{ij} - u^{k}_{ij} ) = 
&\displaystyle  s \left[ -\rho^k_{ij} \omega_{ij}^k v^{k+\frac{1}{2}}_{ij}  \right.  \\
&\displaystyle \quad - \frac{\rho^k}{4h} \left( \mid u^{k+1}_{i+1,j} \mid^2 -\mid u^{k+1}_{i-1,j} \mid^2 \right) \\
&\displaystyle \quad- \frac{\rho^k}{4h} \left( \mid v^{k+1}_{i+1,j} \mid^2 -\mid v^{k+1}_{i-1,j} \mid^2 \right) \\
& \displaystyle  \quad - \frac{\rho^k_{ij}}{2h}  \left( \ln(\rho^{k+1}_{i+1,j}) - \ln (\rho^{k+1}_{i-1,j})  \right)\\
&\displaystyle \quad - \frac{2\rho^k_{ij}(J_e*1)}{2h}\left( \rho^{k+1}_{i+1,j} - \rho^{k+1}_{i-1,j} \right) \\
&\displaystyle \quad 
+ \frac{\rho^k_{ij}}{2h}\left( J*\rho^{k}_{i+1,j} - J*\rho^{k}_{i-1,j} \right) \\
&\displaystyle \quad \left.
 + \frac{2\rho^k_{ij}(J_e*1)}{2h}\left( \rho^{k}_{i+1,j} - \rho^{k}_{i-1,j} \right)  + \gamma \Delta_h u^{k+\frac{1}{2}}_{ij} \right] \label{MG_DFT_b} \\
\end{array}
\end{equation}
and

\begin{equation}
\begin{array}{rl}
\displaystyle \rho^k_{ij} ( v^{k+1}_{ij} - v^{k}_{ij} ) = 
&\displaystyle  s \left[ \rho^k_{ij} \omega_{ij}^k u^{k+\frac{1}{2}}_{ij} \right.  \\ 
&\displaystyle \quad - \frac{\rho^k}{4h} \left( \mid u^{k+1}_{i,j+1} \mid^2 -\mid u^{k+1}_{i,j-1} \mid^2 \right) \\
 &\displaystyle \quad- \frac{\rho^k}{4h} \left( \mid v^{k+1}_{i,j+1} \mid^2 -\mid v^{k+1}_{i,j-1} \mid^2 \right) \\
& \displaystyle  \quad - \frac{\rho^k_{ij}}{2h}  \left( \ln(\rho^{k+1}_{i,j+1}) - \ln(\rho^{k+1}_{i,j-1} ) \right)\\
&\displaystyle \quad - \frac{2\rho^k_{ij}(J_e*1)}{2h}\left( \rho^{k+1}_{i,j+1} - \rho^{k+1}_{i,j-1} \right) \\
&\displaystyle \quad 
+ \frac{\rho^k_{ij}}{2h}\left( J*\rho^{k}_{i,j+1} - J*\rho^{k}_{i,j-1} \right) \\
&\displaystyle \qquad \left.
 + \frac{2\rho^k_{ij}(J_e*1)}{2h}\left( \rho^{k}_{i,j+1} - \rho^{k}_{i,j-1} \right)  + \gamma \Delta_h v^{k+\frac{1}{2}}_{ij} \right], \label{MG_DFT_c} \\
\end{array}
\end{equation}

where 

\begin{equation}
\omega^k_{ij} = \frac{1}{2h}\left( v^k_{i+1,j} - v^k_{i-1,j} - u^k_{i,j+1} + u^k_{i,j-1} \right)\label{MG_DFT_d}.
\end{equation}

The above non-linear problem can be written in terms of a non-linear operator $\mathbf{N}$ and source $\mathbf{S}$ such that
$\mathbf{N}(\bg) = \mathbf{S}$. 
Let $\bg = ( \rho, u,v )^T$ then the $3\times m\times n$ nonlinear operator $\mathbf{N}(\bg) = ( N^{(1)} (\bg), N^{(2)} (\bg), N^{(3)} (\bg) )^T$ can be defined
as
\begin{equation}
{\bf N}_{ij}^{(1)} ( \bg) := \rho_{ij}  + \frac{s}{4h}\left( \rho^k_{i+1,j} u_{i+1,j} - \rho^k_{i-1,j} u_{i-1,j} +\rho^k_{i,j+1} v_{i,j+1} - \rho^k_{i,j-1} v_{i,j-1}  \right), \label{MGN_DFT_a} 
\end{equation}

\begin{equation}
\begin{array}{rl}
{\bf N}_{ij}^{(2)} ( \bg) := \rho^k_{ij} u_{ij} 
&\displaystyle  - s \left[ - \frac{1}{2} \rho^k_{ij} \omega_{ij}^k v_{ij}  \right.  \\ 
&\displaystyle \quad - \frac{\rho^k}{4h} \left( \mid u_{i+1,j} \mid^2 -\mid u_{i-1,j} \mid^2 \right)\\
&\displaystyle \quad- \frac{\rho^k}{4h} \left( \mid v_{i+1,j} \mid^2 -\mid v_{i-1,j} \mid^2 \right) \\
& \displaystyle  \quad - \frac{\rho^k_{ij}}{2h}  \left( \ln(\rho_{i+1,j}) - \ln (\rho_{i-1,j})  \right)\\
&\displaystyle \quad- \frac{2\rho^k_{ij}(J_e*1)}{2h}\left( \rho_{i+1,j} - \rho_{i-1,j} \right) \\
&\displaystyle \quad \left.
  + \frac{\gamma}{2} \Delta_h u_{ij} \right], \label{MGN_DFT_b} \\
\end{array}
\end{equation}
\begin{equation}
\begin{array}{rl}
{\bf N}_{ij}^{(3)} ( \bg) := \rho^k_{ij} v_{ij} 
&\displaystyle   - s \left[  \frac{1}{2} \rho^k_{ij} \omega_{ij}^k u_{ij} \right.  \\ 
& \displaystyle \quad - \frac{\rho^k}{4h} \left( \mid u_{i,j+1} \mid^2 -\mid u_{i,j-1} \mid^2 \right) \\
&\displaystyle \quad- \frac{\rho^k}{4h} \left( \mid v_{i,j+1} \mid^2 -\mid v_{i,j-1} \mid^2 \right) \\
& \displaystyle  \quad - \frac{\rho^k_{ij}}{2h}  \left( \ln(\rho_{i,j+1}) - \ln (\rho_{i,j-1})  \right)\\
&\displaystyle \quad- \frac{2\rho^k_{ij}(J_e*1)}{2h}\left( \rho_{i,j-1} - \rho_{i,j-1} \right) \\
&\displaystyle \qquad \left.
 + \frac{\gamma}{2} \Delta_h v_{ij} \right], \label{MGN_DFT_c} \\
\end{array}
\end{equation}
 where $\omega^{k}_{ij}$ is as defined in Eq. (\ref{MG_DFT_d}).
The $3\times m\times n$ source $\mathbf{S}= ( S^{(1)} , S^{(2)}, S^{(3)})^T$ is defined
as
\begin{equation}
{\bf S}_{ij}^{(1)} := \rho^k_{ij}  - \frac{s}{4h}\left( \rho^k_{i+1,j} u^k_{i+1,j} - \rho^k_{i-1,j} u^k_{i-1,j} +\rho^k_{i,j+1} v^k_{i,j+1} - \rho^k_{i,j-1} v^k_{i,j-1}  \right), \label{MGS_DFT_a} 
\end{equation}

\begin{equation}
\begin{array}{rl}
\displaystyle{\bf S}_{ij}^{(2)} :=  \rho^k_{ij} u^{k}_{ij} 
&\displaystyle  + s \left[ -\frac{1}{2}\rho^k_{ij} \omega_{ij}^k v^{k}_{ij} \right.  \\ 
&\displaystyle \qquad 
+ \frac{\rho^k_{ij}}{2h}\left( J*\rho^{k}_{i+1,j} - J*\rho^{k}_{i-1,j} \right) \\
&\displaystyle \qquad \left.
 + \frac{2\rho^k_{ij}(J_e*1)}{2h}\left( \rho^{k}_{i+1,j} - \rho^{k}_{i-1,j} \right)  + \frac{\gamma}{2} \Delta_h u^{k}_{ij} \right], \label{MGS_DFT_b} \\
\end{array}
\end{equation}
\begin{equation}
\begin{array}{rl}
{\bf S}_{ij}^{(3)} := \rho^k_{ij} v^{k}_{ij}   
&\displaystyle  + s \left[ \frac{1}{2}\rho^k_{ij} \omega_{ij}^k u^{k}_{ij} \right.  \\ 
&\displaystyle \qquad 
+ \frac{\rho^k_{ij}}{2h}\left( J*\rho^{k}_{i,j+1} - J*\rho^{k}_{i,j-1} \right) \\
&\displaystyle \qquad \left.
 + \frac{2\rho^k_{ij}(J_e*1)}{2h}\left( \rho^{k}_{i,j+1} - \rho^{k}_{i,j-1} \right)  + \frac{\gamma}{2} \Delta_h v^{k}_{ij} \right]. \label{MGS_DFT_c} \\
\end{array}
\end{equation}

Given the $\ell$-th guess $(\rho^{\ell}, u^{\ell}, v^{\ell})$ for $(\rho^{k+1}, u^{k+1}, v^{k+1})$ we can get the $(\ell +1)$-th guess
for the $(\rho^{k+1}, u^{k+1}, v^{k+1})$ using the following smoothing scheme.
The smoothing scheme takes the form of a $3 \times 3$ linear system which is in turn solved exactly using Cramer's rule.
\begin{equation}
\displaystyle  \rho^{\ell+1}_{ij} = {\bf S}_{ij}^{(1)}   + \frac{s}{4h}\left( \rho^k_{i+1,j} u^\ell_{i+1,j} - \rho^k_{i-1,j} u^{\ell+1}_{i-1,j} +\rho^k_{i,j+1} v^\ell_{i,j+1} - \rho^k_{i,j-1} v^{\ell+1}_{i,j-1}  \right), \label{MGSM_DFT_a} 
\end{equation}

\begin{equation}
\begin{array}{rl}
\displaystyle \left( \rho^k_{ij}  +\frac{4\gamma}{2 h^2} \right) u^{\ell+1}_{ij} + \frac{s}{2} \rho^k_{ij} \omega_{ij}^k v^{\ell+1}_{ij} = 
&\displaystyle  {\bf S}_{ij}^{(2)} \\
& \displaystyle  + s \left[  - \frac{\rho^k}{4h} \left( \mid u^\ell_{i+1,j} \mid^2 -\mid u^{\ell+1}_{i-1,j} \mid^2 \right)  \right.\\
 & \displaystyle - \frac{\rho^k}{4h} \left( \mid v^\ell_{i+1,j} \mid^2 -\mid v^{\ell+1}_{i-1,j} \mid^2 \right)  \\ 
& \displaystyle   - \frac{\rho^k_{ij}}{2h}  \left( \ln(\rho^\ell_{i+1,j}) - \ln (\rho^{\ell+1}_{i-1,j})  \right) \\
& \displaystyle  - \frac{2\rho^k_{ij}(J_e*1)}{2h}\left( \rho^\ell_{i+1,j} - \rho^{\ell+1}_{i-1,j} \right) \\
&\displaystyle  \left.
  + \frac{\gamma}{2h^2}\left( u^\ell_{i+1,j} +u^{\ell+1}_{i-1,j} +u^\ell_{i,j+1} + u^{\ell+1}_{i,j-1} \right) \right], \label{MGSM_DFT_b} \\
\end{array}
\end{equation}

\begin{equation}
\begin{array}{rl}
 \displaystyle \left(\rho^k_{ij}+\frac{4\gamma}{2 h^2} \right) v^{\ell+1}_{ij} - \frac{s}{2} \rho^k_{ij} \omega_{ij}^k u^{\ell+1}_{ij}  =
&\displaystyle {\bf S}_{ij}^{(3)} \\
& \displaystyle  + s \left[ - \frac{\rho^k}{4h} \left( \mid u^\ell_{i,j+1} \mid^2 -\mid u^{\ell+1}_{i,j-1} \mid^2 \right) \right.\\
& \displaystyle  - \frac{\rho^k}{4h} \left( \mid v^\ell_{i,j+1} \mid^2 -\mid v^{\ell+1}_{i,j-1} \mid^2 \right)   \\ 
& \displaystyle   - \frac{\rho^k_{ij}}{2h}  \left( \ln(\rho^\ell_{i,j+1}) - \ln (\rho^{\ell+1}_{i,j-1})  \right)\\
& \displaystyle - \frac{2\rho^k_{ij}(J_e*1)}{2h}\left( \rho^\ell_{i,j+1} - \rho^{\ell+1}_{i,j-1} \right) \\
&\displaystyle  \left.
  + \frac{\gamma}{2h^2}\left( v^\ell_{i+1,j} +v^{\ell+1}_{i-1,j} +v^\ell_{i,j+1} + v^{\ell+1}_{i,j-1} \right) \right]. \label{MGSM_DFT_c} \\
\end{array}
\end{equation}

\subsection{Multigrid method for the convex splitting scheme for the PFC model}
\label{appendix_MGPFC}
In order to time step the implicit numerical scheme presented in Eqs. (\ref{discrete_a}) - (\ref{discrete_d}) with the PFC free energy
in Eqs. (\ref{CS_PFC_discrete})-(\ref{CS_PFC_discrete2})we need to solve for
$\rho^{k+1},u^{k+1}, v^{k+1} \in \mathcal{C}_{\bar{m}\times \bar{n}}$ such that $\bu^{k+1} = (u^{k+1},v^{k+1})^T$ given $\rho^{k},u^{k}, v^{k} \in \mathcal{C}_{\bar{m}\times \bar{n}}$ such that $\bu^{k} = (u^{k},v^{k})^T$ 
that satisfy the following equations
\begin{equation}
\displaystyle \rho^{k+1}_{ij} - \rho^k_{ij} = -\frac{s}{2h}\left( \rho^k_{i+1,j} u^{k+\frac{1}{2}}_{i+1,j} - \rho^k_{i-1,j} u^{k+\frac{1}{2}}_{i-1,j} +\rho^k_{i,j+1} v^{k+\frac{1}{2}}_{i,j+1} - \rho^k_{i,j-1} v^{k+\frac{1}{2}}_{i,j-1}  \right), \label{MG_PFC_a} 
\end{equation}

\begin{equation}
\begin{array}{rl}
\displaystyle \rho^k_{ij} ( u^{k+1}_{ij} - u^{k}_{ij} ) = 
&\displaystyle  s \left[ -\rho^k_{ij} \omega_{ij}^k v^{k+\frac{1}{2}}_{ij}  \right.  \\ 
& \displaystyle \qquad - \frac{\rho^k}{4h} \left( \mid u^{k+1}_{i+1,j} \mid^2 -\mid u^{k+1}_{i-1,j} \mid^2 \right) \\
& \displaystyle  \qquad- \frac{\rho^k}{4h} \left( \mid v^{k+1}_{i+1,j} \mid^2 -\mid v^{k+1}_{i-1,j} \mid^2 \right) \\
&\displaystyle \qquad \left.
 - \frac{\rho^k_{ij}}{2h}\left( \mu^{k+1}_{i+1,j} - \mu^{k+1}_{i-1,j} \right)  + \gamma \Delta_h u^{k+\frac{1}{2}}_{ij} \right], \label{MG_PFC_b} \\
\end{array}
\end{equation}

\begin{equation}
\begin{array}{rl}
\displaystyle \rho^k_{ij} ( v^{k+1}_{ij} - v^{k}_{ij} ) = 
&\displaystyle  s \left[ \rho^k_{ij} \omega_{ij}^k u^{k+\frac{1}{2}}_{ij}  \right.  \\ 
& \displaystyle \qquad- \frac{\rho^k}{4h} \left( \mid u^{k+1}_{i,j+1} \mid^2 -\mid u^{k+1}_{i,j-1} \mid^2 \right) \\
& \displaystyle  \qquad - \frac{\rho^k}{4h} \left( \mid v^{k+1}_{i,j+1} \mid^2 -\mid v^{k+1}_{i,j-1} \mid^2 \right) \\
&\displaystyle \qquad \left.
 - \frac{\rho^k_{ij}}{2h}\left( \mu^{k+1}_{i,j+1} - \mu^{k+1}_{i,j-1} \right)  + \gamma \Delta_h v^{k+\frac{1}{2}}_{ij} \right] \label{MG_PFC_c} \\
\end{array}
\end{equation}

and 

\begin{equation}
\mu_{ij}^{k+1} =  \frac{1}{3}\left( \rho^{k+1}_{ij}-\frac{3}{2} \right)^3 + \alpha \left( \rho^{k+1}_{ij}-\frac{3}{2} \right) + 2 \Delta_h \rho^k_{ij} +  \Delta_h \kappa^{k+1}_{ij} ,\label{MG_PFC_d} 
\end{equation}

\begin{equation}
\kappa^{k+1}_{ij} = \Delta_h \rho^{k+1}_{ij}, \label{MG_PFC_e} 
\end{equation}
where 

\begin{equation}
\omega^k_{ij} = \frac{1}{2h}\left( v^k_{i+1,j} - v^k_{i-1,j} - u^k_{i,j+1} + u^k_{i,j-1} \right)\label{MG_PFC_f}.
\end{equation}

The above non-linear problem can be written in terms of a non-linear operator $\mathbf{N}$ and source $\mathbf{S}$ such that
$\mathbf{N} = \mathbf{S}$. 
Let $\bg = ( \rho, u,v, \mu,\kappa )^T$ then the $5\times m\times n$ nonlinear operator $\mathbf{N}(\bg) = ( N^{(1)} (\bg), N^{(2)} (\bg), N^{(3)} (\bg), N^{(4)} (\bg), N^{(5)} (\bg))^T$ can be defined
as
\begin{equation}
{\bf N}_{ij}^{(1)} ( \bg) := \rho_{ij}  + \frac{s}{4h}\left( \rho^k_{i+1,j} u_{i+1,j} - \rho^k_{i-1,j} u_{i-1,j} +\rho^k_{i,j+1} v_{i,j+1} - \rho^k_{i,j-1} v_{i,j-1}  \right), \label{MGN_PFC_a} 
\end{equation}

\begin{equation}
\begin{array}{rl}
{\bf N}_{ij}^{(2)} ( \bg) := \rho^k_{ij} u_{ij} 
&\displaystyle  - s \left[ -\frac{1}{2} \rho^k_{ij} \omega_{ij}^k v_{ij}  \right.  \\ 
&\displaystyle \qquad - \frac{\rho^k}{4h} \left( \mid u_{i+1,j} \mid^2 -\mid u_{i-1,j} \mid^2 \right) \\
&\displaystyle \qquad  - \frac{\rho^k}{4h} \left( \mid v_{i+1,j} \mid^2 -\mid v_{i-1,j} \mid^2 \right) \\
& \qquad \displaystyle+ \frac{\rho^k_{ij}}{2h}\left( \mu_{i+1,j} - \mu_{i-1,j} \right) \\
&\displaystyle \qquad \left.
 + \frac{\gamma}{2} \Delta_h u_{ij} \right], \label{MGN_PFC_b} \\
\end{array}
\end{equation}
\begin{equation}
\begin{array}{rl}
{\bf N}_{ij}^{(3)} ( \bg) := \rho^k_{ij} v_{ij} 
&\displaystyle  - s \left[  \frac{1}{2} \rho^k_{ij} \omega_{ij}^k u_{ij} \right.  \\ 
& \displaystyle \qquad  - \frac{\rho^k}{4h} \left( \mid u_{i,j+1} \mid^2 -\mid u_{i,j-1} \mid^2 \right)\\
 &\displaystyle \qquad  - \frac{\rho^k}{4h} \left( \mid v_{i,j+1} \mid^2 -\mid v_{i,j-1} \mid^2 \right) \\
& \qquad \displaystyle+ \frac{\rho^k_{ij}}{2h}\left( \mu_{i,j-1} - \mu_{i,j-1} \right) \\
&\displaystyle \qquad \left.
  + \frac{\gamma}{2} \Delta_h v_{ij} \right], \label{MGN_PFC_c} \\
\end{array}
\end{equation}
\begin{equation}
{\bf N}_{ij}^{(4)} ( \bg) := \mu_{ij}^{k+1} -  \frac{1}{3}\left( \rho^{k+1}_{ij}-\frac{3}{2} \right)^3 - \alpha \left( \rho^{k+1}_{ij}-\frac{3}{2} \right) -  \Delta_h \kappa^{k+1}_{ij} ,
\label{MGN_PFC_d}
\end{equation}
\begin{equation}
{\bf N}_{ij}^{(5)} ( \bg) :=  \kappa^{k+1}_{ij} - \Delta_h \rho^{k+1}_{ij}, 
\label{MGN_PFC_e}
\end{equation}
 where $\omega^{k}_{ij}$ is as defined in Eq. (\ref{MG_PFC_f}). \\
The $5\times m\times n$ source $\mathbf{S}= ( S^{(1)} , S^{(2)}, S^{(3)}, S^{(4)}, S^{(5)})^T$ is defined
as
\begin{equation}
{\bf S}_{ij}^{(1)} := \rho^k_{ij}  - \frac{s}{4h}\left( \rho^k_{i+1,j} u^k_{i+1,j} - \rho^k_{i-1,j} u^k_{i-1,j} +\rho^k_{i,j+1} v^k_{i,j+1} - \rho^k_{i,j-1} v^k_{i,j-1}  \right), \label{MGS_PFC_a} 
\end{equation}

\begin{equation}
{\bf S}_{ij}^{(2)} :=  \rho^k_{ij} u^{k}_{ij} 
 + s \left[ -\frac{1}{2}\rho^k_{ij} \omega_{ij}^k v^{k}_{ij} + \frac{\gamma}{2} \Delta_h u^{k}_{ij} \right], \label{MGS_PFC_b} 
\end{equation}
\begin{equation}
{\bf S}_{ij}^{(3)} := \rho^k_{ij} v^{k}_{ij}   
  + s \left[ \frac{1}{2}\rho^k_{ij} \omega_{ij}^k u^{k}_{ij} + \frac{\gamma}{2} \Delta_h v^{k}_{ij} \right], \label{MGS_PFC_c} 
\end{equation}
\begin{equation}
{\bf S}_{ij}^{(4)} := 2 \Delta_h \rho^k_{ij}  \label{MGS_PFC_d} 
\end{equation}
\begin{equation}
{\bf S}_{ij}^{(5)} := 0  \label{MGS_PFC_e} 
\end{equation}

Given the $\ell$-th guess $(\rho^{\ell}, u^{\ell}, v^{\ell}, \mu^\ell, \kappa^{\ell})$ for $(\rho^{k+1}, u^{k+1}, v^{k+1},\mu^{k+1}, \kappa^{k+1})$ we can get the $(\ell +1)$-th guess
for the $(\rho^{k+1}, u^{k+1}, v^{k+1},\mu^{k+1}, \kappa^{k+1})$ using the following smoothing scheme.
The smoothing scheme presented here takes the form of a $5 \times 5$ linear system which is in turn solved exactly using Cramer's rule.
\begin{equation}
 \rho^{\ell+1}_{ij} = {\bf S}_{ij}^{(1)}   + \frac{s}{4h}\left( \rho^k_{i+1,j} u^\ell_{i+1,j} - \rho^k_{i-1,j} u^{\ell+1}_{i-1,j} +\rho^k_{i,j+1} v^\ell_{i,j+1} - \rho^k_{i,j-1} u^{\ell+1}_{i,j-1}  \right), \label{MGSM_PFC_a} 
\end{equation}

\begin{equation}
\begin{array}{rl}
\displaystyle \left( \rho^k_{ij}  +\frac{4\gamma}{2 h^2} \right) u^{\ell+1}_{ij} + \frac{s}{2} \rho^k_{ij} \omega_{ij}^k v^{\ell+1}_{ij} = 
&\displaystyle  {\bf S}_{ij}^{(2)} + s \left[  - \frac{\rho^k}{4h} \left( \mid u^\ell_{i+1,j} \mid^2 -\mid u^{\ell+1}_{i-1,j} \mid^2 \right) \right. \\
&\displaystyle- \frac{\rho^k}{4h} \left( \mid v^\ell_{i+1,j} \mid^2 -\mid v^{\ell+1}_{i-1,j} \mid^2 \right)   \\ 
& \displaystyle   
+ \frac{\rho^k_{ij}}{2h}\left( \mu^\ell_{i+1,j} - \mu^{\ell+1}_{i-1,j} \right) \\
&\displaystyle  \left.
  + \frac{\gamma}{2h^2}\left( u^\ell_{i+1,j} +u^{\ell+1}_{i-1,j} +u^\ell_{i,j+1} + u^{\ell+1}_{i,j-1} \right) \right], \label{MGSM_PFC_b} \\
\end{array}
\end{equation}

\begin{equation}
\begin{array}{rl}
\displaystyle \left(\rho^k_{ij}+\frac{4\gamma}{2 h^2} \right) v^{\ell+1}_{ij} - \frac{s}{2} \rho^k_{ij} \omega_{ij}^k u^{\ell+1}_{ij}  =
&\displaystyle {\bf S}_{ij}^{(3)} + s \left[ - \frac{\rho^k}{4h} \left( \mid u^\ell_{i,j+1} \mid^2 -\mid u^{\ell+1}_{i,j-1} \mid^2 \right)  \right.  \\
&\displaystyle- \frac{\rho^k}{4h} \left( \mid v^\ell_{i,j+1} \mid^2 -\mid v^{\ell+1}_{i,j-1} \mid^2 \right)  \\ 
& \displaystyle   + \frac{\rho^k_{ij}}{2h}\left( \mu^\ell_{i,j+1} - \mu^{\ell+1}_{i,j-1} \right) \\
&\displaystyle  \left.
  + \frac{\gamma}{2h^2}\left( v^\ell_{i+1,j} +v^{\ell+1}_{i-1,j} +v^\ell_{i,j+1} + v^{\ell+1}_{i,j-1} \right) \right], \label{MGSM_PFC_c} \\
\end{array}
\end{equation}

\begin{equation}
\begin{array}{rl}
\displaystyle-\left( \alpha +  \left( \rho_{ij}^\ell -3/2 \right)^2 \right) \rho_{ij}^{\ell +1} & \\
\displaystyle + \mu_{ij}^{\ell+1} + \frac{4}{h^2} \kappa_{ij}^{\ell +1}  \\
= &\displaystyle  {\bf S}_{ij}^{(4)} \\
&\displaystyle - \frac{2}{3} \left(\rho_{ij}^{\ell} -\frac{3}{2} \right)^3 + \frac{3}{2} \left( \alpha +  \left( \rho_{ij}^\ell -3/2 \right)^2 \right)  \\
& \displaystyle + \frac{1}{h^2}\left( \kappa_{i+1,j}^\ell + \kappa_{i-1,j}^{\ell +1} + \kappa_{i,j+1}^{\ell} + \kappa_{i,j-1}^{\ell+1} \right), 
\end{array}\label{MGSM_PFC_d} 
\end{equation}

\begin{equation}
 \frac{4}{h^2} \rho_{ij}^{\ell +1}  + \kappa_{ij}^{\ell +1} =  {\bf S}_{ij}^{(5)} + \frac{1}{h^2}\left( \rho_{i+1,j}^\ell + \rho_{i-1,j}^{\ell +1} + \rho_{i,j+1}^{\ell} + \rho_{i,j-1}^{\ell+1} \right). \label{MGSM_PFC_e} 
\end{equation}

In Eq. (\ref{MGSM_PFC_d}) we have linearized the cubic term using the approximation $(\phi^{\ell+1})^3 \approx 3 (\phi^{\ell})^2 \phi^{\ell+1}  - 2 (\phi^{\ell})^3$ within the lexicographic index $\ell$.

\subsection{ Damped Gauss-Seidel Iteration}
\label{DAMP}
To further aid the convergence of the method and improve the stability of the Gauss-Seidel method 
we use a damped Gauss- Seidel method.
Given $\rho^{k}$ and $\rho^{k+1,\ell}$ the $ell$-th guess for $\rho^{k+1}$, let $\rho^{k+1,\ell+1,GS}$ be the $(\ell+1)$-th guess obtained from a smoothing scheme.
The damped Gauss-Seidel method updates the new guess $\rho^{k+1,\ell +1}$ as
\begin{equation}
\rho^{k+1,\ell+1} = ( 1 - w_{damp} ) \rho^{k+1,\ell+1,GS}  + w_{damp} \rho^{k+1,\ell}
\end{equation}
where $w_{damp}$ is a damping co-efficient set to be $w_{damp} =0.5$ in all simulations presented.
\end{appendix}

%%%%%%%%%%%%%%%%%%%%%%%%%%%%%%%%%%%%%%%%%%%%%%%%%%%%%\

%%%%%%%%%%%%%%%% Bibliography %%%%%%%%%%%%%%%%%%%%%%%%%%%%%%%%%

%%%%%%%%%%%%%%%%%%%%%%%%%%%%%%%%%%%%%%%%%%%%%%%%%%%%%%

%\vfill\eject

%\bibliographystyle{elsarticle-num}
%\bibliography{kdft} 

\begin{thebibliography}{10}
\expandafter\ifx\csname url\endcsname\relax
  \def\url#1{\texttt{#1}}\fi
\expandafter\ifx\csname urlprefix\endcsname\relax\def\urlprefix{URL }\fi
\expandafter\ifx\csname href\endcsname\relax
  \def\href#1#2{#2} \def\path#1{#1}\fi

\bibitem{Hansen2006}
J.~P. Hansen, I.~R. McDonald, {Theory of Simple Fluids}, 3rd Edition, Academic
  Press, 2006.

\bibitem{Lutsko2010}
J.~Lutsko, Recent developments in classical density functional theory, Advances
  in chemical physics 144.

\bibitem{Marconi1999}
U.~Marconi, P.~Tarazona, {Dynamic density functional theory of fluids}, Journal
  of Chemical Physics 110 (1999) 8032--8044.

\bibitem{2012Kaliadasis}
B.~D. Goddard, G.~A. Pavliotis, S.~Kalliadasis, The overdamped limit of dynamic
  density functional theory: Rigorous results, SIAM Multiscale Model. Simul.
  10~(2) (2012) 633--663.

\bibitem{Espanol2009}
P.~Espa\~{n}ol, H.~L\"{o}wen, {Derivation of dynamical density functional
  theory using the projection operator technique.}, The Journal of Chemical
  Physics 131~(24) (2009) 244101.
\newblock \href {http://dx.doi.org/10.1063/1.3266943}
  {\path{doi:10.1063/1.3266943}}.

\bibitem{Yoshimori2005}
A.~Yoshimori, {Microscopic derivation of time-dependent density functional
  methods}, Physical Review E 71~(3) (2005) 031203.
\newblock \href {http://dx.doi.org/10.1103/PhysRevE.71.031203}
  {\path{doi:10.1103/PhysRevE.71.031203}}.

\bibitem{Archer2009}
A.~J. Archer, {Dynamical density functional theory for molecular and colloidal
  fluids: a microscopic approach to fluid mechanics}, The Journal of Chemical
  Physics 130 (2009) 014509.
\newblock \href {http://dx.doi.org/10.1063/1.3054633}
  {\path{doi:10.1063/1.3054633}}.

\bibitem{Lutsko2012}
J.~F. Lutsko, {A dynamical theory of nucleation for colloids and
  macromolecules.}, The Journal of Chemical Physics 136~(3) (2012) 034509.
\newblock \href {http://dx.doi.org/10.1063/1.3677191}
  {\path{doi:10.1063/1.3677191}}.

\bibitem{Chavanis2011}
P.-H. Chavanis, Brownian particles with long- and short-range interactions,
  Physica A: Statistical Mechanics and its Applications 390 (2011) 1546--1574.

\bibitem{Baskaran2014}
A.~Baskaran, A.~Baskaran, J.~Lowengrub, Kinetic density functional theory of
  freezing, Journal of Chemical Physics 141 (2013) 174506.

\bibitem{PFC_book}
N.~Provatas, K.~Elder, Phase field methods in material science and engineering,
  John Wiley \& Sons, 2010.

\bibitem{Eyre}
D.~Eyre, Unconditionally gradient stable time marching the Cahn-Hilliard
  equation, in: J.~W. Bullard, R.~Kalia, M.~Stonenham, L.~Q. Chen (Eds.),
  Computational and Mathematical Models of Microstructure Evolution,, no. 529,
  Materials Research Society, Warrendale, PA, 1998, pp. 1686--1712.

\bibitem{WISE2009}
S.~M. Wise, C.~Wang, J.~Lowengrub, An energy-stable and convergent
  finite-difference scheme for the phase field crystal equation, SIAM Journal
  of Numerical Analysis 47~(3) (2009) 2269--2288.

\bibitem{WISE2009_2}
Z.~Hu, S.~M. Wise, C.~Wang, J.~Lowengrub, Stable and efficient
  finite-difference nonlinear-multigrid schemes for the phase field crystal
  equation, Journal of Computational Physics 228~(15) (2009) 5323--5339.

\bibitem{MPFC2011}
S.~M. Wise, C.~Wang, An energy stable and convergent finite-difference scheme
  for the modified phase field crystal equation, SIAM Journal of Numerical
  Analysis 46~(3) (2011) 945--969.

\bibitem{MPFC2013}
A.~Baskaran, Z.~Hu, J.~Lowengrub, C.~Wang, S.~M. Wise, P.~Zhou, Energy stable
  and efficient finite-difference nonlinear multigrid schemes for the modified
  phase field crystal equation, Journal of Computational Physics 250 (2013)
  270--292.

\bibitem{MPFC2013_2}
A.~Baskaran, J.~Lowengrub, C.~Wang, S.~M. Wise, Convergence analysis of a
  second order convex splitting scheme for the modified phase field crystal
  equation, SIAM Journal of Numerical Analysis 51~(3) (2013) 2851--2873.

\bibitem{Zhen1}
Z.~Guan, C.~Wang, S.~M. Wise, A convergent convex splitting scheme for the
  periodic nonlocal cahn-hilliard equation, Numerische Mathematik 128~(2)
  (2014) 377--406.

\bibitem{Zhen2}
Z.~Guan, J.~Lowengrub, C.~Wang, S.~M. Wise, Second order convex splitting
  schemes for periodic nonlocal cahn--hilliard and allen--cahn equations,
  Journal of Computational Physics 277 (2014) 48--71.

\bibitem{Kim2012}
J.~Kim, Phase-field models for multi-component fluid flows, Communications in
  Computational Physics 12~(3) (2012) 612--661.

\bibitem{Axel2015}
S.~Praetorius, A.~Voigt, A phase field crystal model for colloidal suspensions
  with hydrodynamic interactions, arXiv:1310.5495.

\bibitem{Laszlo2014}
G.~I. T\'oth, L.~Gr\'an\'asy, G.~Tegze, Nonlinear hydrodynamic theory of
  crystallization, Journal of Physics: Condensed Matter 26 (2014) 055001.

\bibitem{Ramakrishnan1979}
T.~Ramakrishnan, M.~Yussouff, {First-principles order-parameter theory of
  freezing}, Physical Review B 19~(5) (1979) 2775.

\bibitem{Elder2004}
K.~Elder, M.~Grant, Modeling elastic and plastic deformations in nonequilibrium
  processing using phase field crystals, Physical Review E 70 (2004) 051605.

\bibitem{VanTeeffelen2009}
S.~van Teeffelen, R.~Backofen, A.~Voigt, H.~L\"{o}wen, {Derivation of the
  phase-field-crystal model for colloidal solidification}, Physical Review E
  79~(5) (2009) 051404.
\newblock \href {http://dx.doi.org/10.1103/PhysRevE.79.051404}
  {\path{doi:10.1103/PhysRevE.79.051404}}.

\bibitem{Akusti}
M.~A. Jaatinen, T.~Ala-Nissia, Extended phase diagram of the three-dimensional
  phase field crystal model, Journal of Physics: Condensed Matter 22 (2010)
  205402.

\bibitem{Trottenberg}
U.~Trottenberg, C.~Oosterlee, A.~Sch\"{u}ller, Multigrid, Elsevier Academic
  Press, 2001.

\bibitem{Baskaran2010}
A.~Baskaran, J.~P. Devita, P.~Smereka, Kinetic monte carlo simulation of
  strained heteroepitaxial growth with intermixing, Continuum Mechanics and
  Thermodynamics 22 (2010) 1--26.

%\bibitem{KDFT2}
%A.~Baskaran, A.~Baskaran, J.~Lowengrub, Effect of flow on solid liquid phase
%  transition: A kinetic density functional theory study, In preparation.

\end{thebibliography}

%%%%%%%%%%%%%%%%%%%%%%%%%%%%%%%%%%%%%%%%%%%%%%%%%%%%%\

%%%%%%%%%%%%%%%% Tables and Figs%%%%%%%%%%%%%%%%%%%%%%%%%%%%%%%%%

%%%%%%%%%%%%%%%%%%%%%%%%%%%%%%%%%%%%%%%%%%%%%%%%%%%%%%

\let\thefigure\thefigureSAVED
\let\thetable\thetableSAVED

\newpage
\begin{table}[ht]
	\begin{center}
	\begin{tabular}{|c|c|c|c|}
	\hline
$h_c$ & $h_f$ & $\parallel e_{\rho}^{h_c;h_f} \parallel_2$ & Rate of Convergence of $\rho$
	\\ 
	\hline
$\frac{32}{16}$ & $\frac{32}{32}$  & $4.4800 \times 10^{-5}$ & ---
	\\ 
	\hline
$\frac{32}{32}$ & $\frac{32}{64}$  & $2.4951 \times 10^{-7}$ & 7.4883
	\\
	\hline
$\frac{32}{64}$ & $\frac{32}{128}$  & $4.8685 \times 10^{-8}$ & 2.3576
	\\
	\hline
$\frac{32}{128}$ & $\frac{32}{256}$  & $1.0802 \times 10^{-8}$ & 2.1722
	\\
	\hline
$\frac{32}{256}$ & $\frac{32}{512}$  & $2.2465 \times 10^{-9}$ & 2.2655
	\\
	\hline
	\end{tabular}
\caption{The error and convergence rate for the density  for the convergence test of the hydrodynamic CDFT model discussed in Section \ref{Convergence_DFT}. }
	\label{Table_DFT1}
	\end{center}
	\end{table}
	
	\begin{table}[ht]
	\begin{center}
	\begin{tabular}{|c|c|c|c|}
	\hline
$h_c$ & $h_f$ & $\parallel e_{u}^{h_c;h_f} \parallel_2$ & Rate of Convergence of $u$
	\\ 
	\hline
$\frac{32}{16}$ & $\frac{32}{32}$  & $5.3251 \times 10^{-5}$ & ---
	\\ 
	\hline
$\frac{32}{32}$ & $\frac{32}{64}$  & $2.4077 \times 10^{-7}$ & 7.7890
	\\
	\hline
$\frac{32}{64}$ & $\frac{32}{128}$  & $4.6670 \times 10^{-8}$ & 2.3671
	\\
	\hline
$\frac{32}{128}$ & $\frac{32}{256}$  & $1.0715 \times 10^{-8}$ & 2.1229
	\\
	\hline
$\frac{32}{256}$ & $\frac{32}{512}$  & $2.5396  \times 10^{-9}$ & 2.0770
	\\
	\hline
	\end{tabular}
\caption{ The error and convergence rate for the horizontal component of the velocity for the convergence test of the hydrodynamic CDFT model  discussed in Section \ref{Convergence_DFT}. }
	\label{Table_DFT2}
	\end{center}
	\end{table}

\begin{table}[ht]
	\begin{center}
	\begin{tabular}{|c|c|c|c|}
	\hline
$h_c$ & $h_f$ & $\parallel e_{v}^{h_c;h_f} \parallel_2$ & Rate of Convergence of $v$
	\\ 
	\hline
$\frac{32}{16}$ & $\frac{32}{32}$  & $5.3467\times 10^{-5}$ & ---
	\\ 
	\hline
$\frac{32}{32}$ & $\frac{32}{64}$  & $2.5007\times 10^{-7}$ & 7.7402
	\\
	\hline
$\frac{32}{64}$ & $\frac{32}{128}$  & $4.8696e \times 10^{-8}$ & 2.3604
	\\
	\hline
$\frac{32}{128}$ & $\frac{32}{256}$  & $1.1342\times 10^{-8}$ & 2.1021
	\\
	\hline
$\frac{32}{256}$ & $\frac{32}{512}$  & $2.7102  \times 10^{-9}$ & 2.0652
	\\
	\hline
	\end{tabular}
\caption{ The error and convergence rate for the vertical component of the velocity  for the convergence test of the hydrodynamic CDFT model discussed in Section  \ref{Convergence_DFT} }
	\label{Table_DFT3}
	\end{center}
	\end{table}

%%%%%%%%%%%%%PFC Convergence %%%%%%%%%%%%%%%%%%

\begin{table}[ht]
	\begin{center}
	\begin{tabular}{|c|c|c|c|}
	\hline
$h_c$ & $h_f$ & $\parallel e_{rho}^{h_c;h_f} \parallel_2$ & Rate of Convergence of $\rho$
	\\ 
	\hline
$\frac{32}{16}$ & $\frac{32}{32}$  & $5.9041\times 10^{-6}$ & ---
	\\ 
	\hline
$\frac{32}{32}$ & $\frac{32}{64}$  & $1.1318\times 10^{-6}$ & 2.3831
	\\
	\hline
$\frac{32}{64}$ & $\frac{32}{128}$  & $2.6468 \times 10^{-7}$ & 2.0963
	\\
	\hline
$\frac{32}{128}$ & $\frac{32}{256}$  & $6.2968 \times 10^{-8}$ & 2.0716
	\\
	\hline
$\frac{32}{256}$ & $\frac{32}{512}$  & $1.5300  \times 10^{-8}$ & 2.0411 
	\\
	\hline
	\end{tabular}
\caption{The error and convergence rate for the density  for the convergence test of the hydrodynamic PFC model discussed in Section  \ref{Convergence_PFC}. }
	\label{Table_PFC1}
	\end{center}
	\end{table}
	
	\begin{table}[ht]
	\begin{center}
	\begin{tabular}{|c|c|c|c|}
	\hline
$h_c$ & $h_f$ & $\parallel e_{u}^{h_c;h_f} \parallel_2$ & Rate of Convergence of $u$
	\\ 
	\hline
$\frac{32}{16}$ & $\frac{32}{32}$  & $7.8360\times 10^{-7}$ & ---
	\\ 
	\hline
$\frac{32}{32}$ & $\frac{32}{64}$  & $1.3122\times 10^{-7}$ & 2.5781
	\\
	\hline
$\frac{32}{64}$ & $\frac{32}{128}$  & $2.5793 \times 10^{-8}$ & 2.3470
	\\
	\hline
$\frac{32}{128}$ & $\frac{32}{256}$  & $5.6521 \times 10^{-9}$ & 2.1901
	\\
	\hline
$\frac{32}{256}$ & $\frac{32}{512}$  & $1.3191  \times 10^{-9}$ & 2.0992
	\\
	\hline
	\end{tabular}
\caption{ The error and convergence rate for the horizontal component of the velocity   for the convergence test of the hydrodynamic PFC model discussed in Section  \ref{Convergence_PFC}. }
	\label{Table_PFC2}
	\end{center}
	\end{table}

\begin{table}[ht]
	\begin{center}
	\begin{tabular}{|c|c|c|c|}
	\hline
$h_c$ & $h_f$ & $\parallel e_{v}^{h_c;h_f} \parallel_2$ & Rate of Convergence of $v$
	\\ 
	\hline
$\frac{32}{16}$ & $\frac{32}{32}$  & $6.6965\times 10^{-7}$ & ---
	\\ 
	\hline
$\frac{32}{32}$ & $\frac{32}{64}$  & $1.4285\times 10^{-7}$ & 2.2289
	\\
	\hline
$\frac{32}{64}$ & $\frac{32}{128}$  & $3.2625 \times 10^{-8}$ & 2.1305
	\\
	\hline
$\frac{32}{128}$ & $\frac{32}{256}$  & $7.7450 \times 10^{-9}$ & 2.0747
	\\
	\hline
$\frac{32}{256}$ & $\frac{32}{512}$  & $1.8833  \times 10^{-9}$ & 2.0400
	\\
	\hline
	\end{tabular}
\caption{ The error and convergence rate for the vertical component of the velocity for the convergence test of the hydrodynamic PFC model discussed in Section  \ref{Convergence_PFC}.}
	\label{Table_PFC3}
	\end{center}
	\end{table}

\begin{figure}[!ht]
\begin{center}
\begin{picture}(270,500)
\put(0,-5){\includegraphics[scale=0.6]{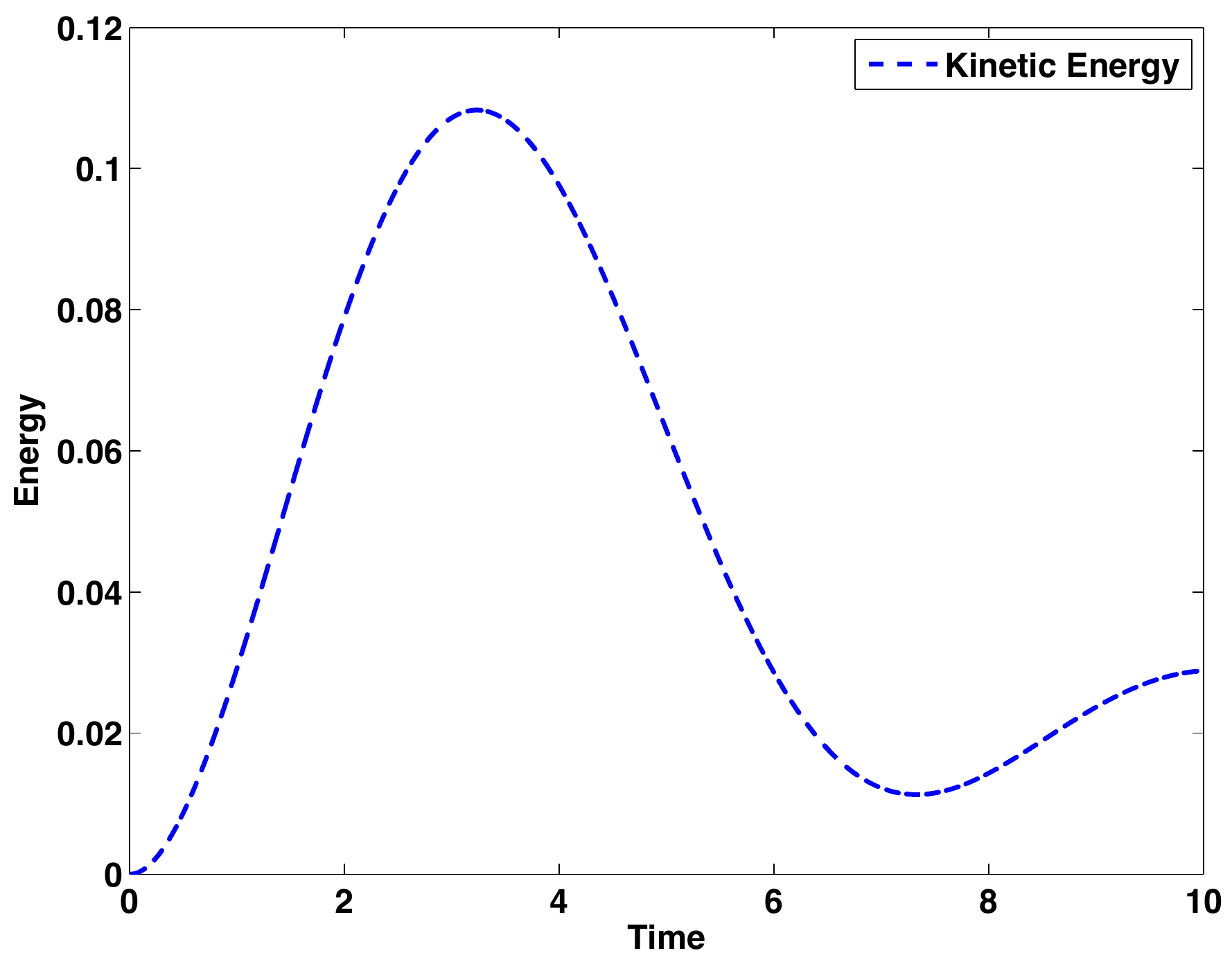}}
\put(0, 235){\includegraphics[scale=0.6]{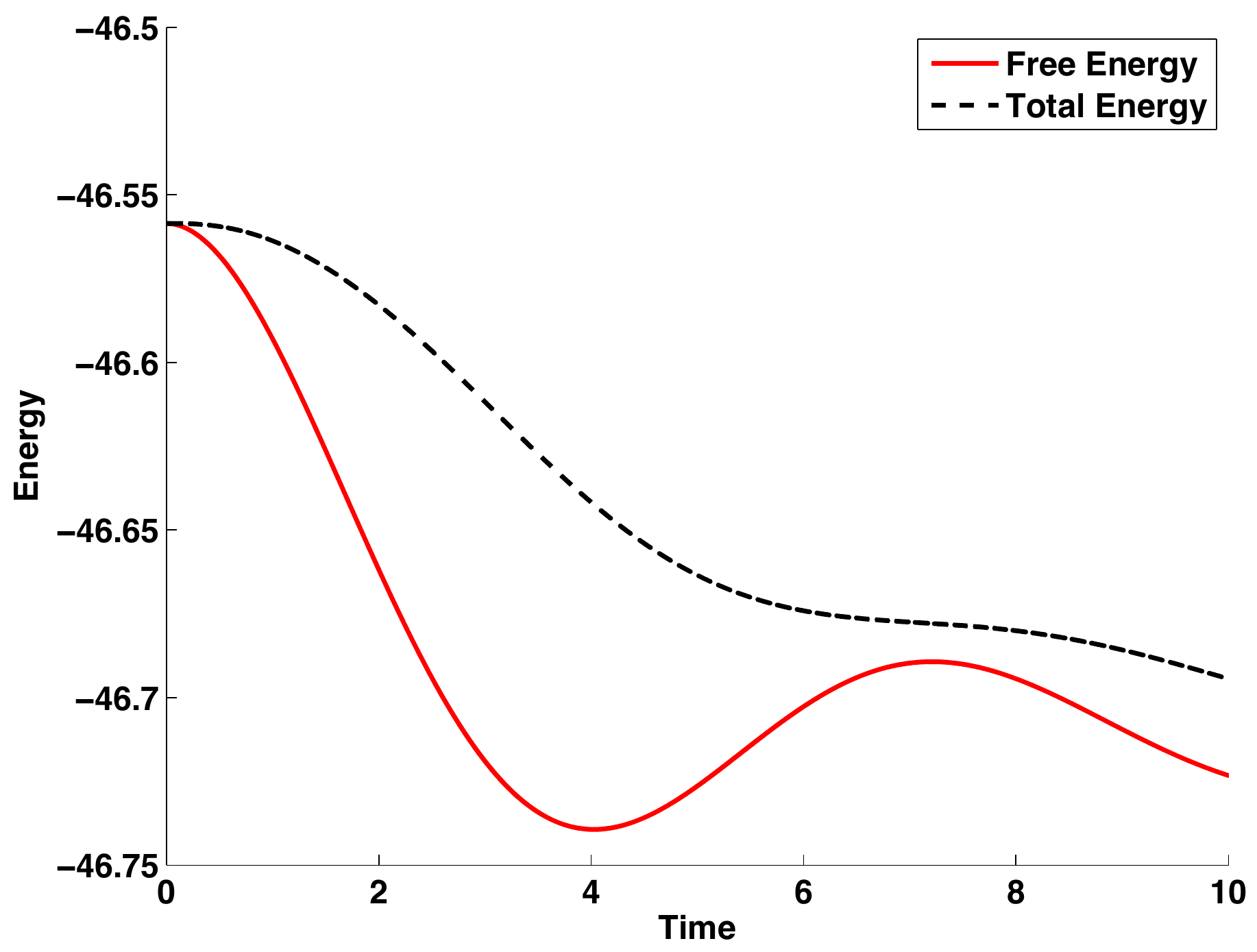}}
\end{picture}
\end{center}%\includegraphics{Figs/grow_15000.png}
\caption{The evolution of kinetic energy, Helmholtz free energy $\mathcal{F}[\rho]$ and total energy $\mathcal{E}[\rho, \bu]$, as labeled,  for the
simulation corresponding to the convergence test of the Hydrodynamic  CDFT model described in Section  \ref{Convergence_DFT}  with ($h=512$).}
\label{fig_DFT_t1}
\end{figure}

\begin{figure}[!ht]
\begin{center}
\begin{picture}(270,500)
\put(0,-5){\includegraphics[scale=0.6]{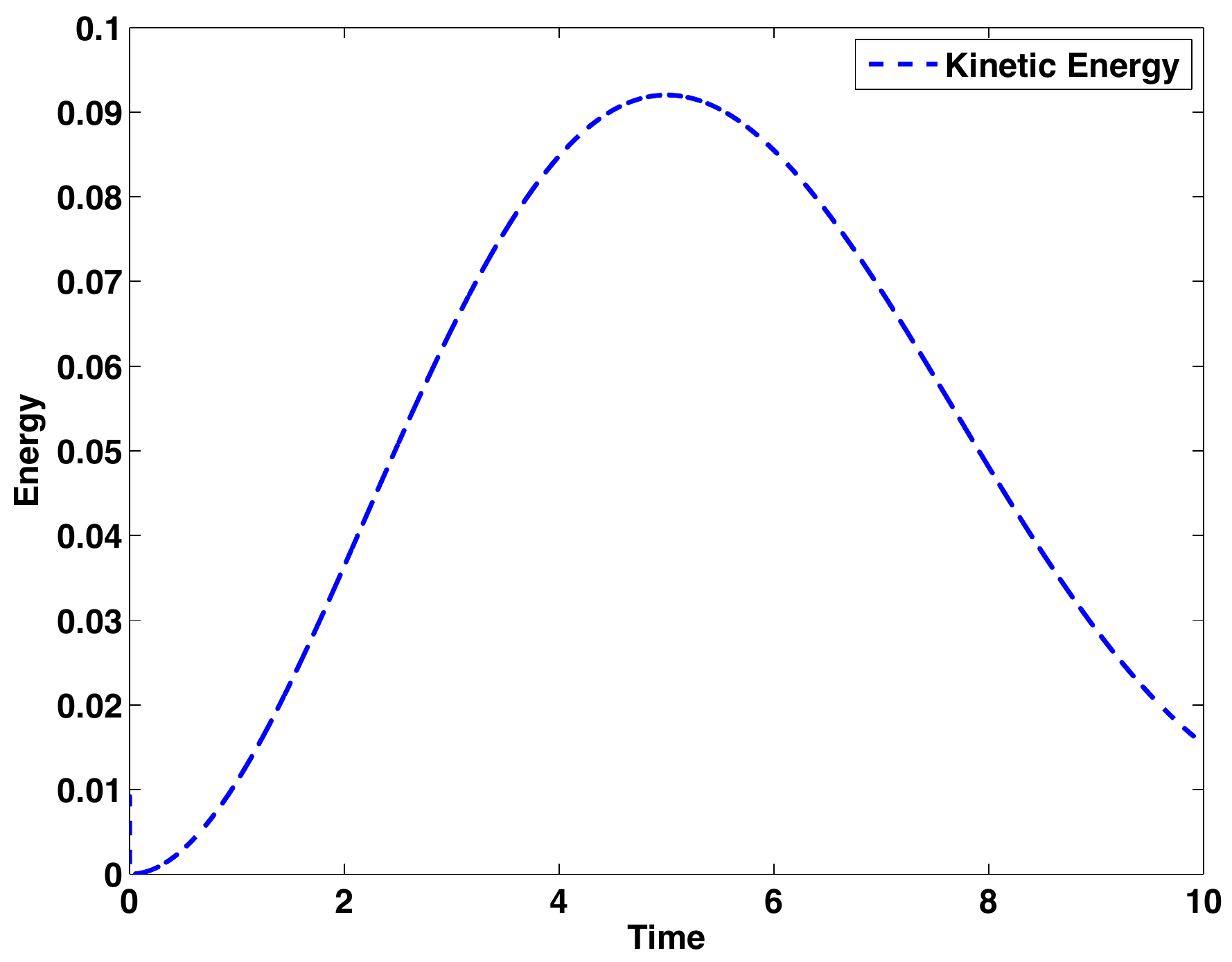}}
\put(0, 235){\includegraphics[scale=0.6]{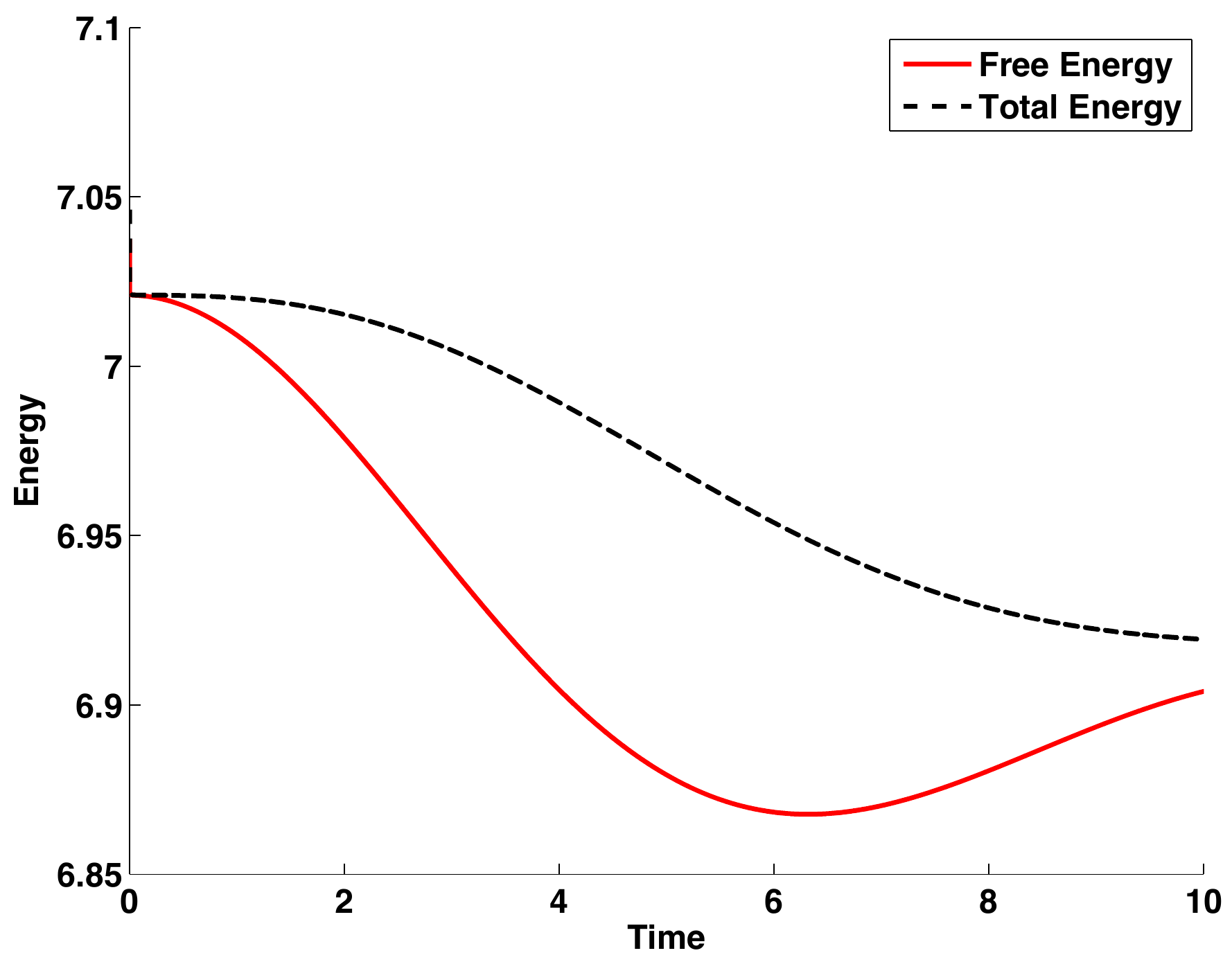}}
\end{picture}
\end{center}%\includegraphics{Figs/grow_15000.png}
\caption{The evolution of kinetic energy, HelmholtzfFree energy $\mathcal{F}[\rho]$ and total energy $\mathcal{E}[\rho, \bu]$, as labeled, for the
simulation corresponding to the convergence test of the Hydrodynamic  PFC model described in Section  \ref{Convergence_PFC} with ($h=512$).}
\label{fig_PFC_t1}
\end{figure}

\begin{figure}[!ht]
\begin{center}
\begin{picture}(350,140)
\put(-30,5){\includegraphics[scale=0.39]{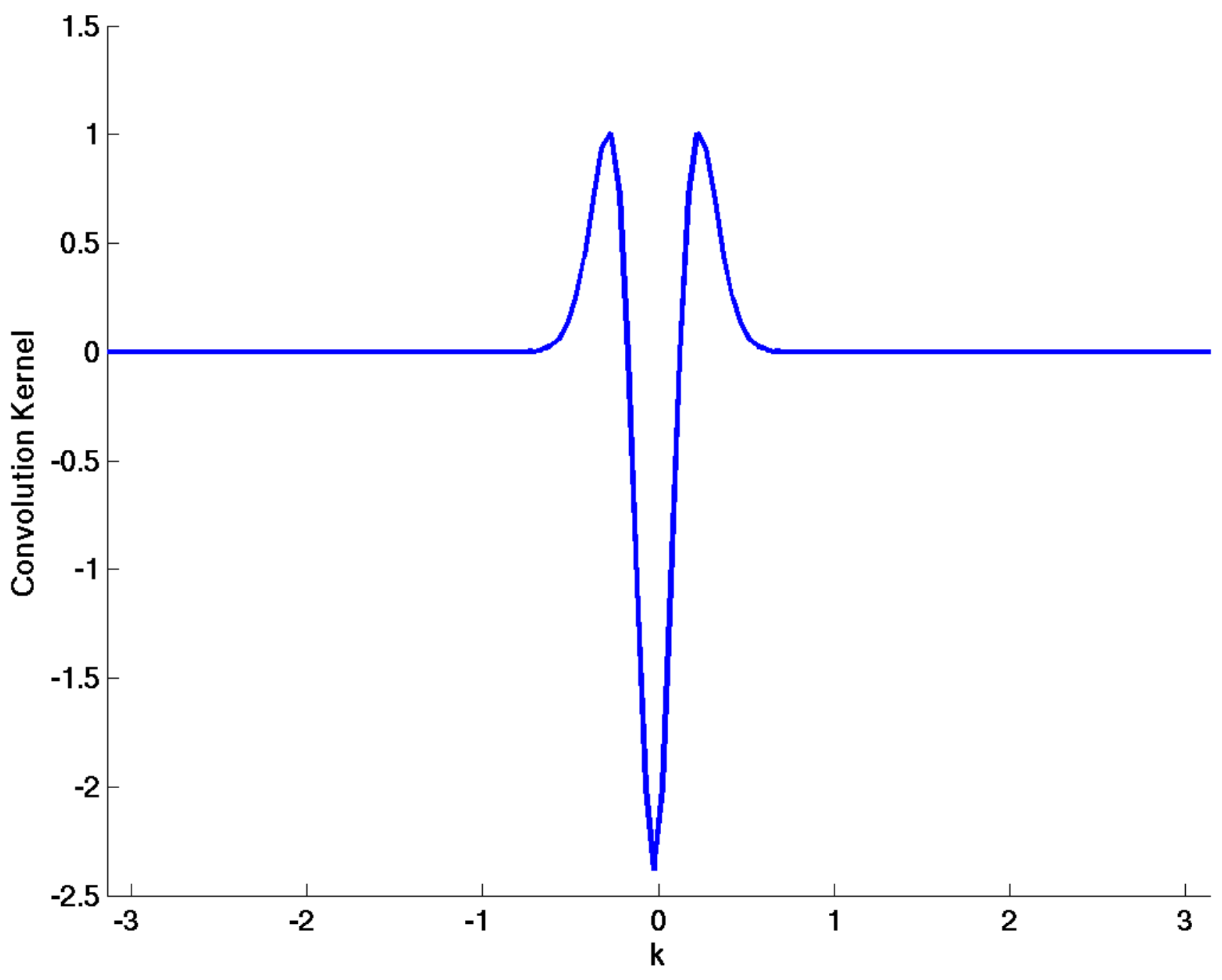}}
\put(175,5){\includegraphics[scale=0.39]{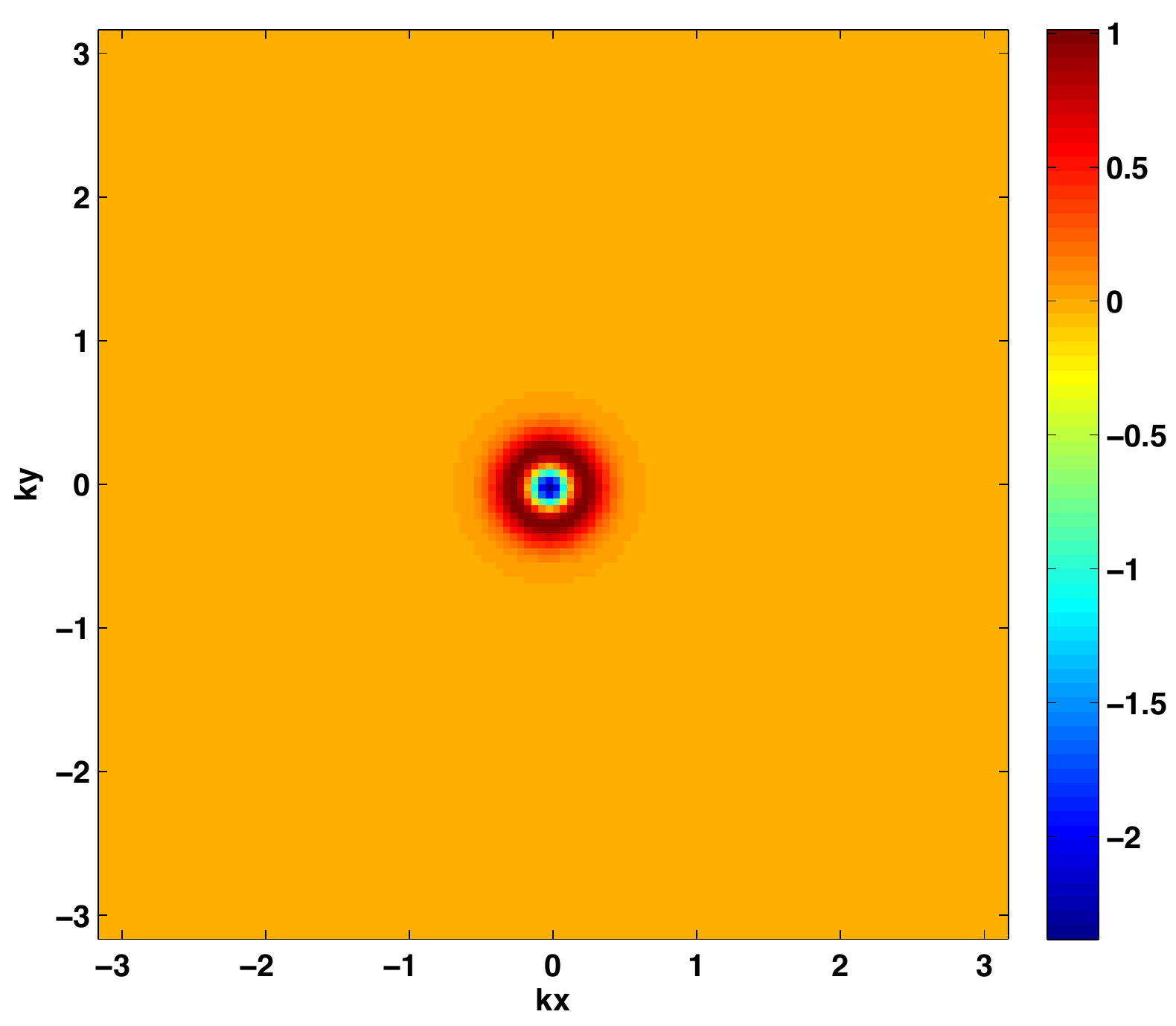}}
\end{picture}
\end{center}%\includegraphics{Figs/grow_15000.png}
\caption{The figure shows the discrete Fourier image of the convolution kernel $J$ used in the freezing simulation (see Section  \ref{freezing}). 
The Fourier image is radially symmetric and represented in 2-D (right) and in 1-D (left) along the radial direction.}
\label{fig_conv}

\end{figure}

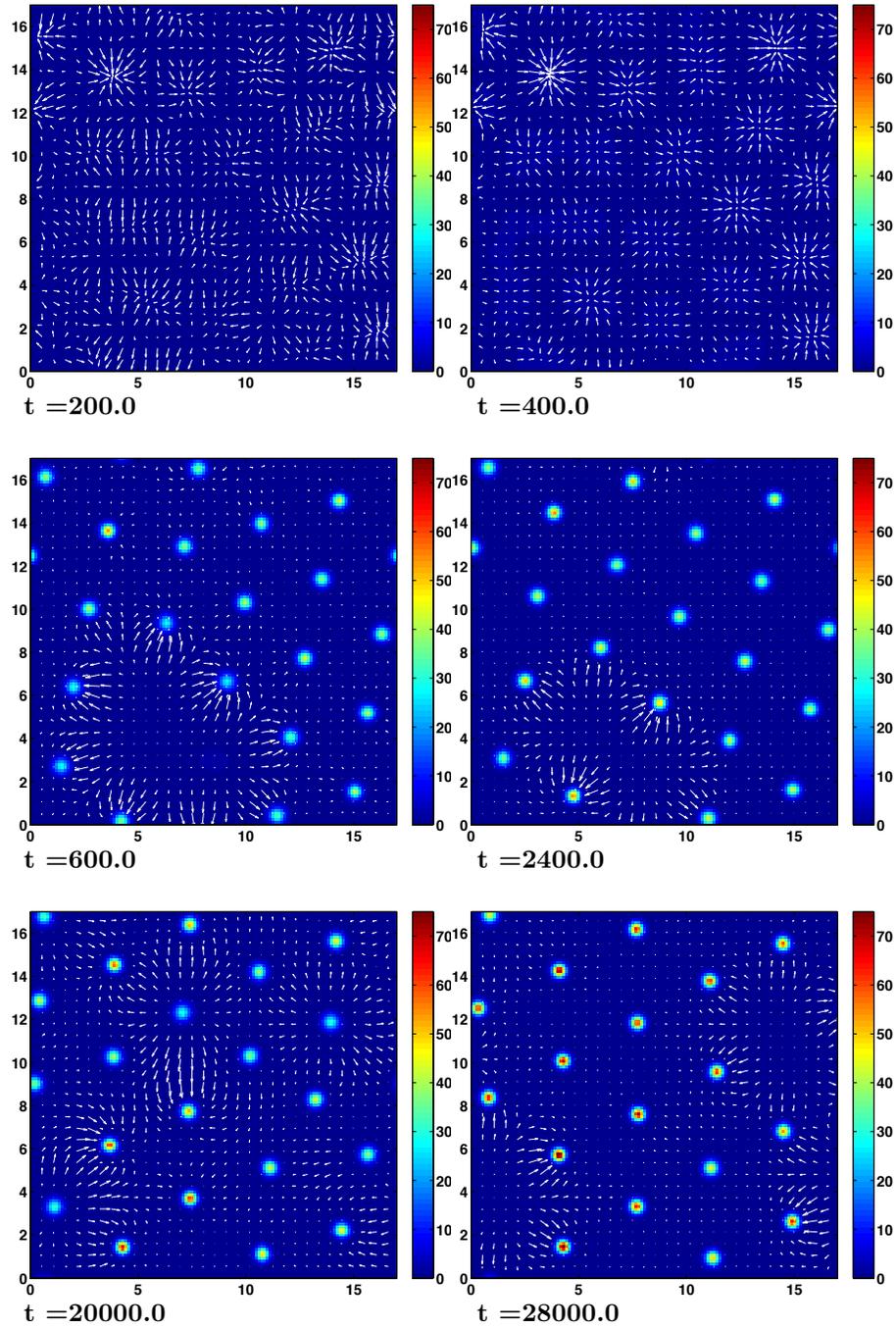
\begin{figure}[!ht]
\begin{center}
\begin{picture}(350,530)
\put(-0,0){\includegraphics[scale=0.4]{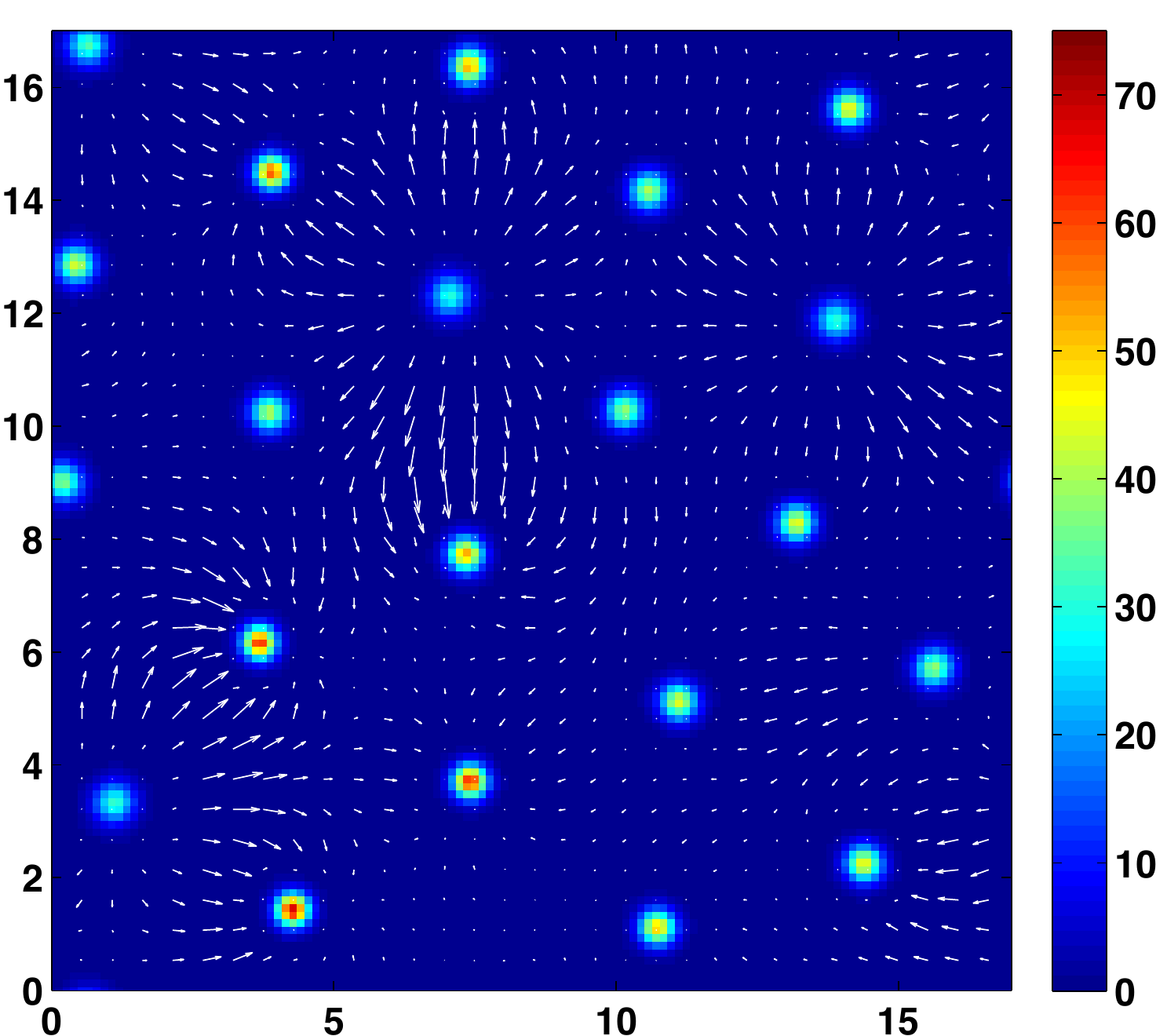}}
\put(170,0){\includegraphics[scale=0.4]{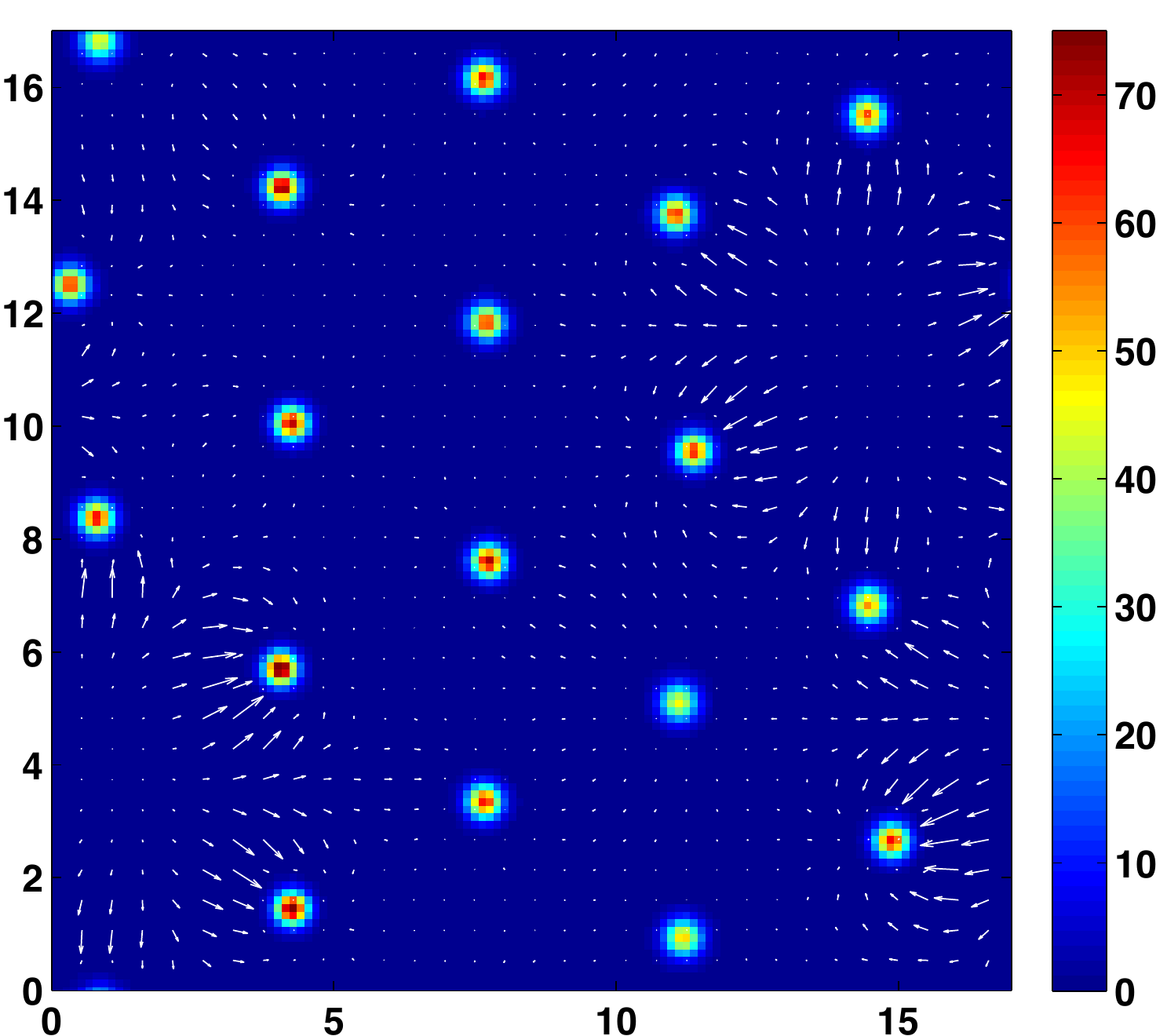}}
\put(-0,175){\includegraphics[scale=0.4]{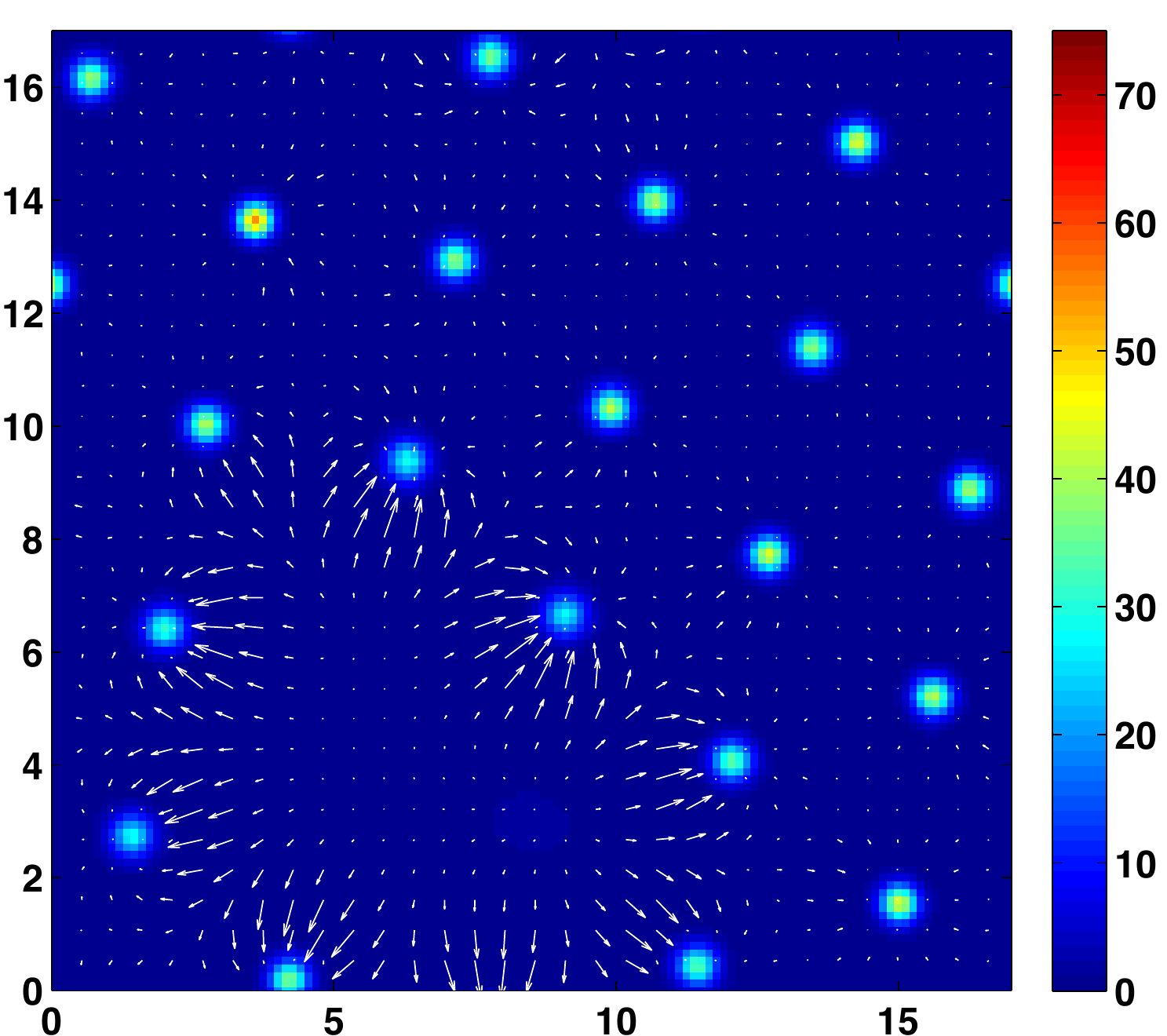}}
\put(170,175){\includegraphics[scale=0.4]{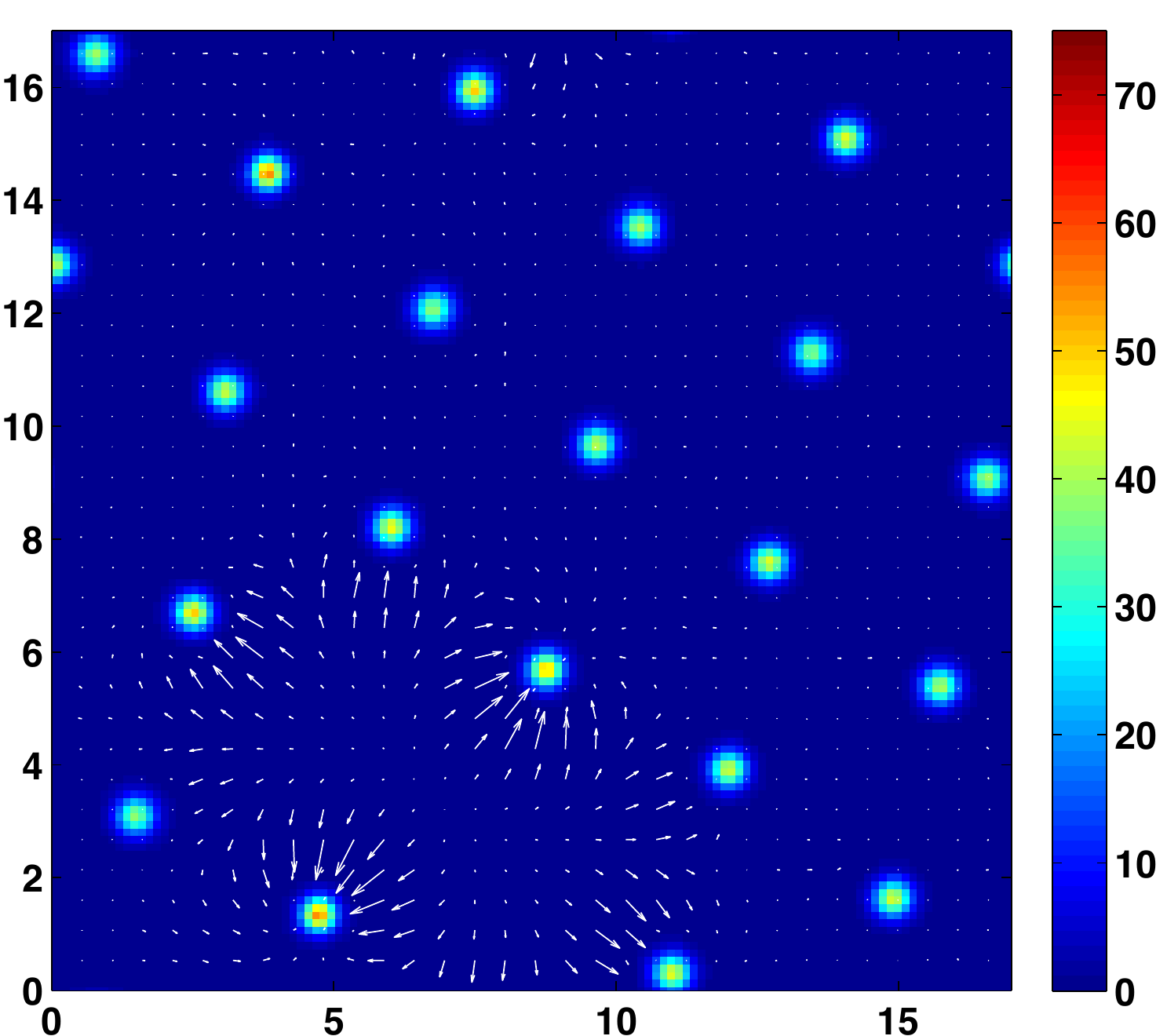}}
\put(-0,350){\includegraphics[scale=0.4]{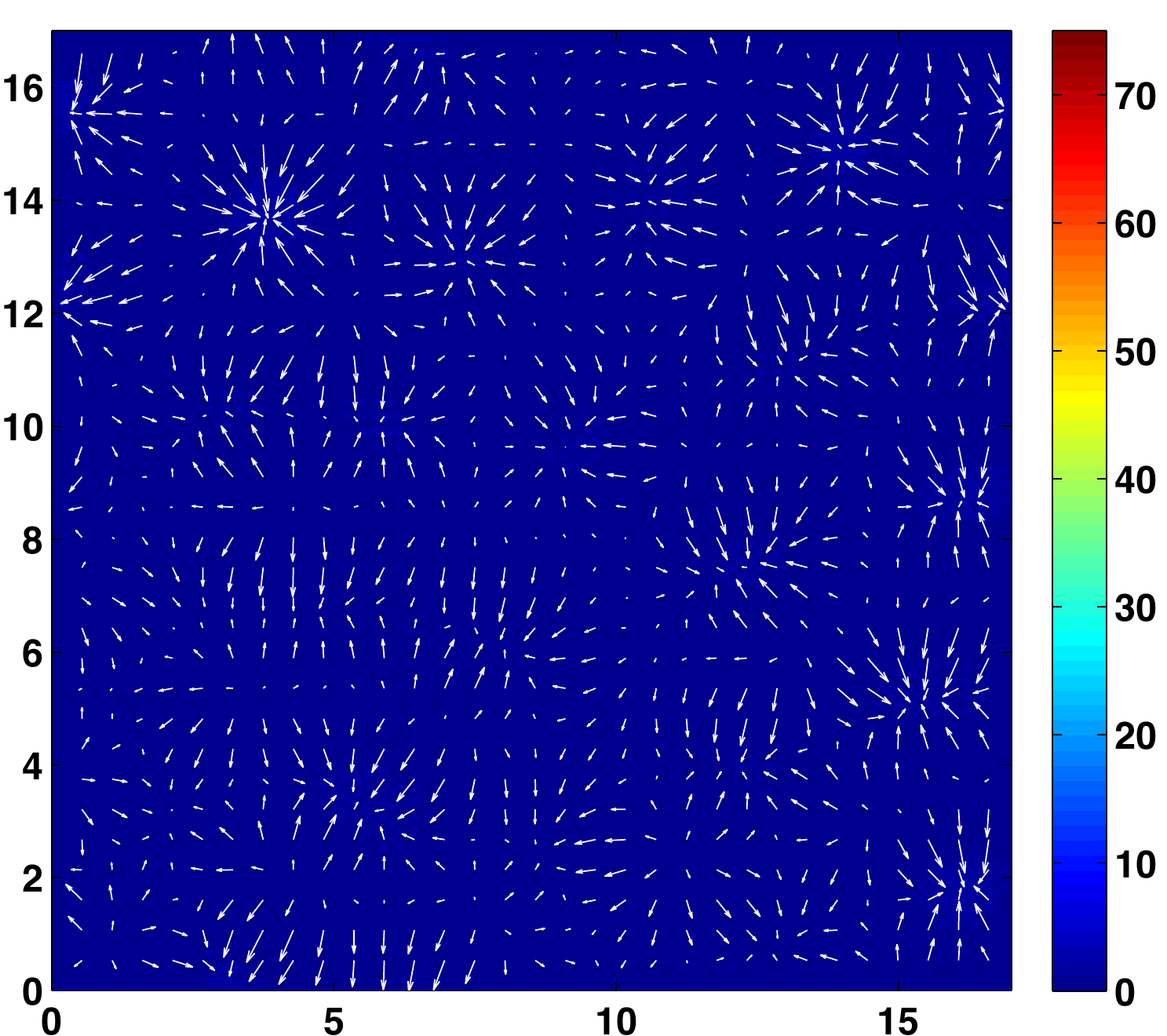}}
\put(170,350){\includegraphics[scale=0.4]{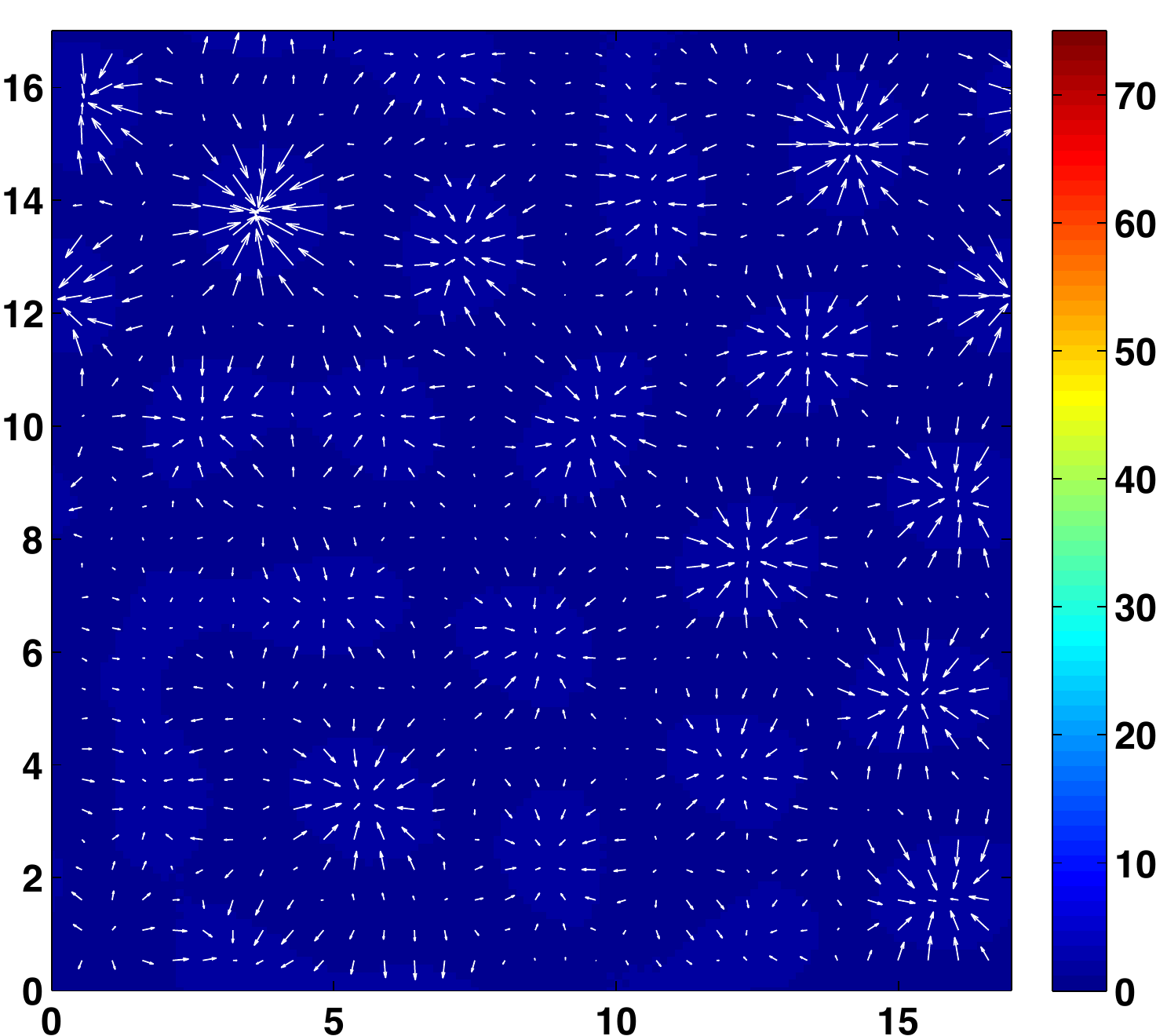}}
\put(5,-10){\text{\bf t =20000.0}}
\put(180,-10){\text{\bf t =28000.0}}
\put(5,165){\text{\bf t =600.0}}
\put(180,165){\text{\bf t =2400.0}}
\put(5,340){\text{\bf t =200.0}}
\put(180,340){\text{\bf t =400.0}}
\end{picture}
\end{center}%\includegraphics{Figs/grow_15000.png}
\caption{The time evolution of the density field with the velocity field superimposed. The figure shows the liquid to solid phase transition
in the hydrodynamic CDFT model discussed in Section \ref{freezing}. }
\label{fig_DFT_t2b}
\end{figure}

\begin{figure}[!ht]
\begin{center}
\begin{picture}(270,530)
\put(0,-5){\includegraphics[scale=0.48]{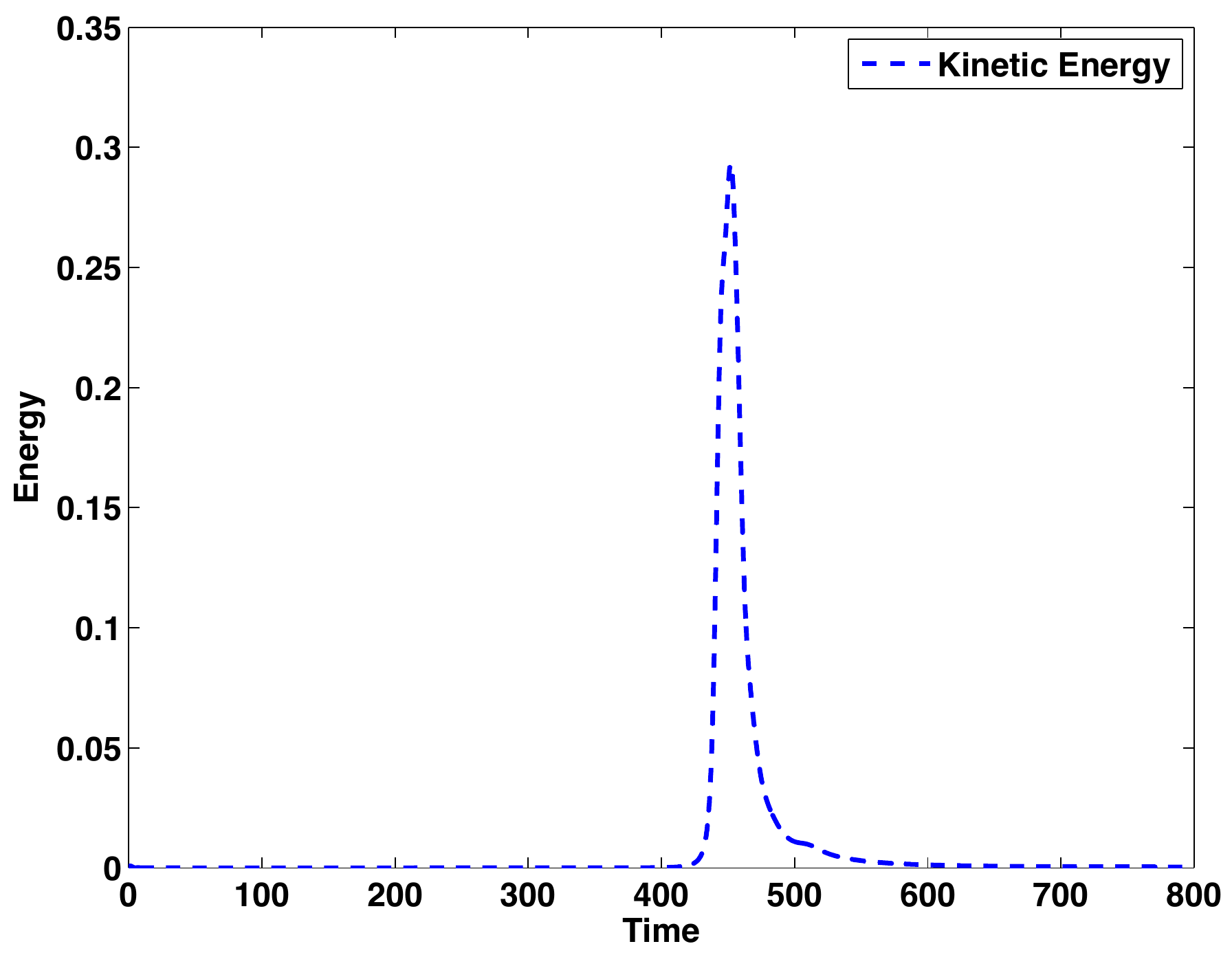}}
\put(0, 195){\includegraphics[scale=0.48]{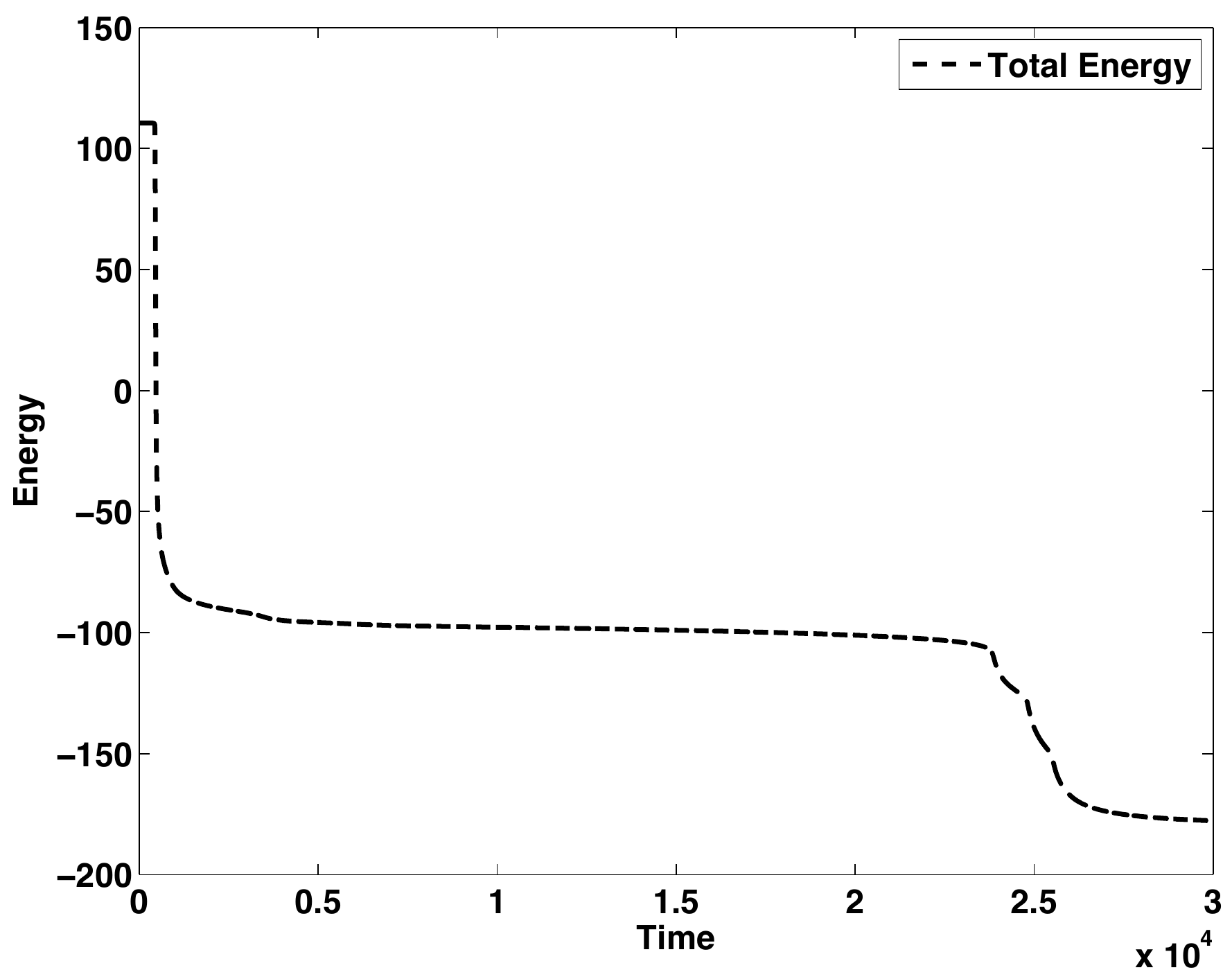}}
\put(0, 395){\includegraphics[scale=0.48]{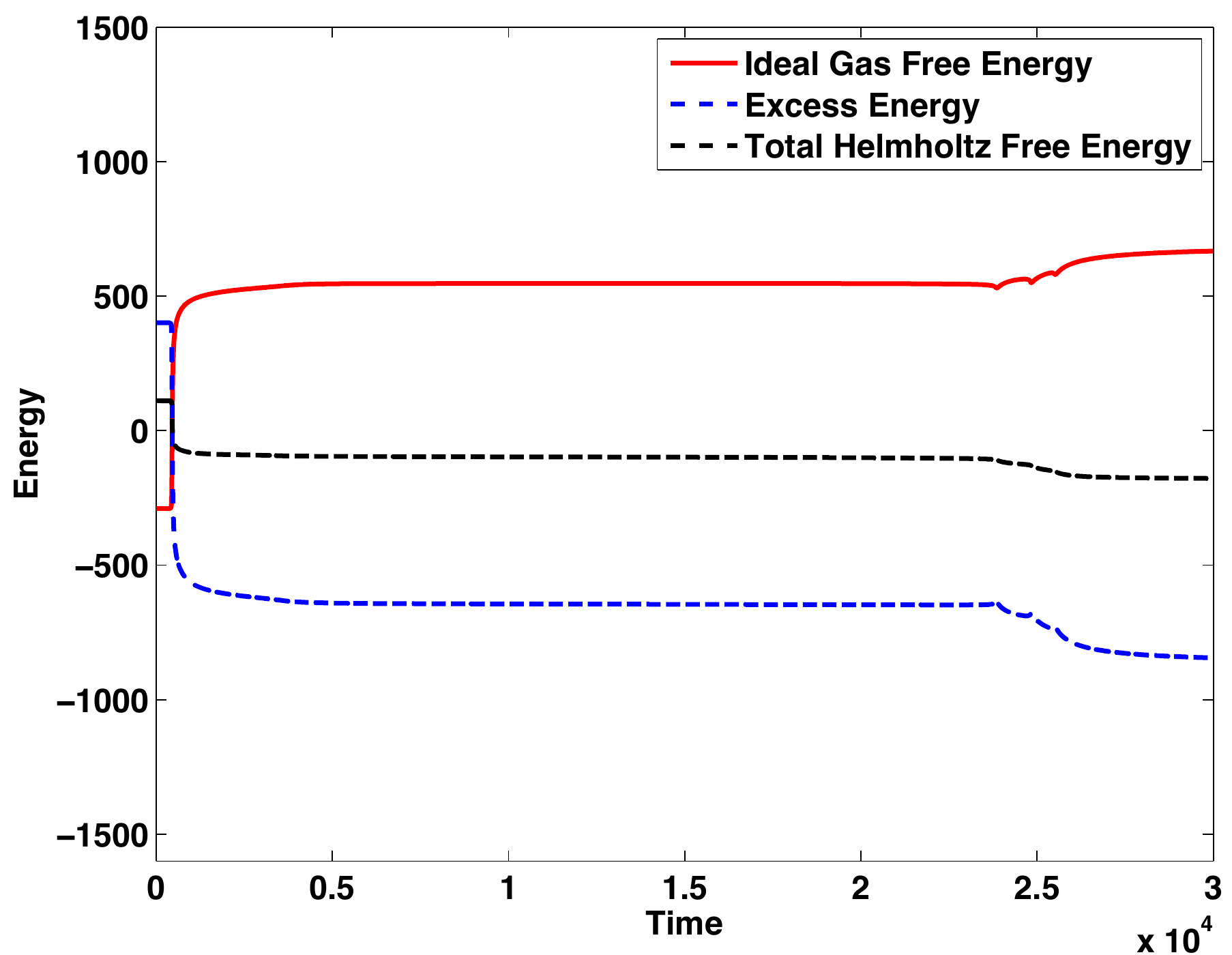}}
\end{picture}
\end{center}%\includegraphics{Figs/grow_15000.png}
\caption{The evolution of kinetic energy, Helmholtz free energy $\mathcal{F}[\rho]$ and total energy $\mathcal{E}[\rho, \bu]$, as labeled,   for the simulation of freezing of 
the hydrodynamic  CDFT model described in Section  \ref{freezing} and shown in Figure \ref{fig_DFT_t2b}. }
\label{fig_DFT_t2a}
\end{figure}

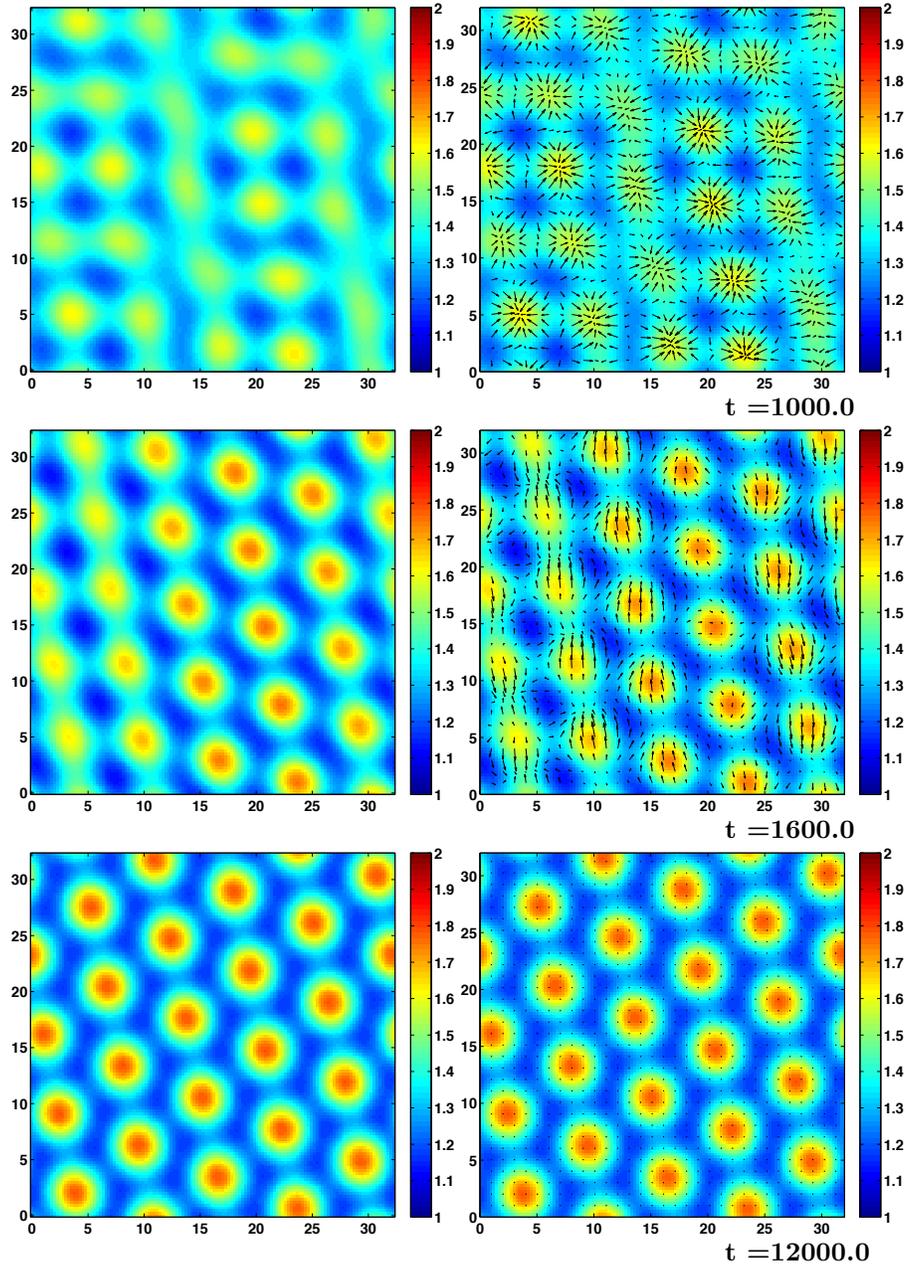
\begin{figure}[!ht]
\begin{center}
\begin{picture}(350,480)
\put(0,0){\includegraphics[scale=0.39]{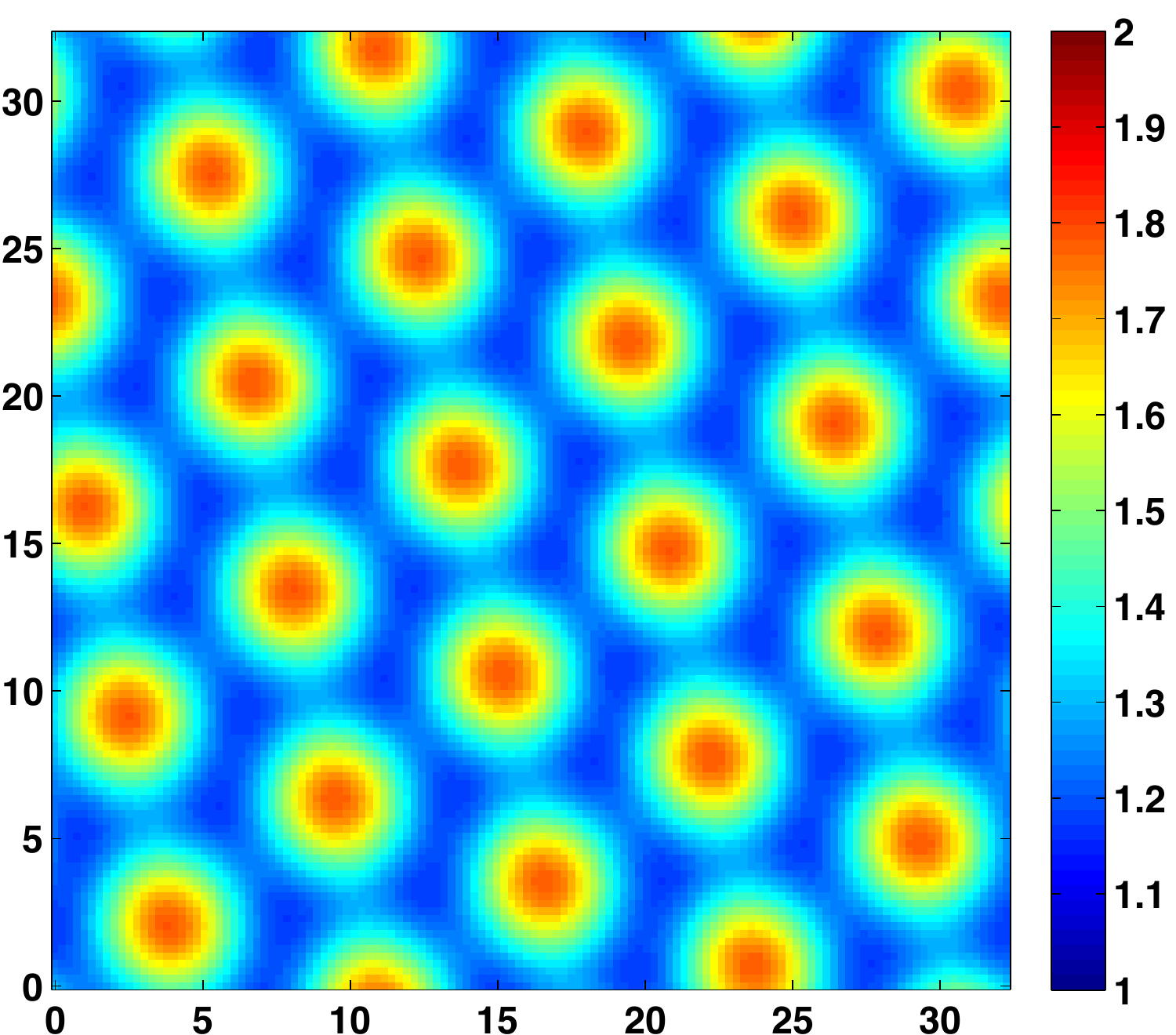}}
\put(0,160){\includegraphics[scale=0.39]{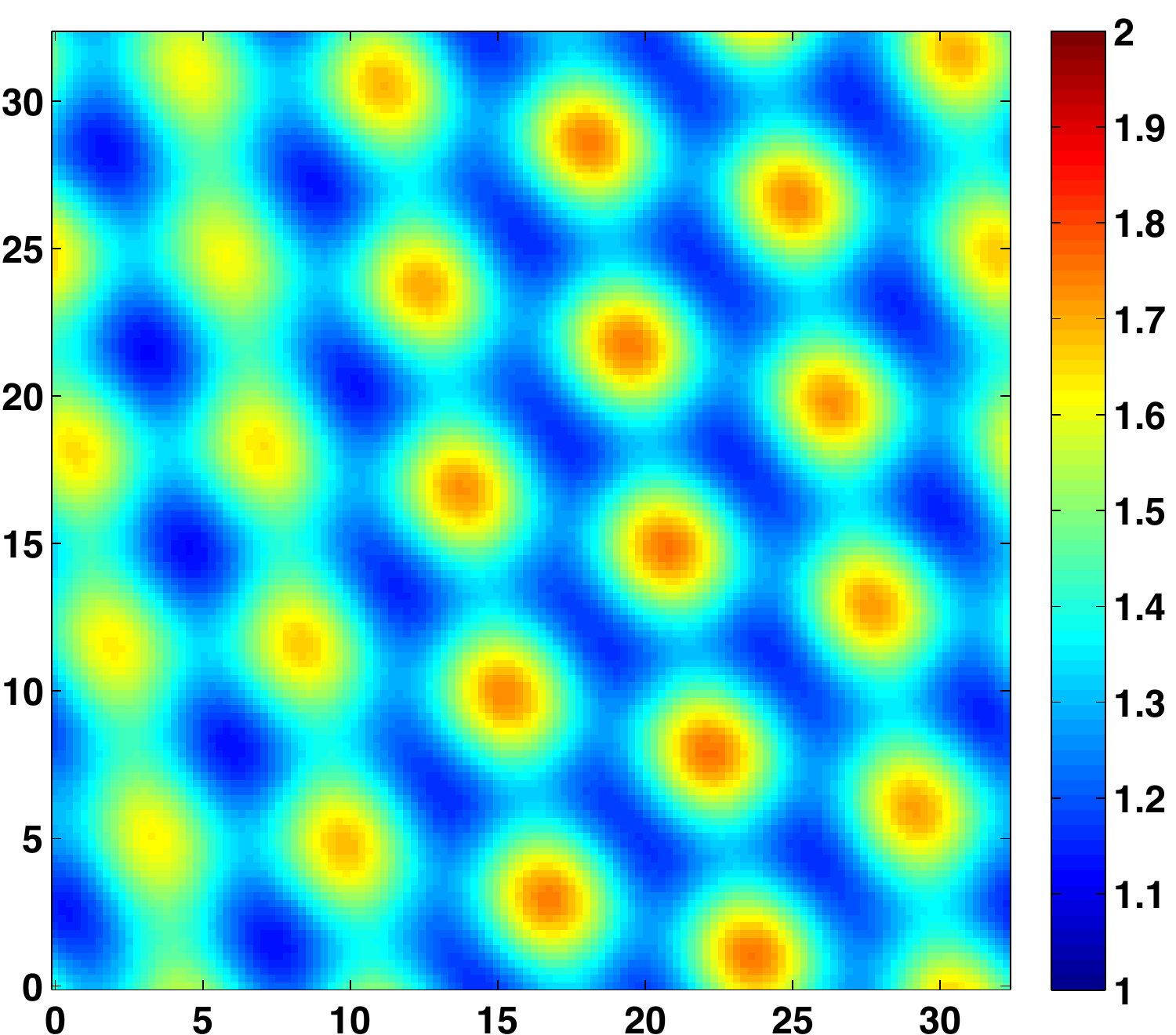}}
\put(0,320){\includegraphics[scale=0.39]{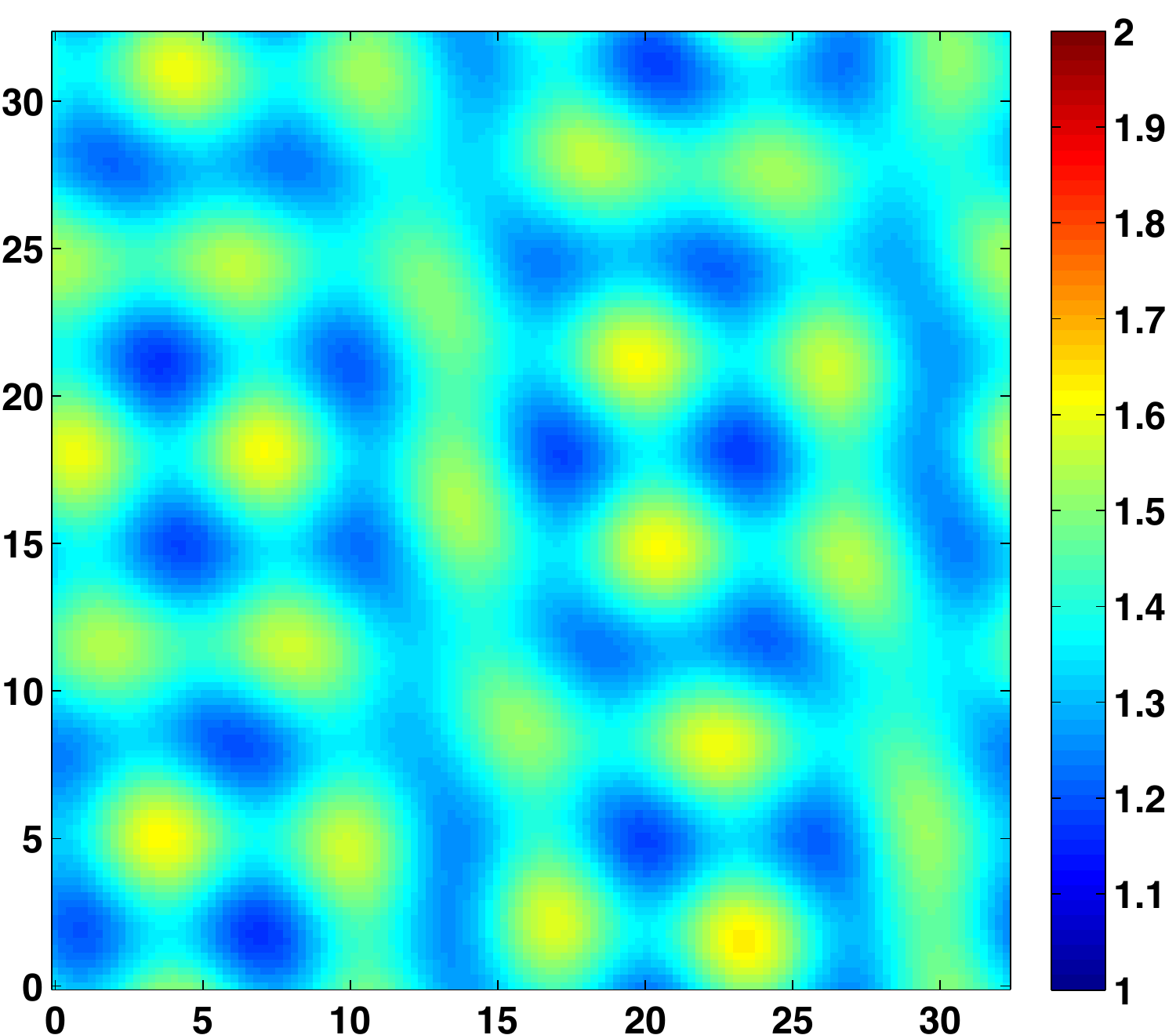}}
\put(170,0){\includegraphics[scale=0.39]{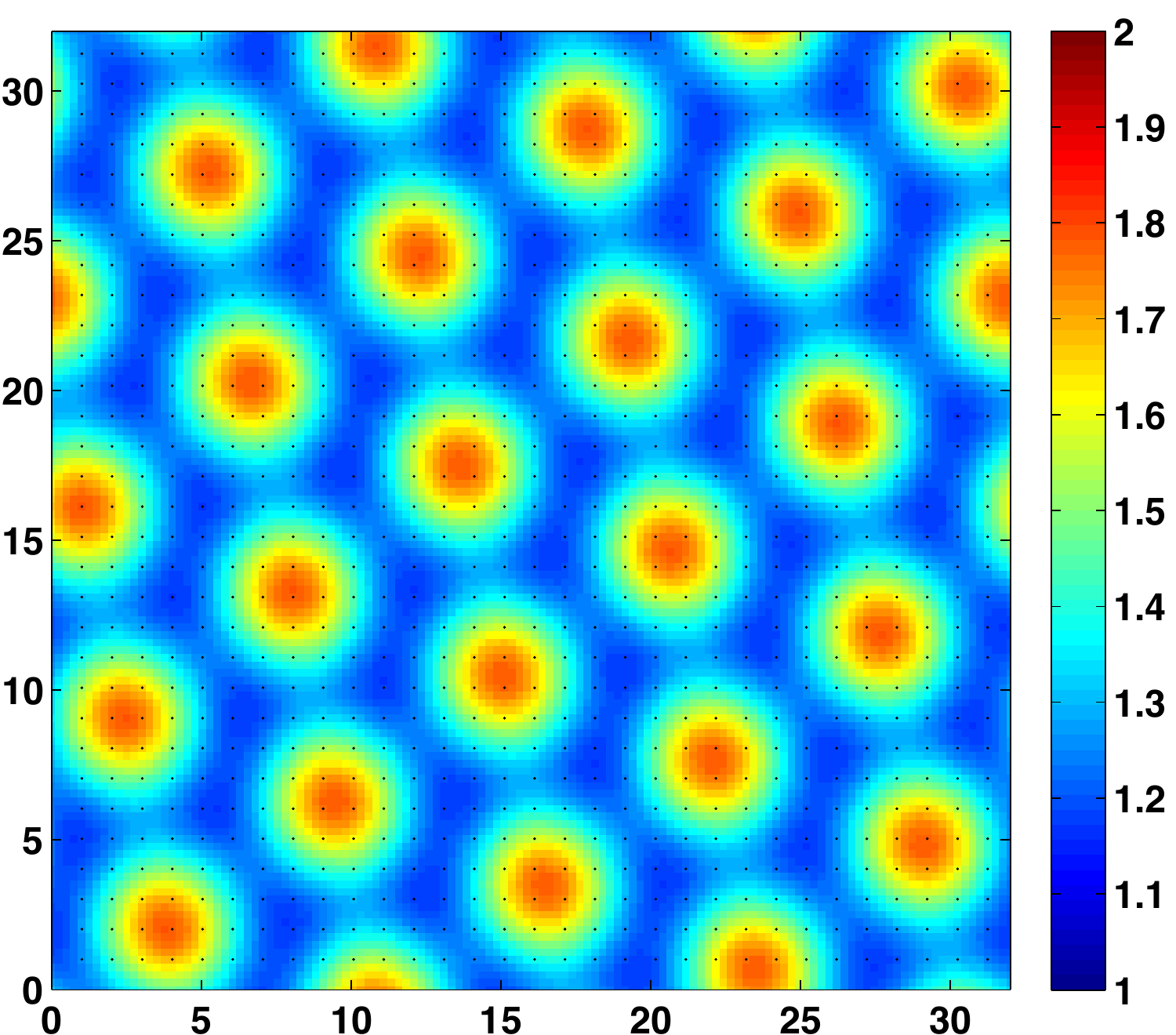}}
\put(170,160){\includegraphics[scale=0.39]{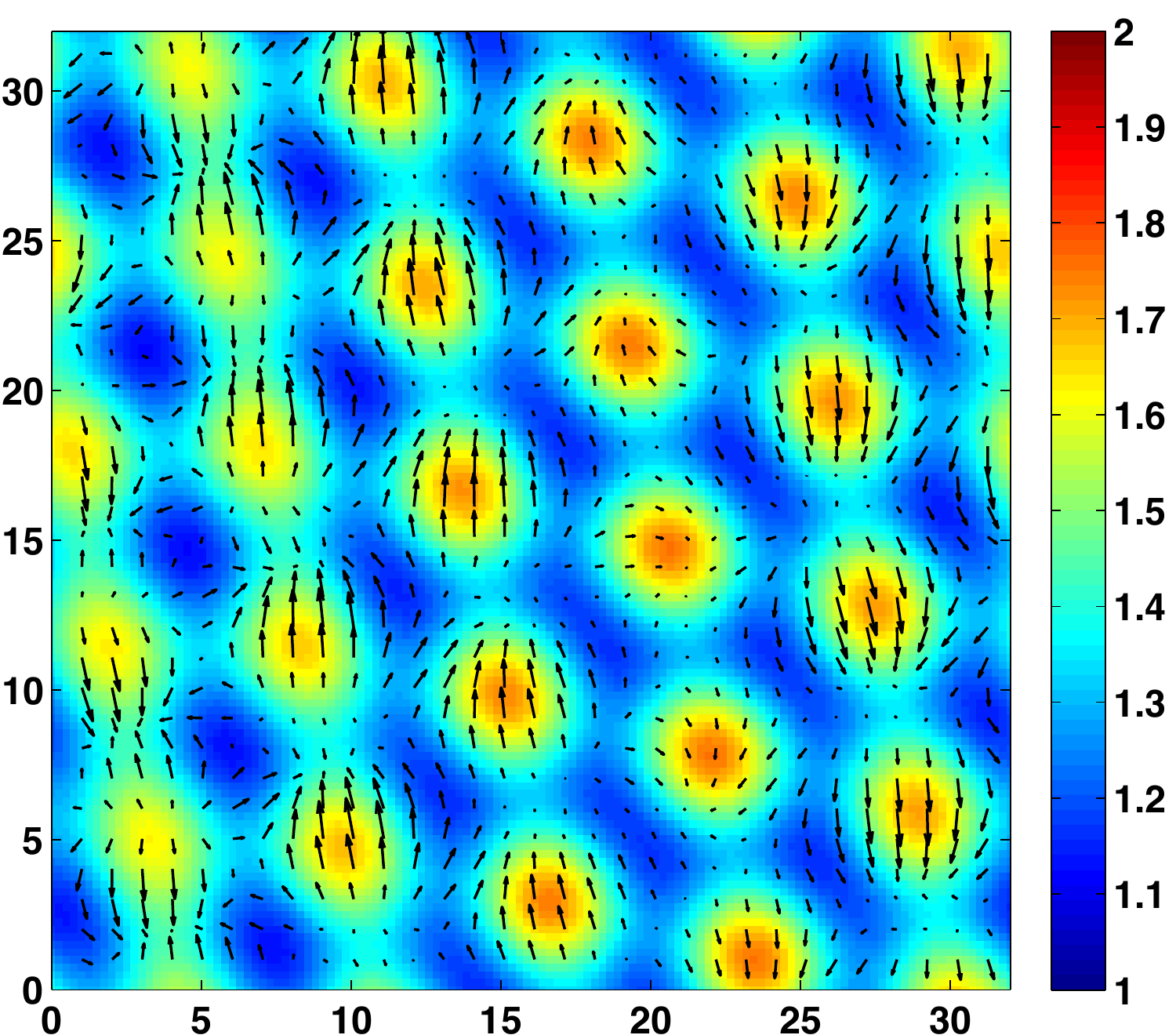}}
\put(170,320){\includegraphics[scale=0.39]{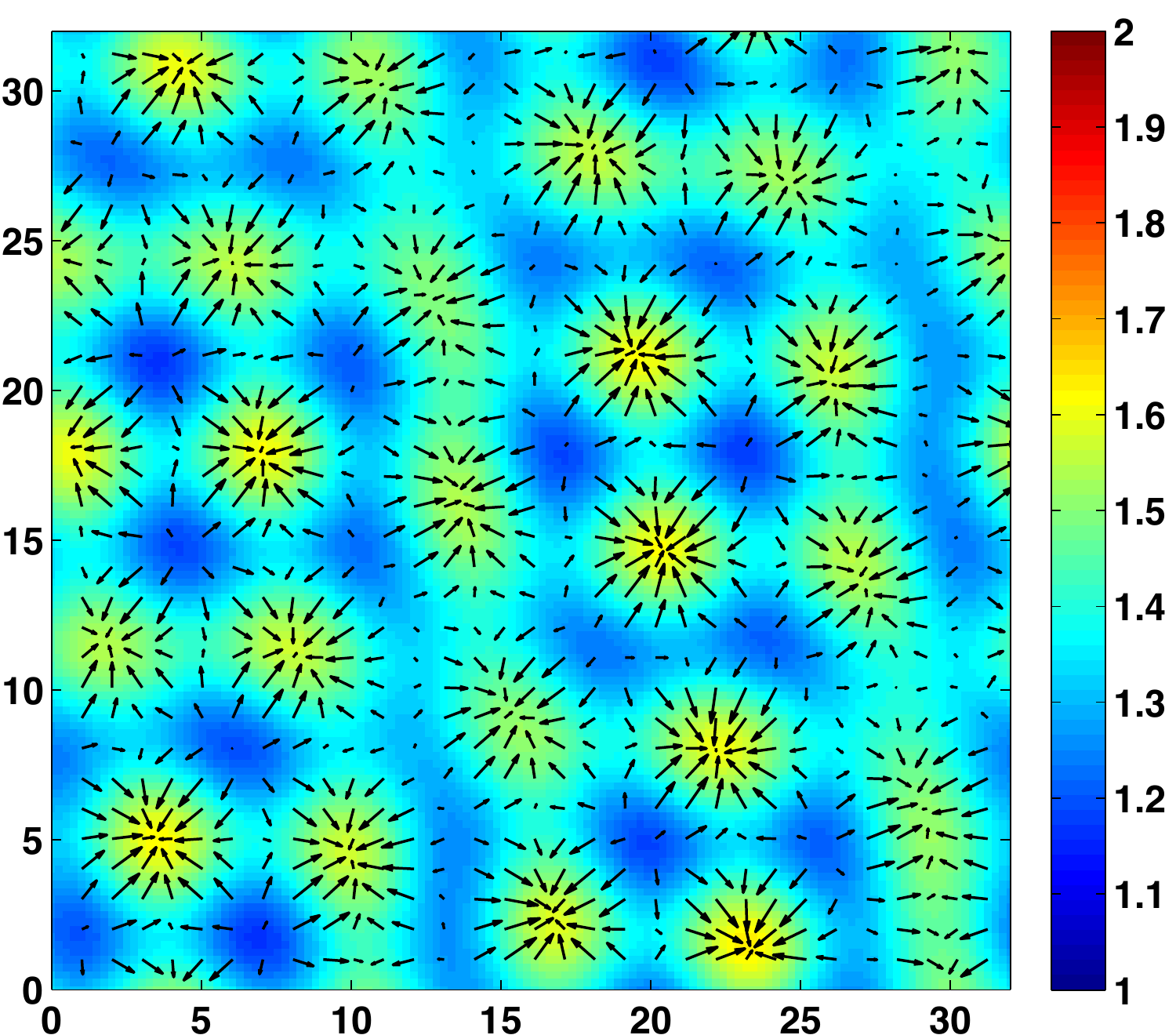}}
\put(270,-10){\text{\bf t =12000.0}}
\put(270,150){\text{\bf t =1600.0}}
\put(270,310){\text{\bf t =1000.0}}
\end{picture}
\end{center}%\includegraphics{Figs/grow_15000.png}
\caption{The time evolution of the density field. The figure shows the liquid to solid phase transition
in the hydrodynamic PFC model discussed in Section  \ref{freezing}. }
\label{fig_PFC_t2b}
\end{figure}

\begin{figure}[!ht]
\begin{center}
\begin{picture}(270,500)
\put(0,-5){\includegraphics[scale=0.6]{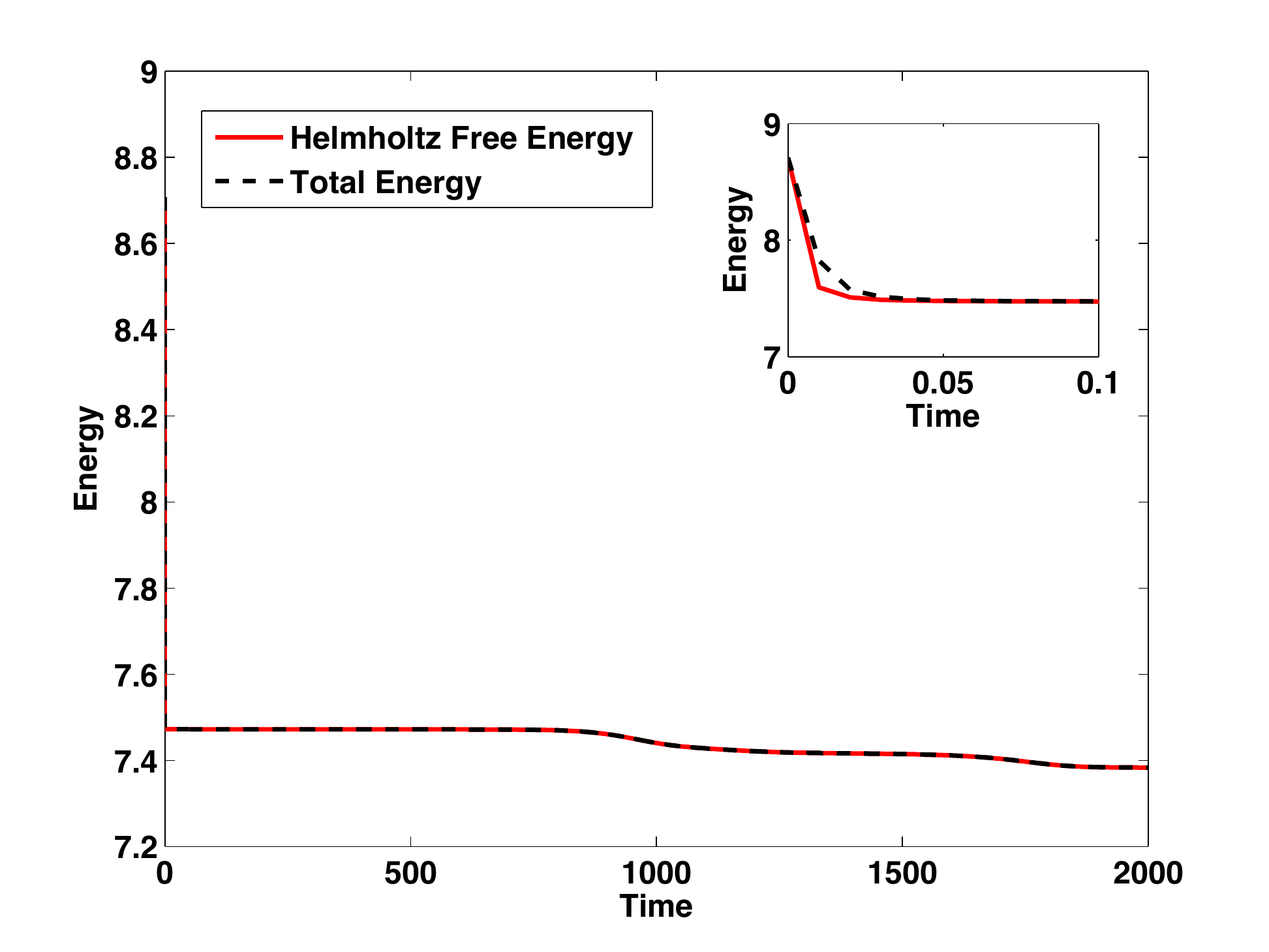}}
\put(0, 240){\includegraphics[scale=0.6]{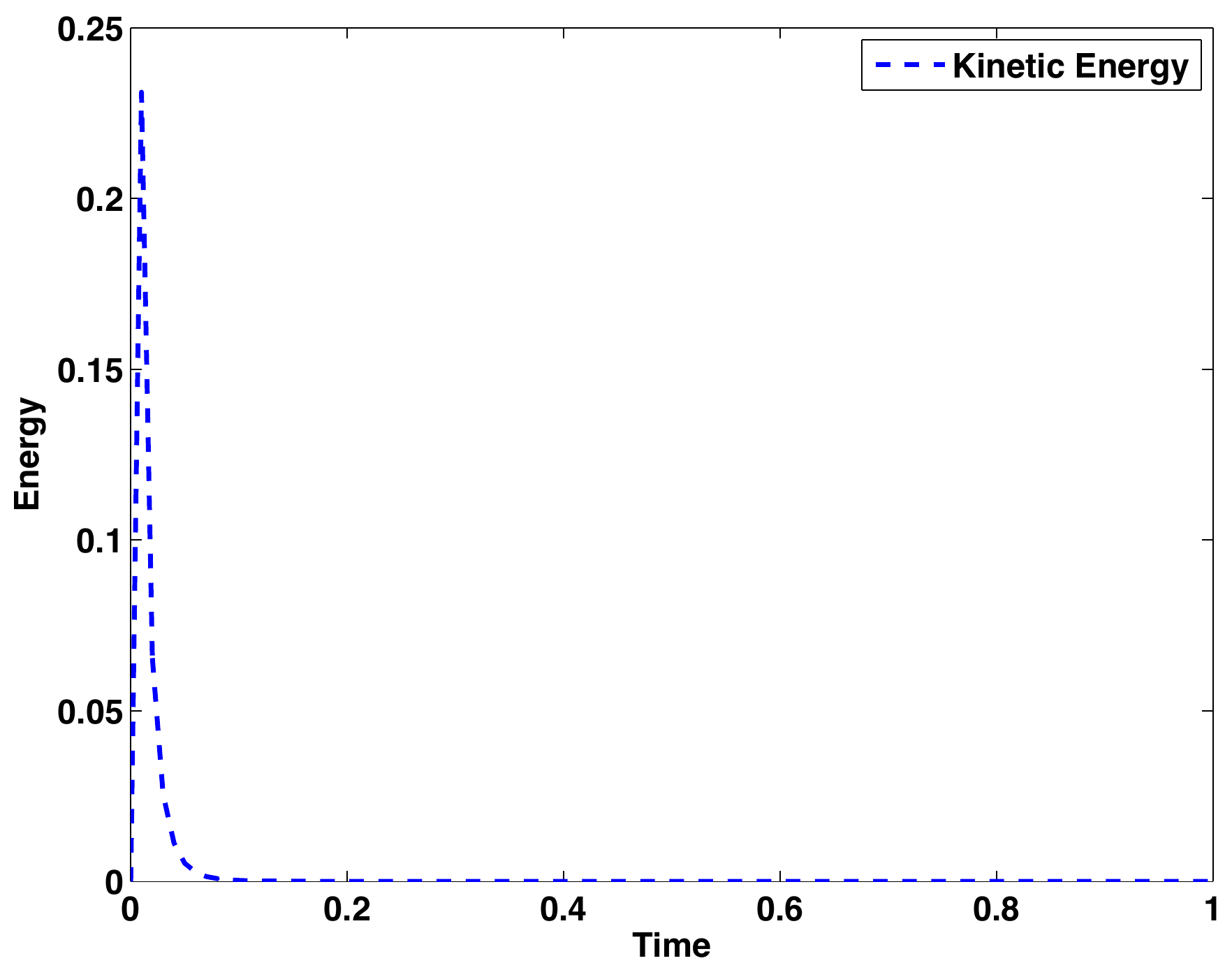}}
\end{picture}
\end{center}%\includegraphics{Figs/grow_15000.png}
\caption{The evolution of kinetic energy, Helmholtz free energy $\mathcal{F}[\rho]$ and total energy $\mathcal{E}[\rho, \bu]$, as labeled,  for the 
simulation of freezing of the hydrodynamic PFC model described in Section  \ref{freezing} and shown in Figure \ref{fig_PFC_t2b}. }
\label{fig_PFC_t2a}
\end{figure}

\begin{figure}[!ht]
\begin{center}
\begin{picture}(350,140)
\put(-30,5){\includegraphics[scale=0.39]{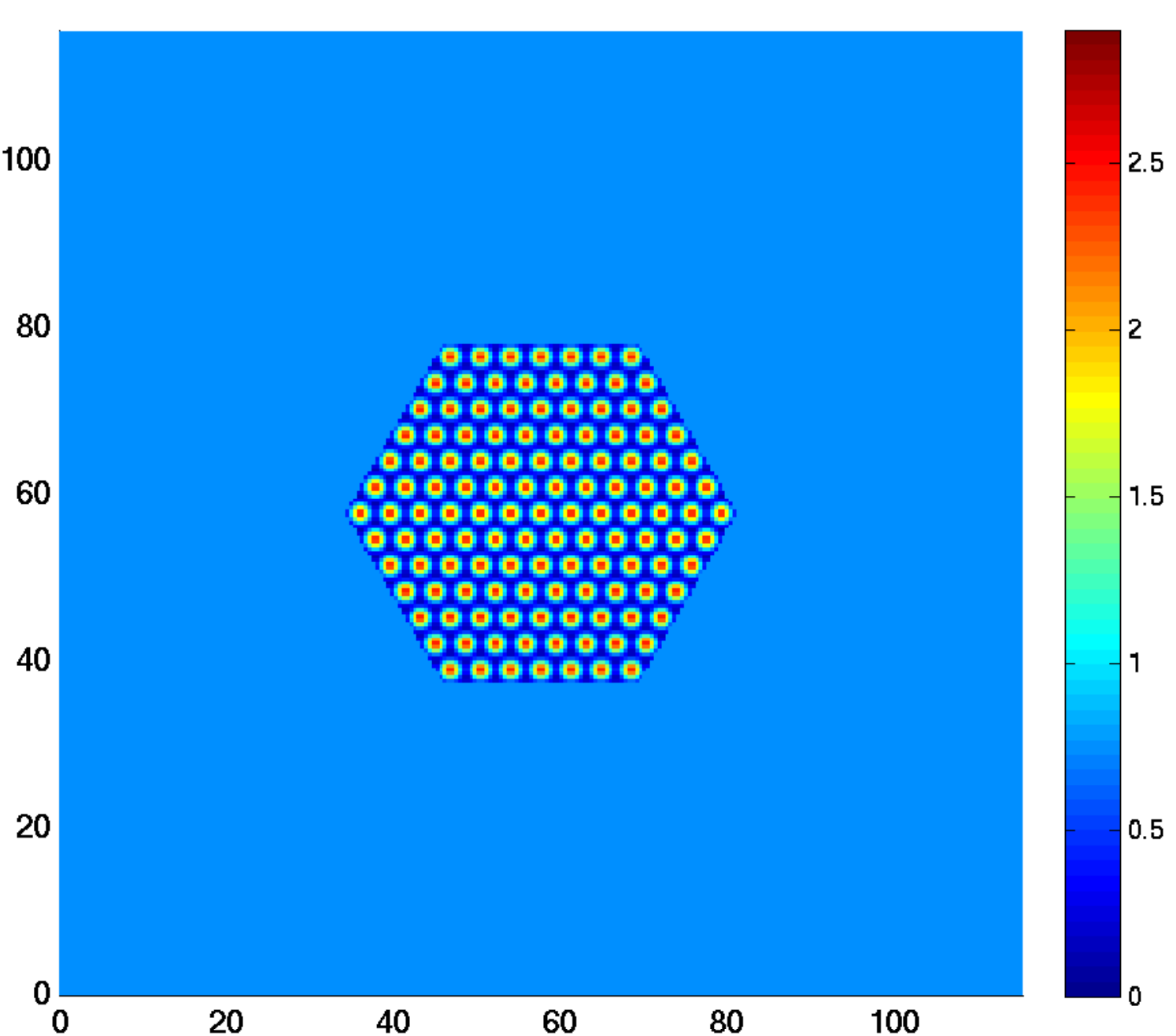}}
\put(175,5){\includegraphics[scale=0.39]{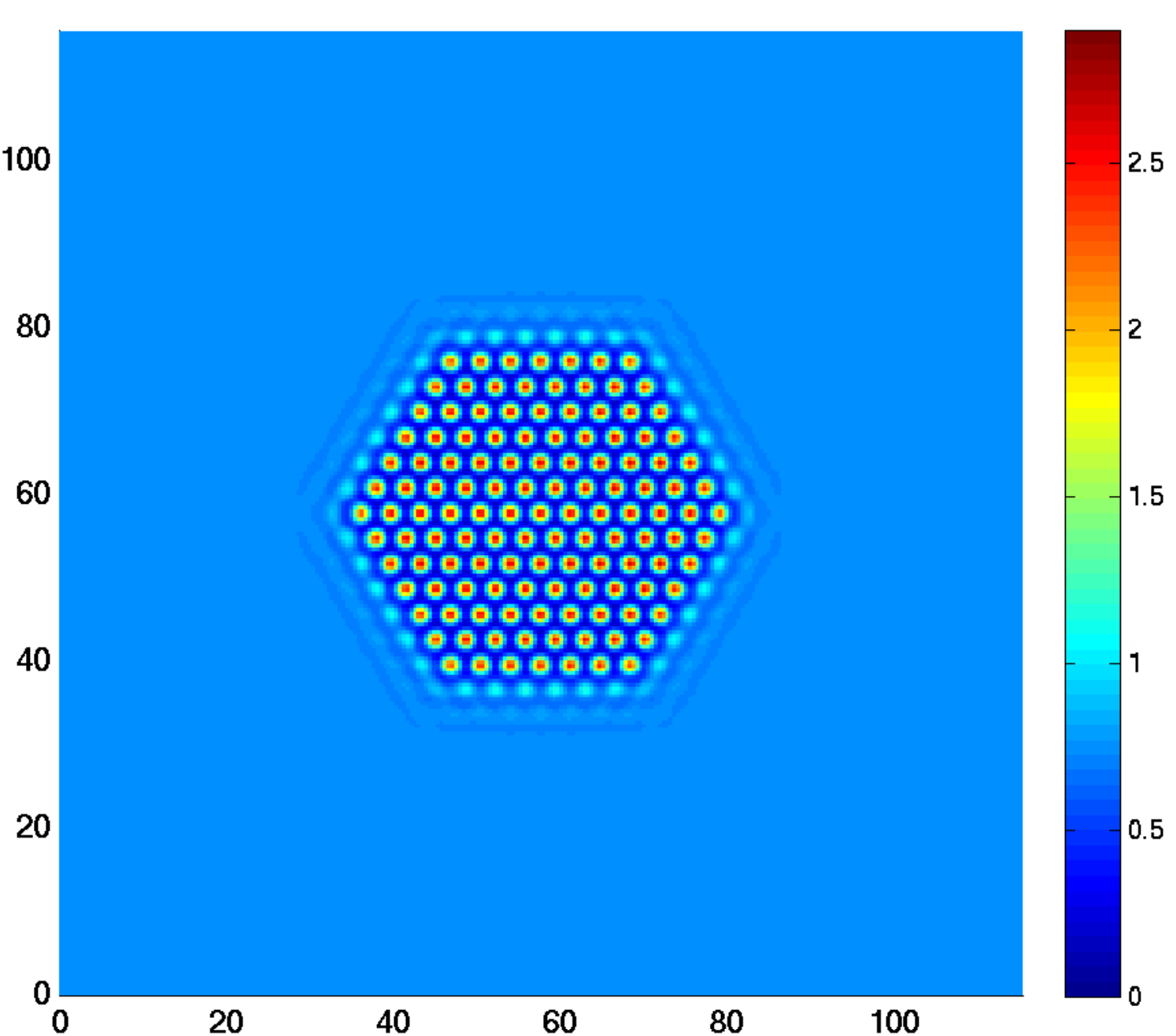}}
\put(0,-5){\text{\bf Initial Density Field (t=0)}}
\put(175,-5){\text{\bf Final Density Field (t=20000)}}
\end{picture}
\end{center}%\includegraphics{Figs/grow_15000.png}
\caption{The initial and final configuration during the annealing process of a nano-crystal placed in a channel with density of solid and liquid phases 
chosen in the co-existence regime at equilibrium (see Section  \ref{flow_PFC}). }
\label{fig_nucleate}

\end{figure}

\begin{figure}[!ht]
\begin{center}
\begin{picture}(350,490)
\put(0,390){\includegraphics[scale=0.34]{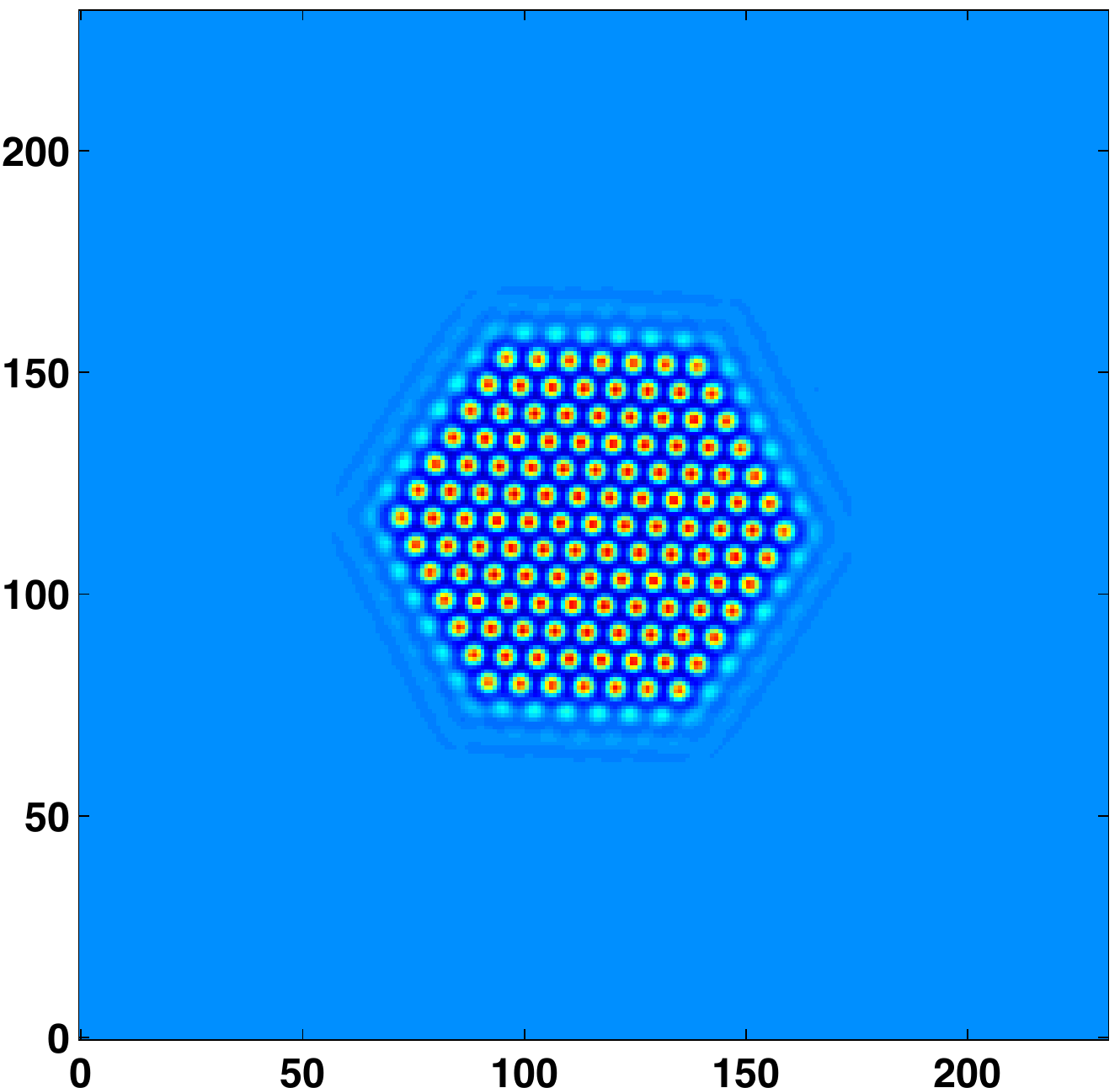}}
\put(130,390){\includegraphics[scale=0.34]{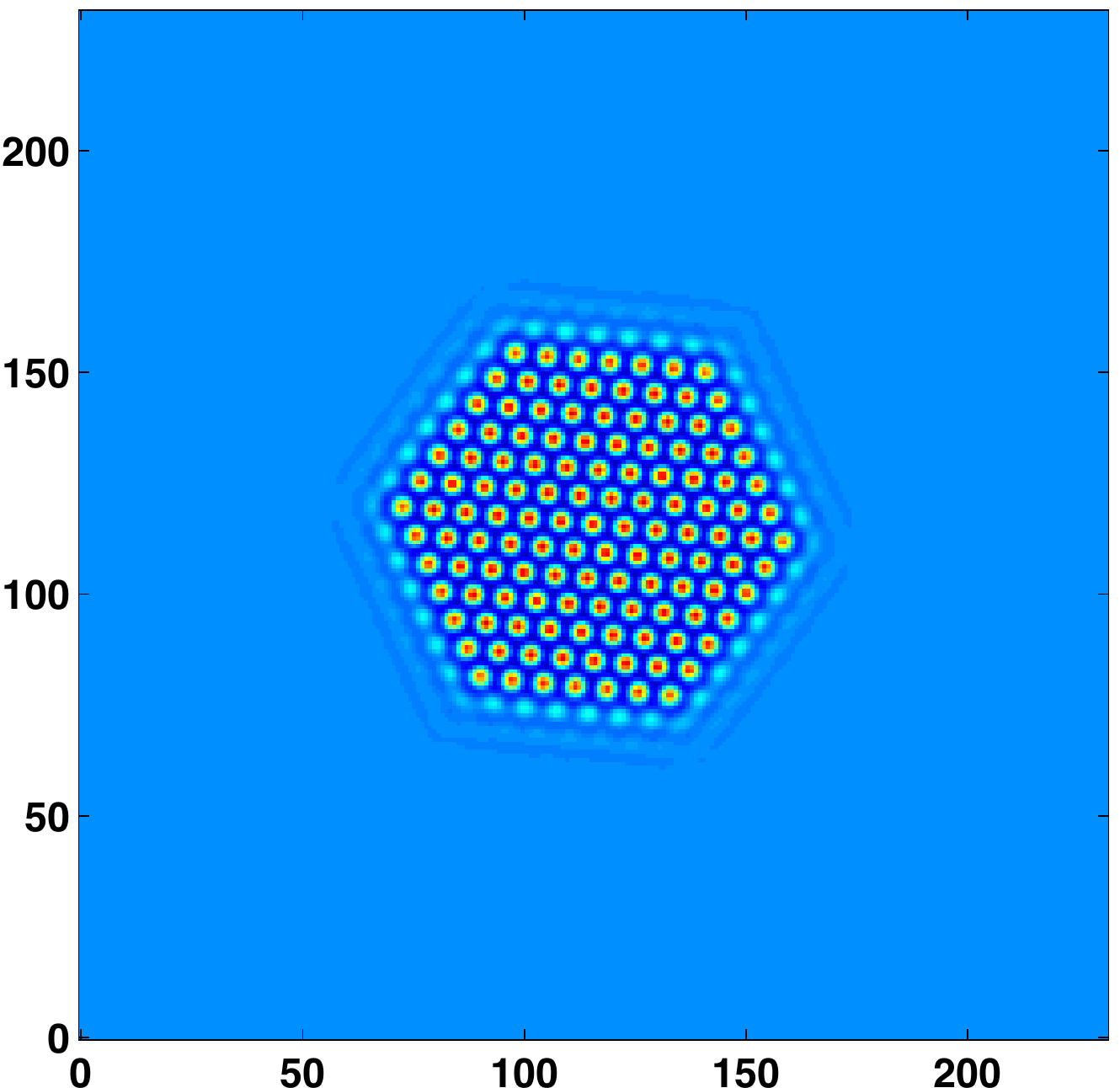}}
\put(260,390){\includegraphics[scale=0.34]{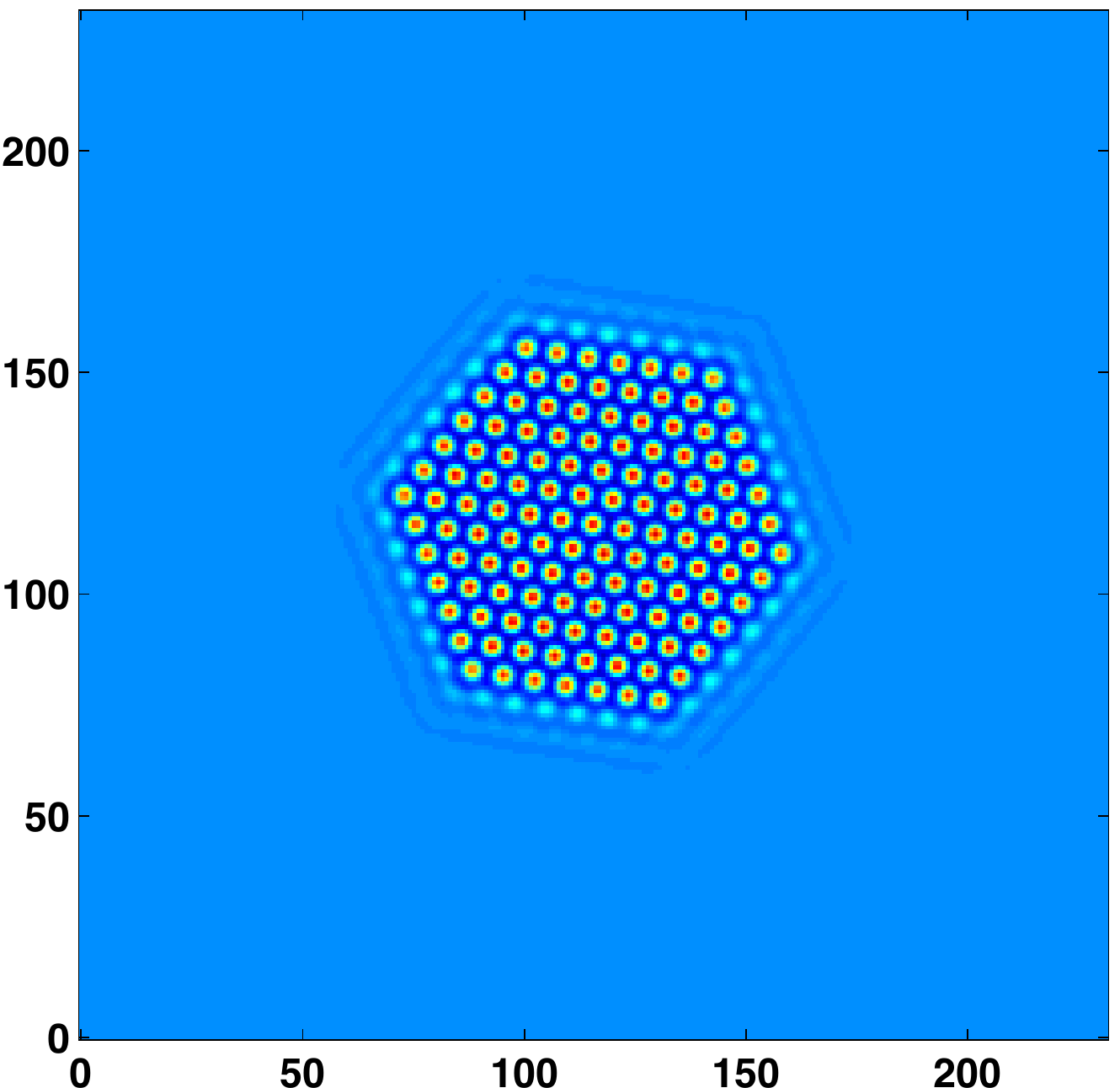}}
\put(0,260){\includegraphics[scale=0.34]{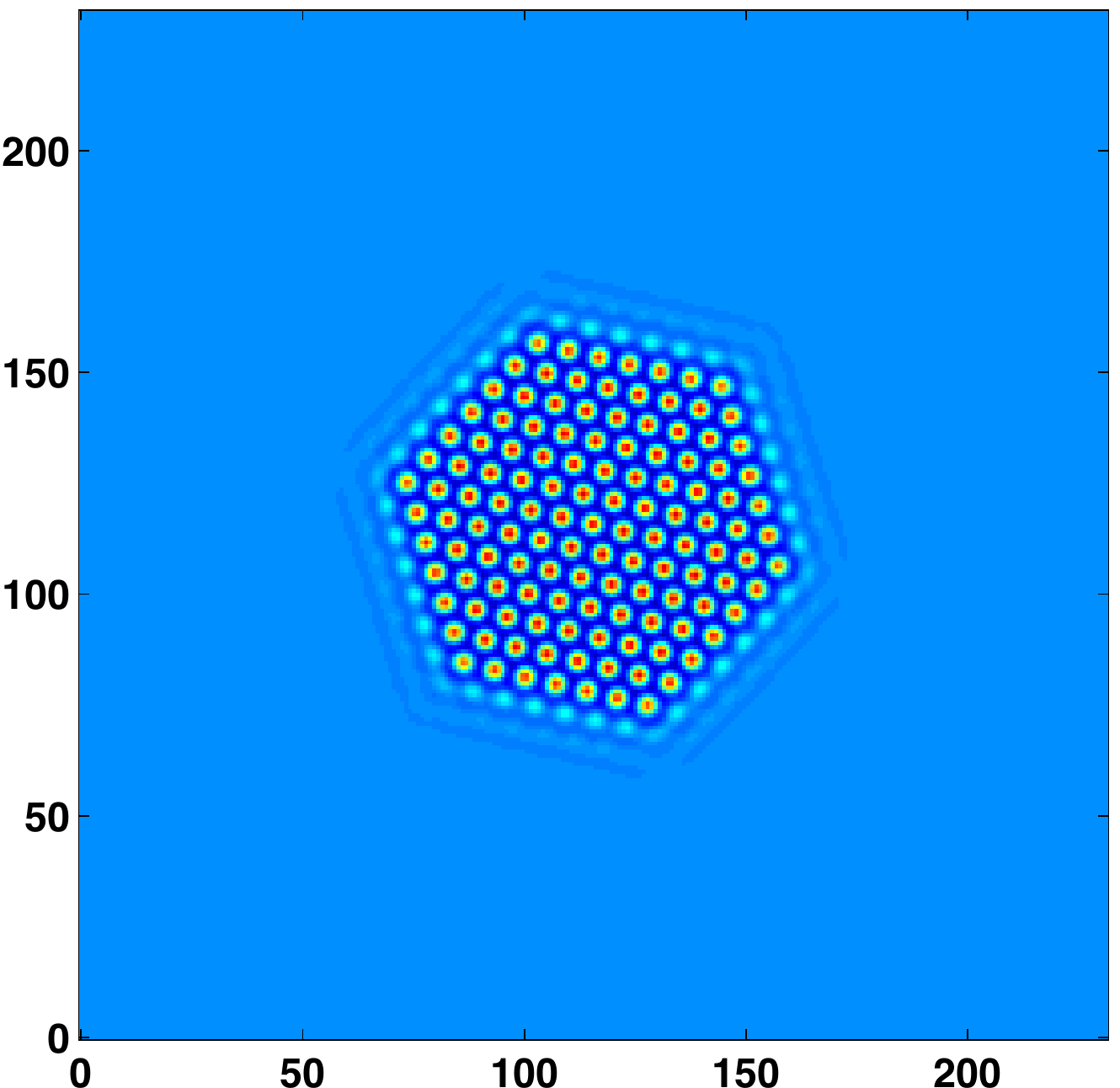}}
\put(130,260){\includegraphics[scale=0.34]{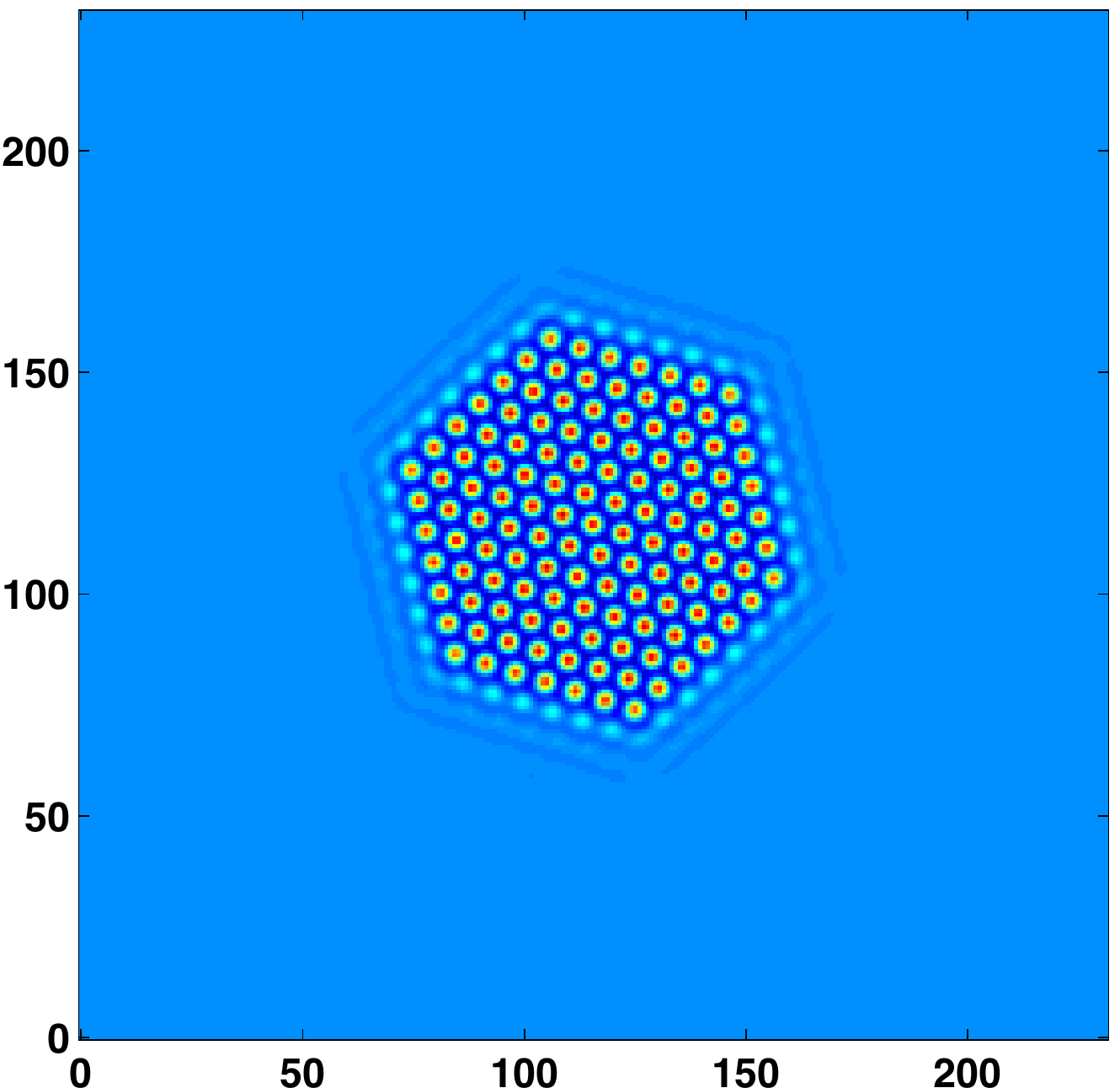}}
\put(260,260){\includegraphics[scale=0.34]{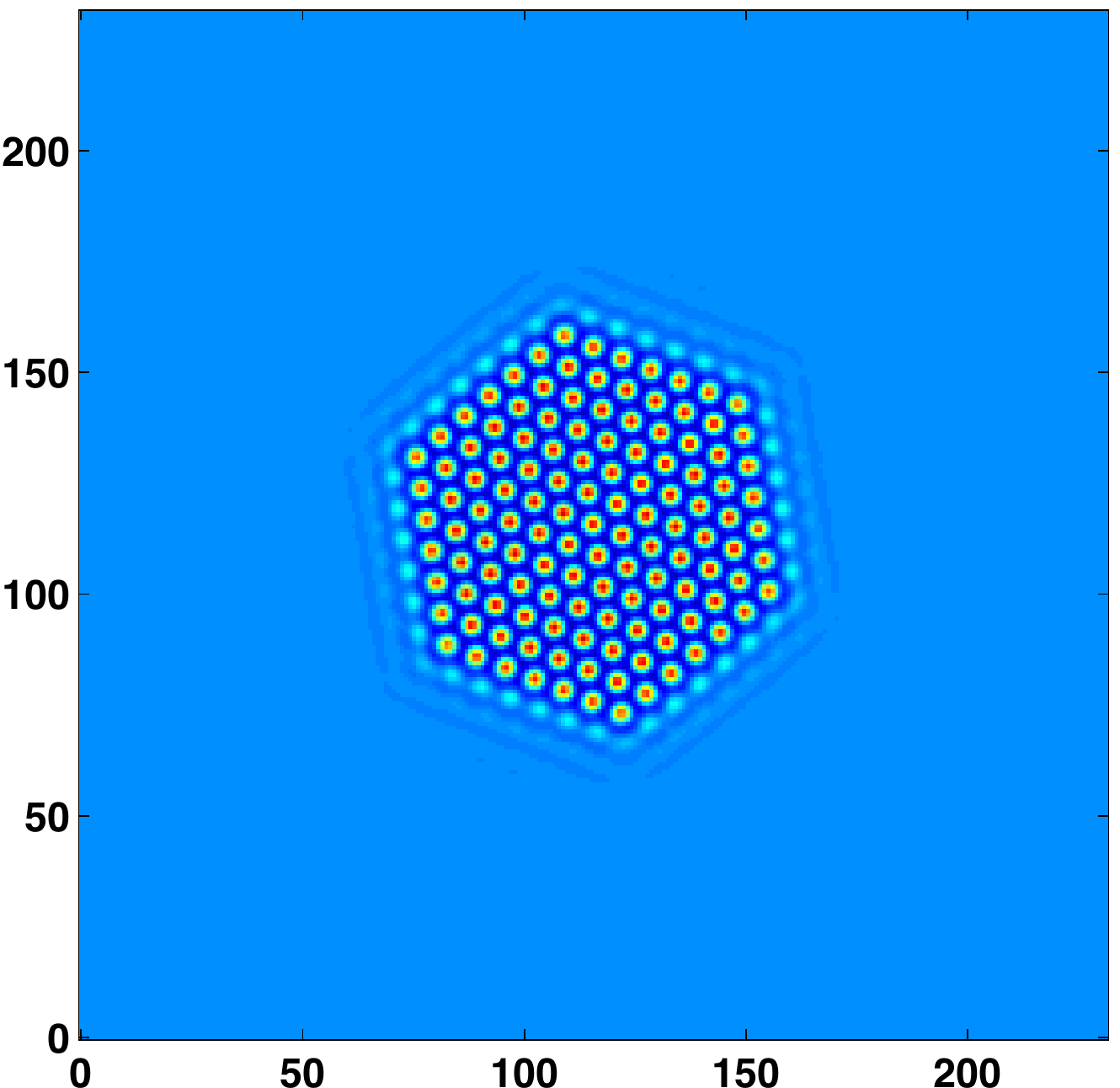}}
\put(0,130){\includegraphics[scale=0.34]{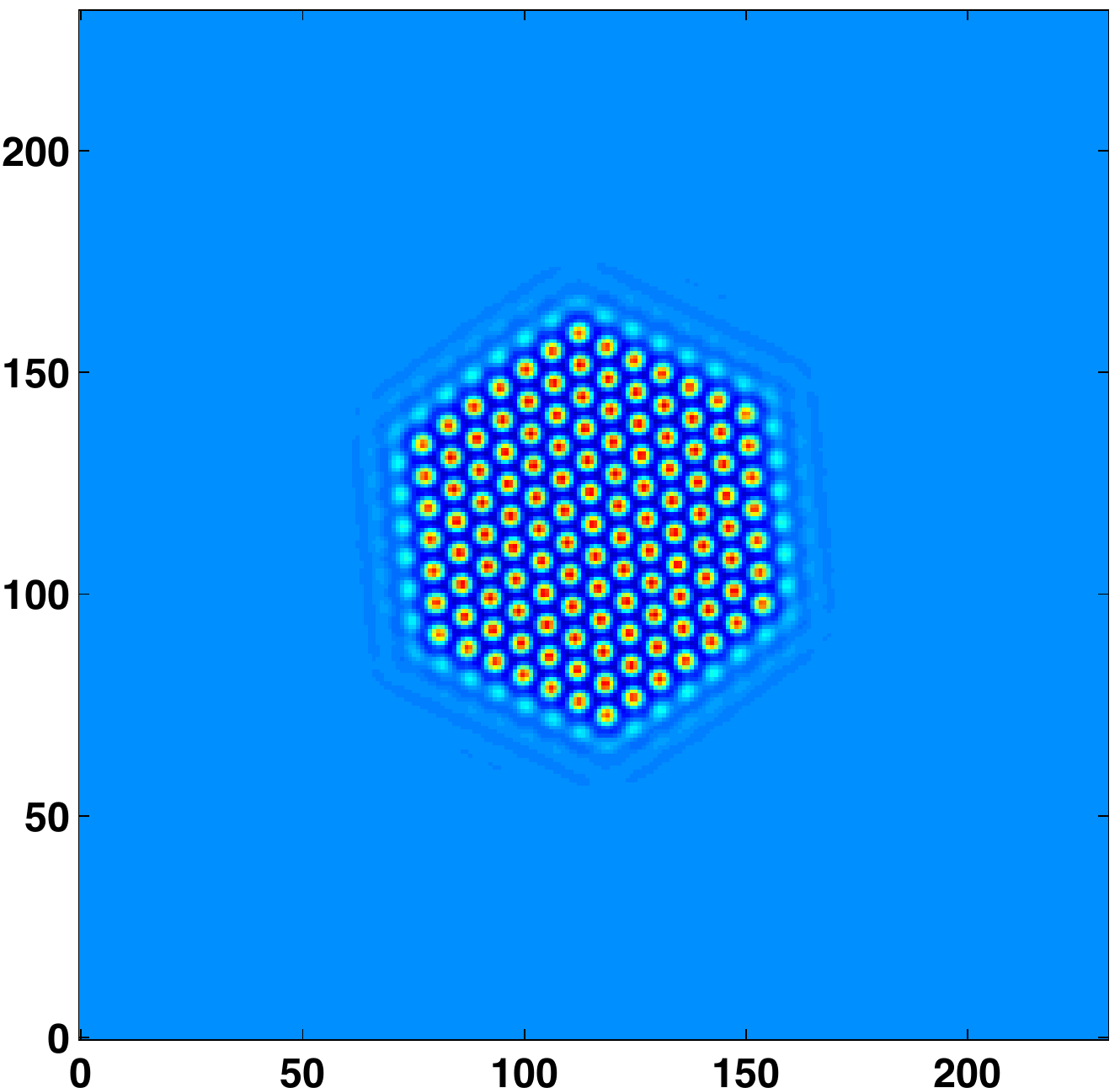}}
\put(130,130){\includegraphics[scale=0.34]{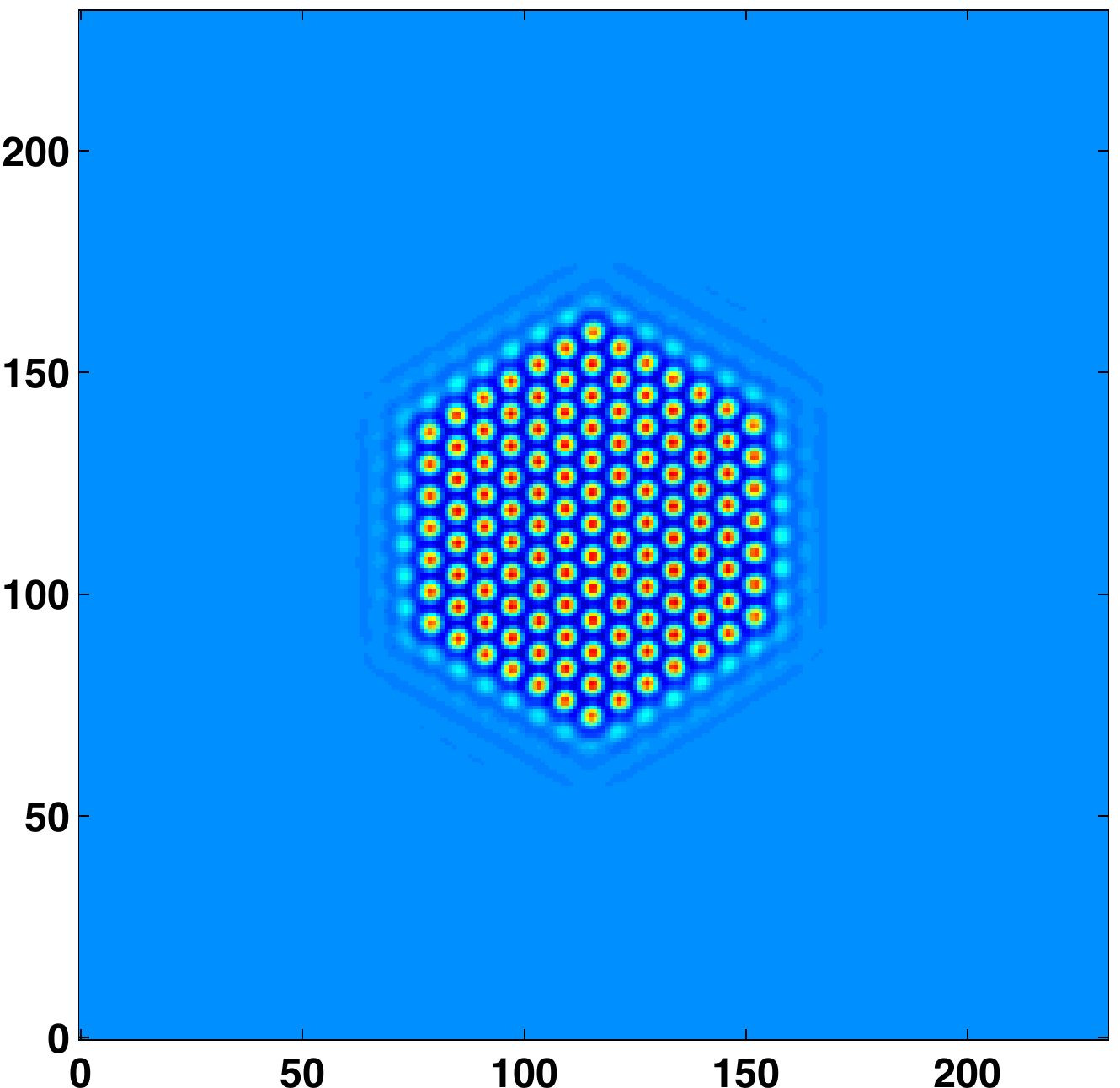}}
\put(260,130){\includegraphics[scale=0.34]{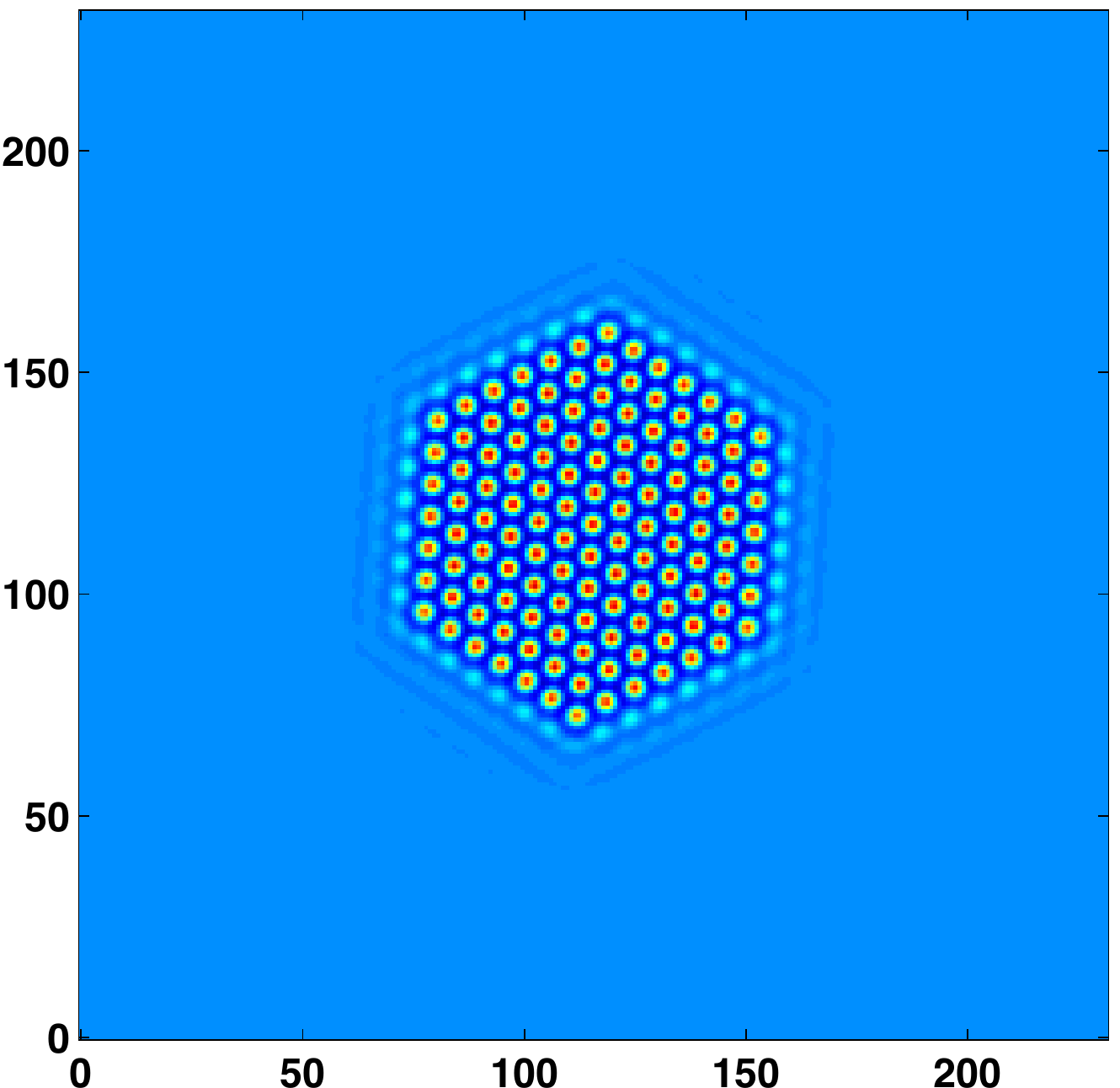}}
\put(0,0){\includegraphics[scale=0.34]{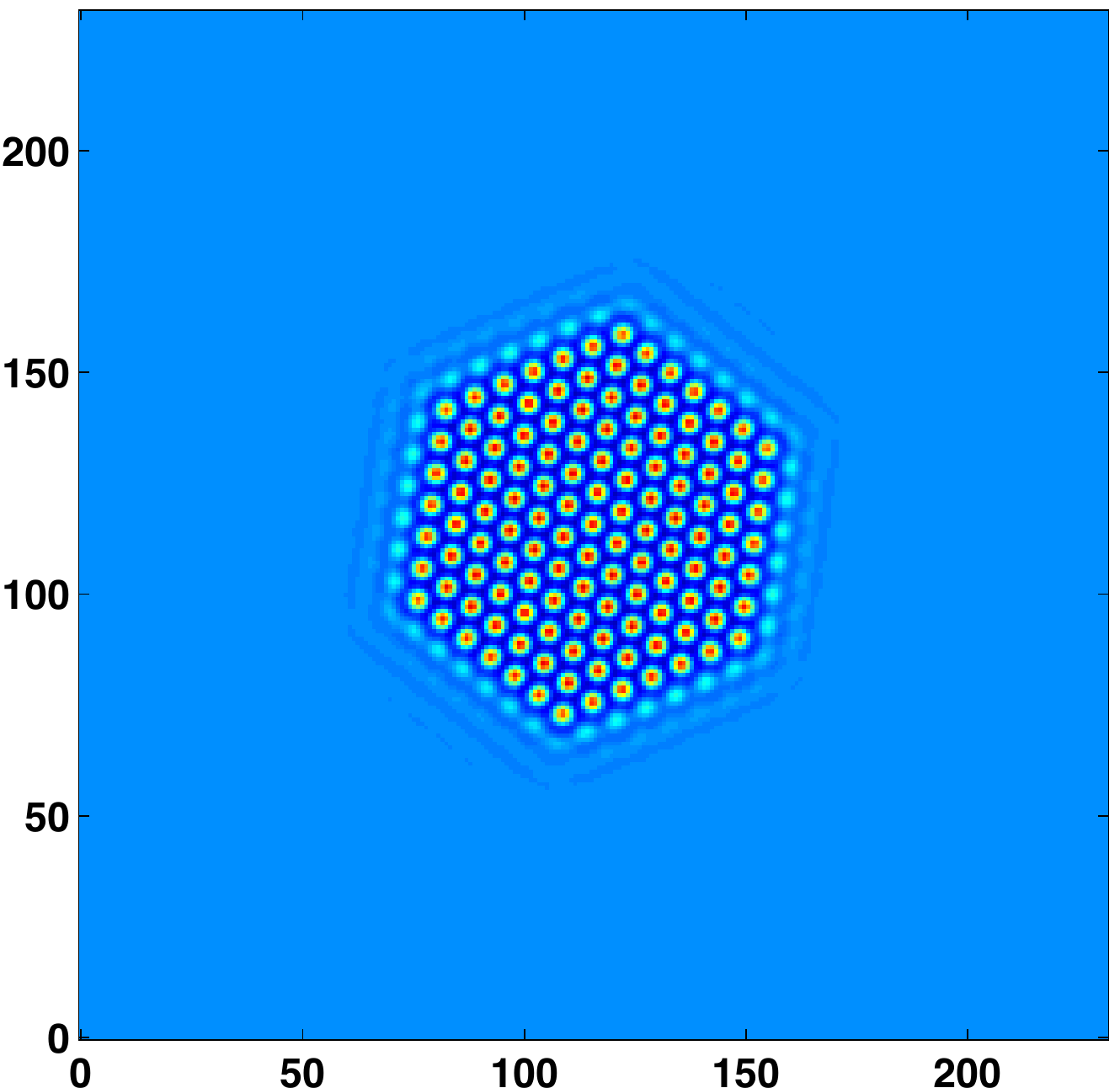}}
\put(130,0){\includegraphics[scale=0.34]{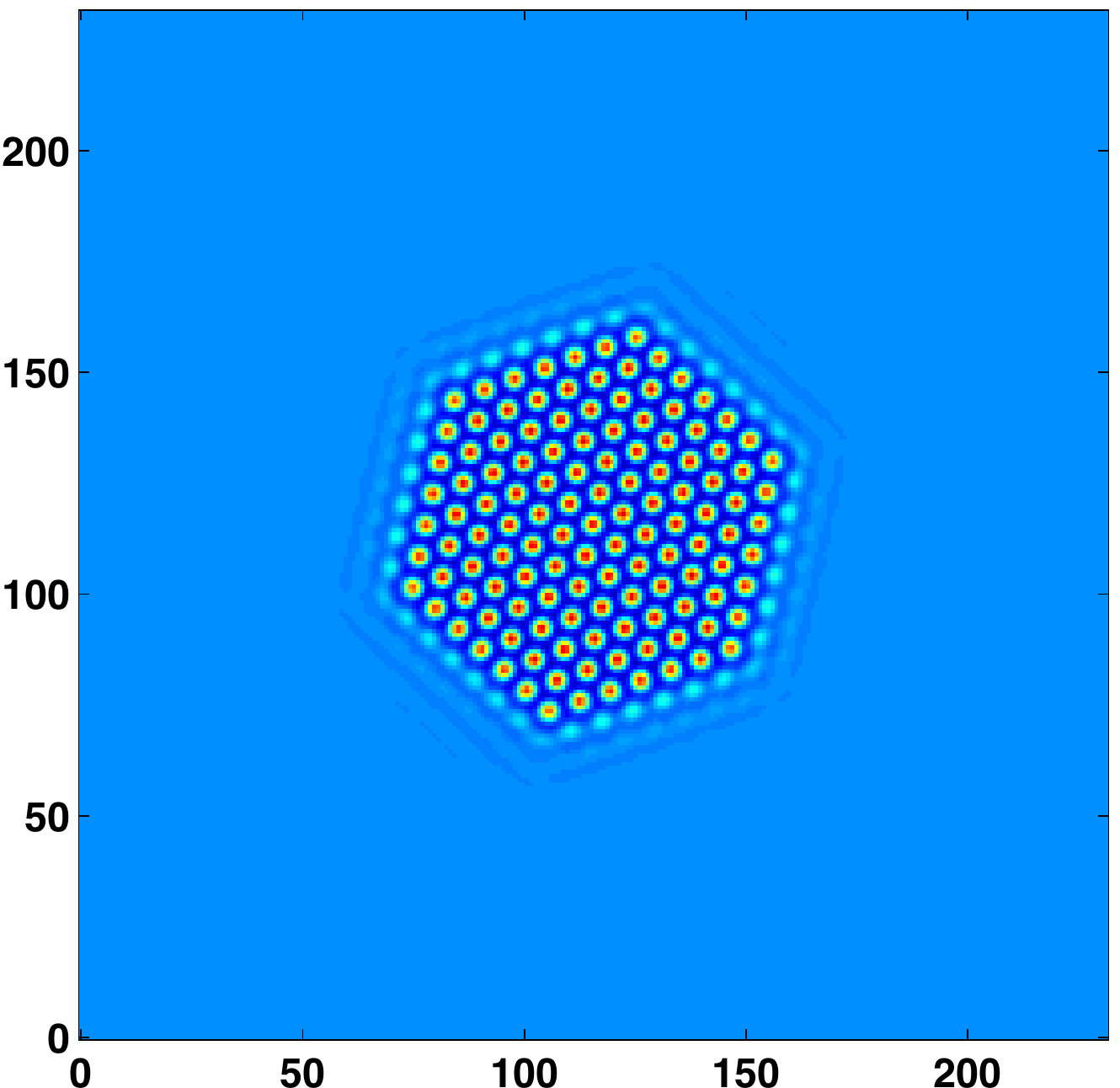}}
\put(260,0){\includegraphics[scale=0.34]{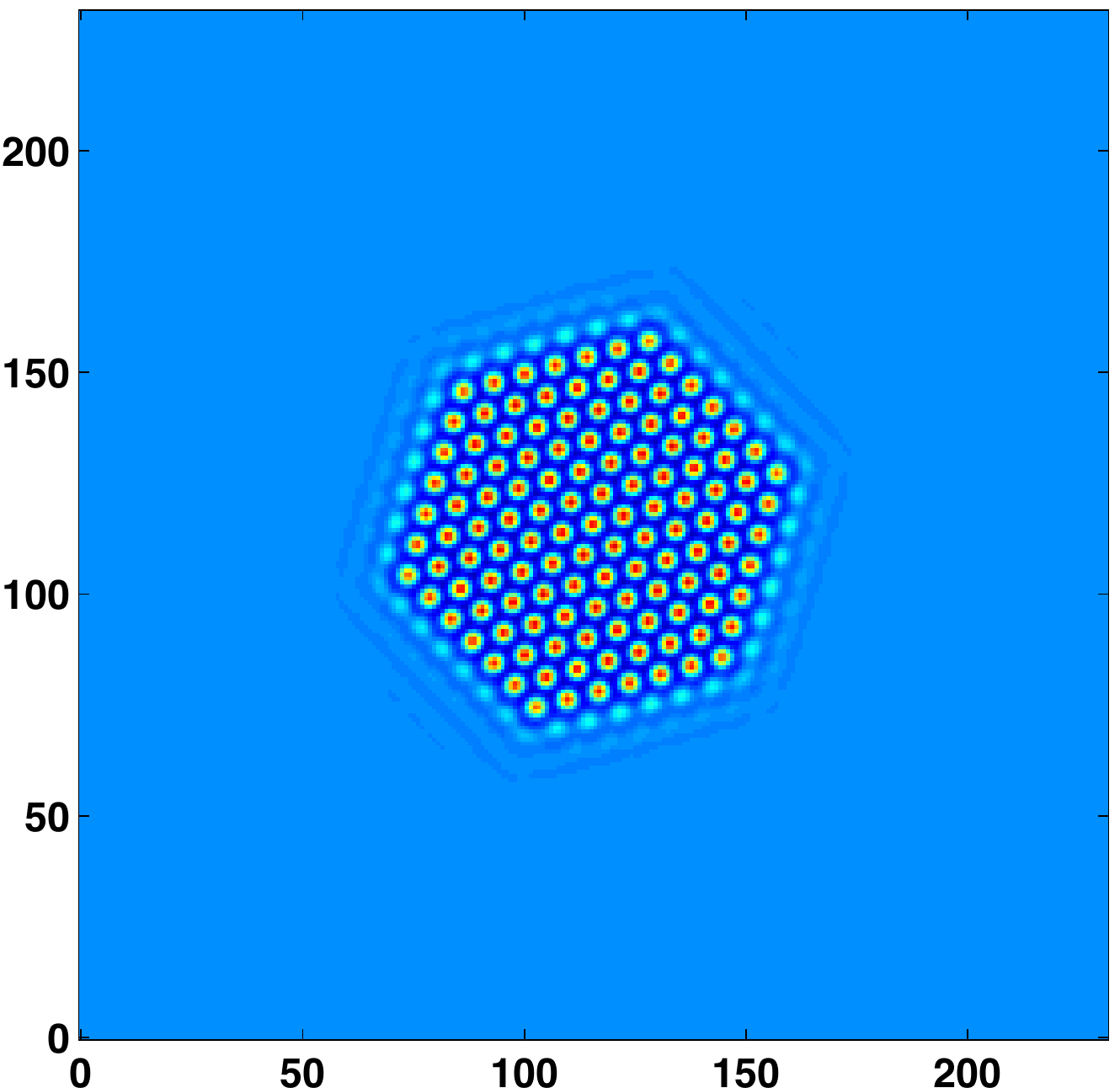}}

\put(15,400){\text{\bf t =800.0}}
\put(145,400){\text{\bf t =1200.0}}
\put(275,400){\text{\bf t =1600.0}}
\put(15,270){\text{\bf t =2000.0}}
\put(145,270){\text{\bf t =2400.0}}
\put(275,270){\text{\bf t =2800.0}}
\put(15,140){\text{\bf t =3200.0}}
\put(145,140){\text{\bf t =3600.0}}
\put(275,140){\text{\bf t =4000.0}}
\put(15,10){\text{\bf t =4400.0}}
\put(145,10){\text{\bf t =48000.0}}
\put(275,10){\text{\bf t =5200.0}}

\end{picture}
\end{center}%\includegraphics{Figs/grow_15000.eps}
\caption{Snapshots of the density field at different times are shown as a nanocrystal is sheared in a driven channel with wall speed $u_{wall,n/0} = \pm 0.1$
as described in Section \ref{flow_PFC}.
The initial conditions correspond to a system at equilibrium with $\bu =0$ and $u_{wall,0/n} =0 $ (Fig. \ref{fig_nucleate}).
The figure shows that the crustal tumbles in the channel while holding its shape fixed. 
The colors (available on-line) correspond to a linear RGB scheme from 0 to 2.9 as shown in Figure \ref{fig_nucleate}.}
\label{fig_shear1}
\end{figure}
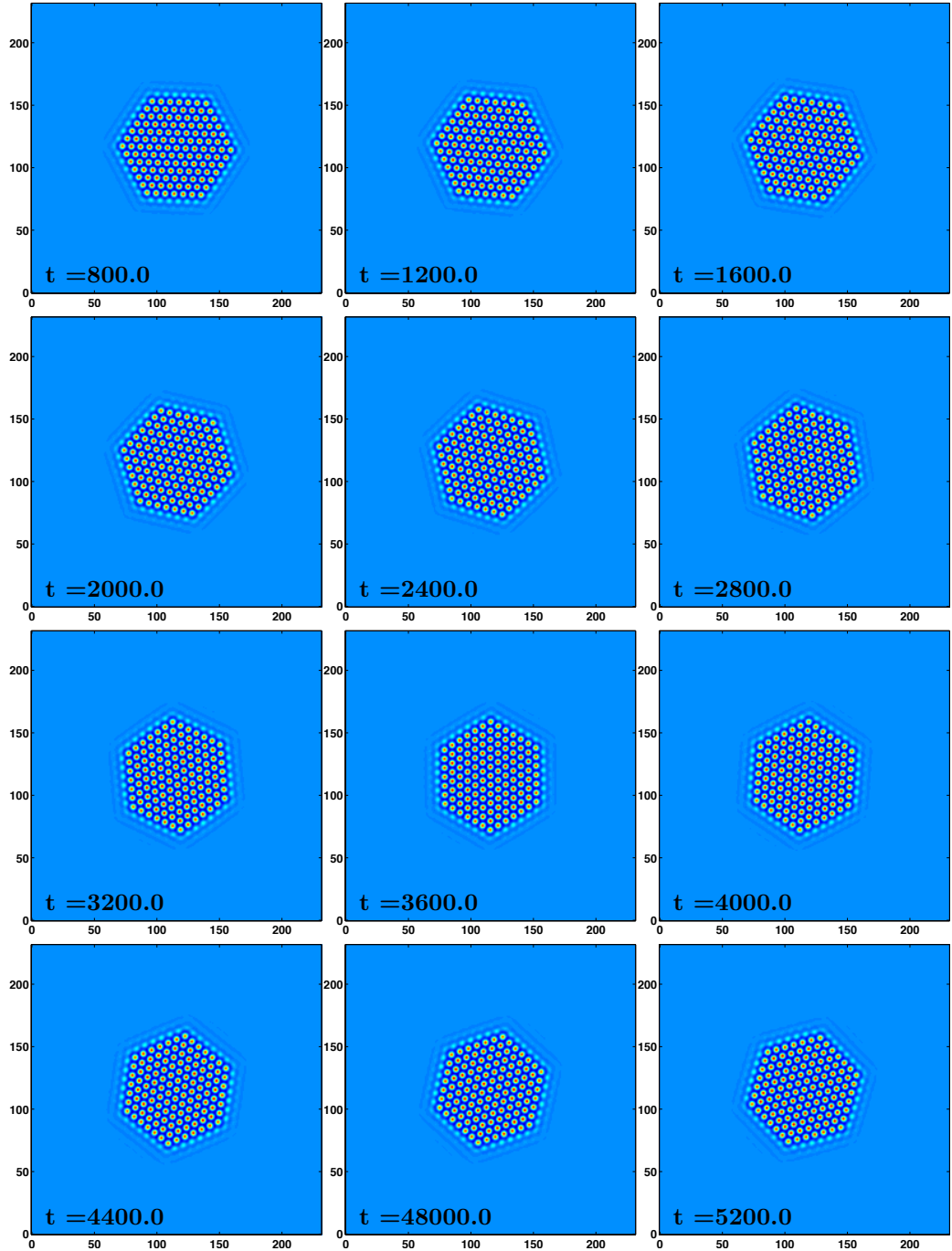

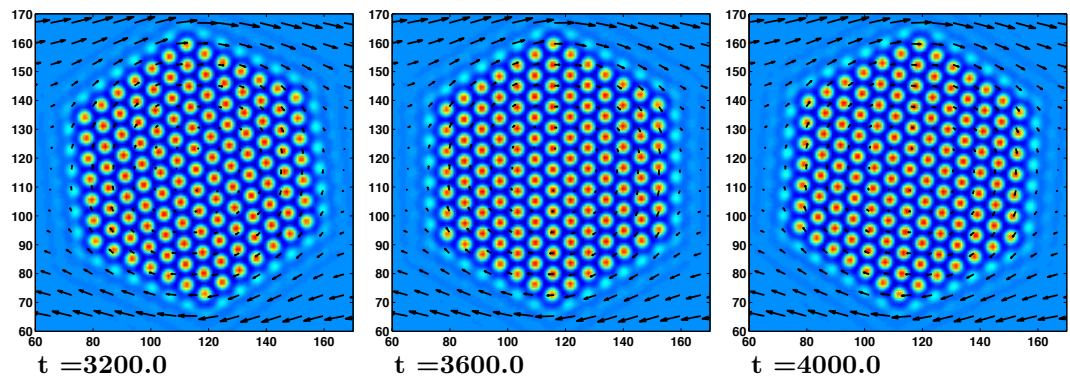
\begin{figure}[!ht]
\begin{center}
\begin{picture}(350,140)
\put(0,5){\includegraphics[scale=0.34]{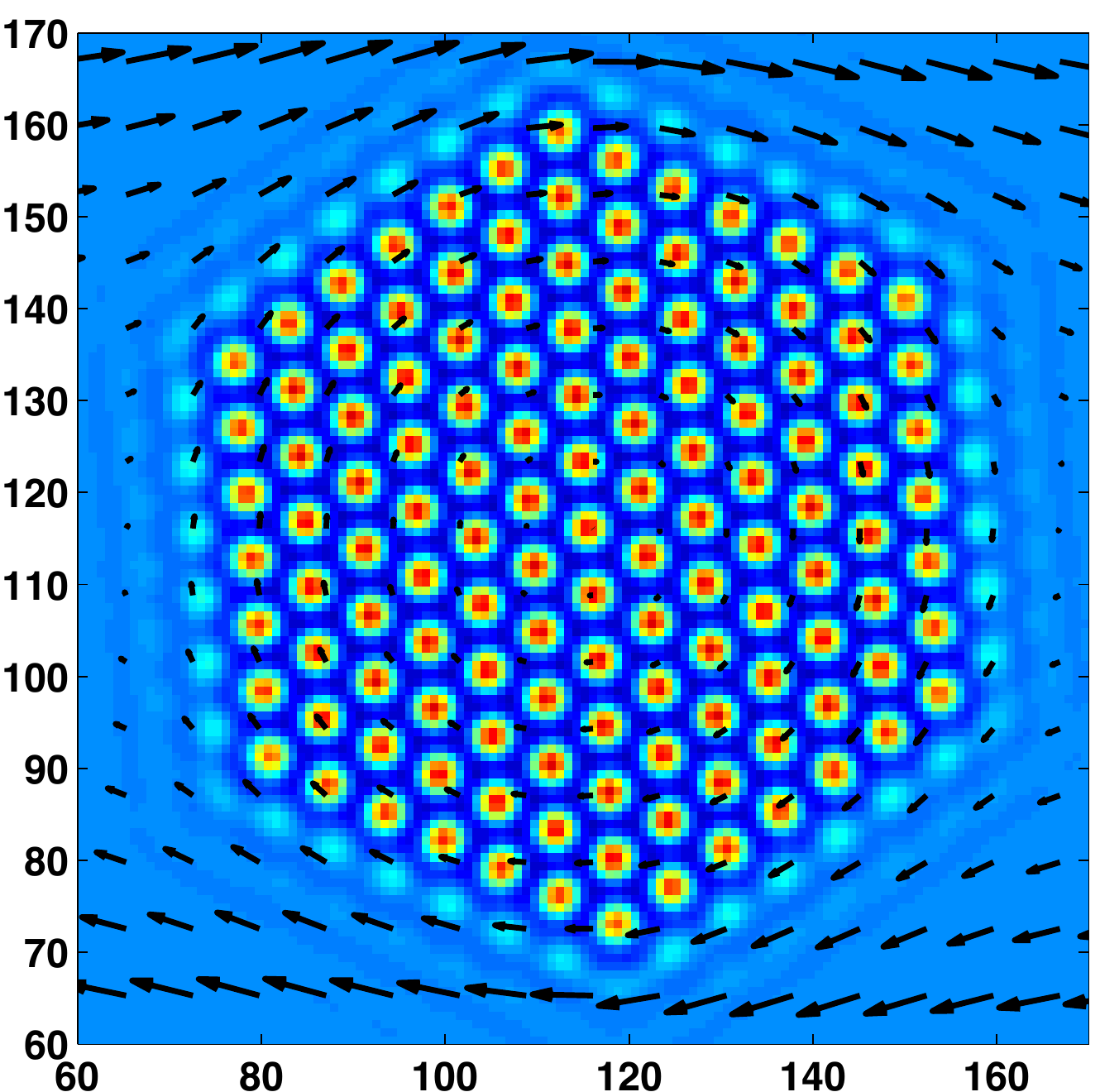}}
\put(135,5){\includegraphics[scale=0.34]{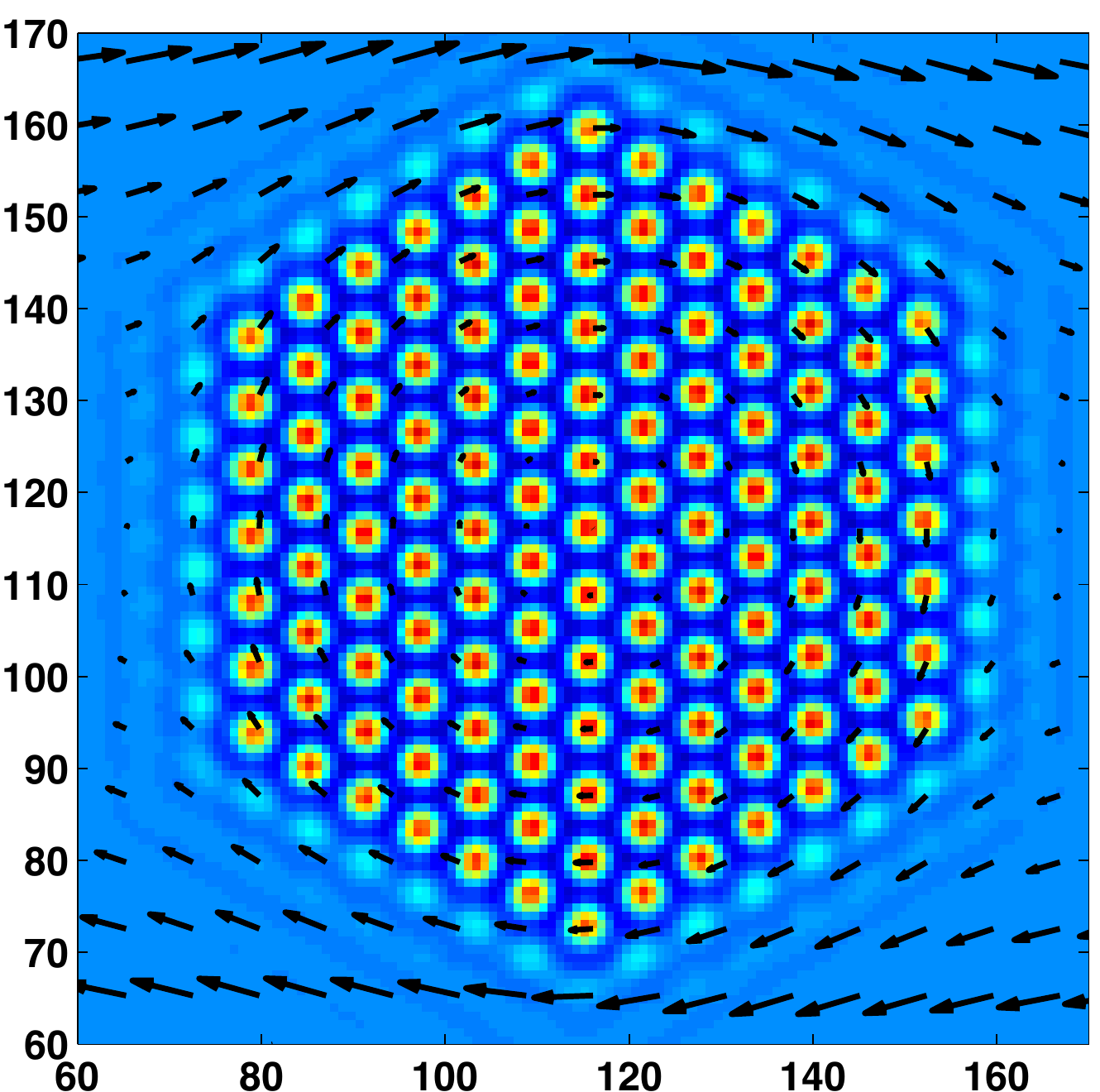}}
\put(270,5){\includegraphics[scale=0.34]{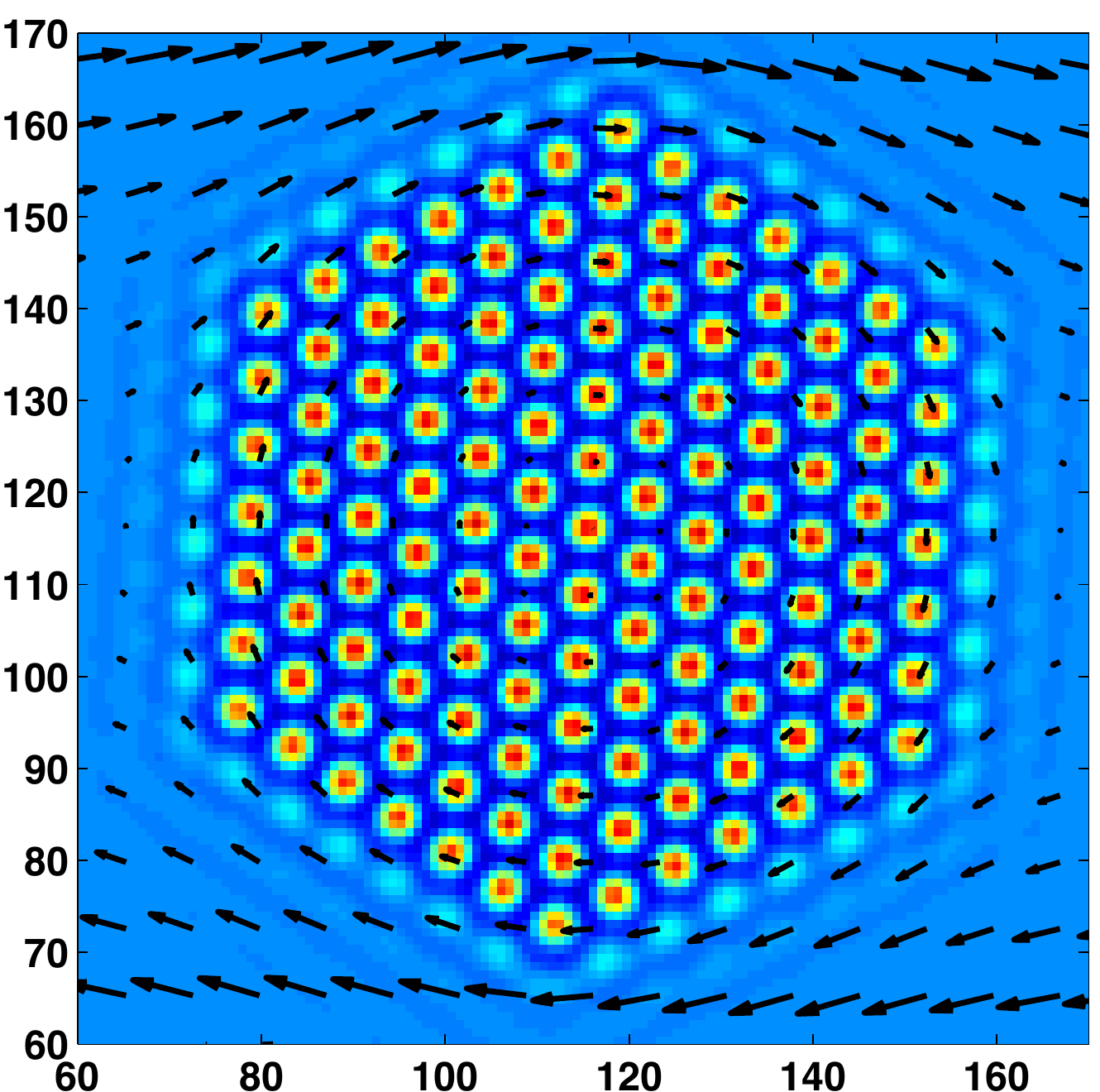}}
\put(10,-5){\text{\bf t =3200.0}}
\put(145,-5){\text{\bf t =3600.0}}
\put(280,-5){\text{\bf t =4000.0}}
\end{picture}
\end{center}%\includegraphics{Figs/grow_15000.eps}
\caption{The flow field superimposed on the density field over a zoomed-in region of the domain at times $t=3200.0$, $3600.0$ and $t=4000$ from the 
simulation presented in Figure \ref{fig_shear1}.
 The figure shows the velocity field does not correspond to that of a plane Couette flow but changes to accommodate 
the tumbling nanocrystal. The colors (available on-line) correspond to a linear RGB scheme from 0 to 2.9 as shown in Figure \ref{fig_nucleate}. }
\label{fig_shear1_v}
\end{figure}

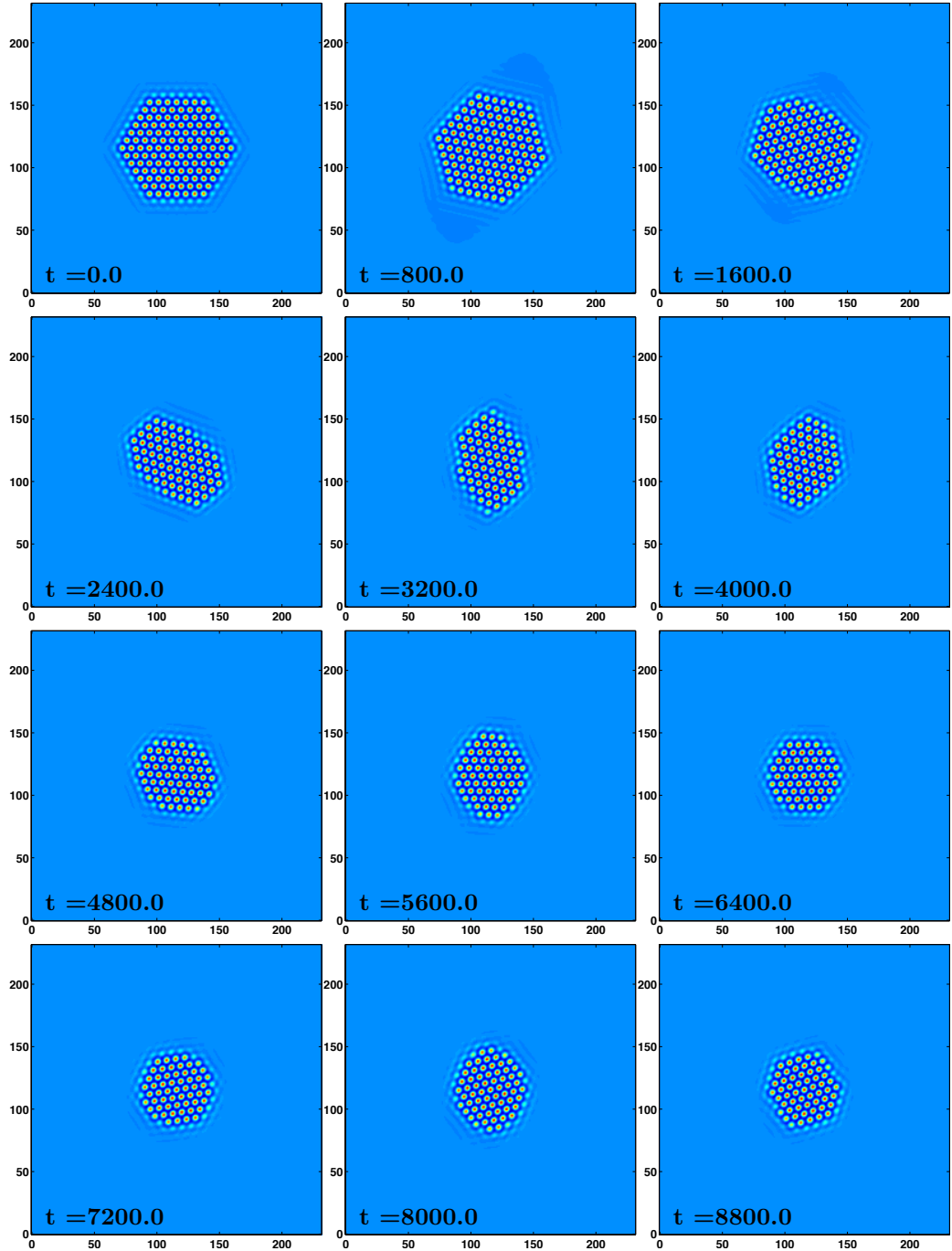
\begin{figure}[!ht]
\begin{center}
\begin{picture}(350,490)
\put(0,390){\includegraphics[scale=0.34]{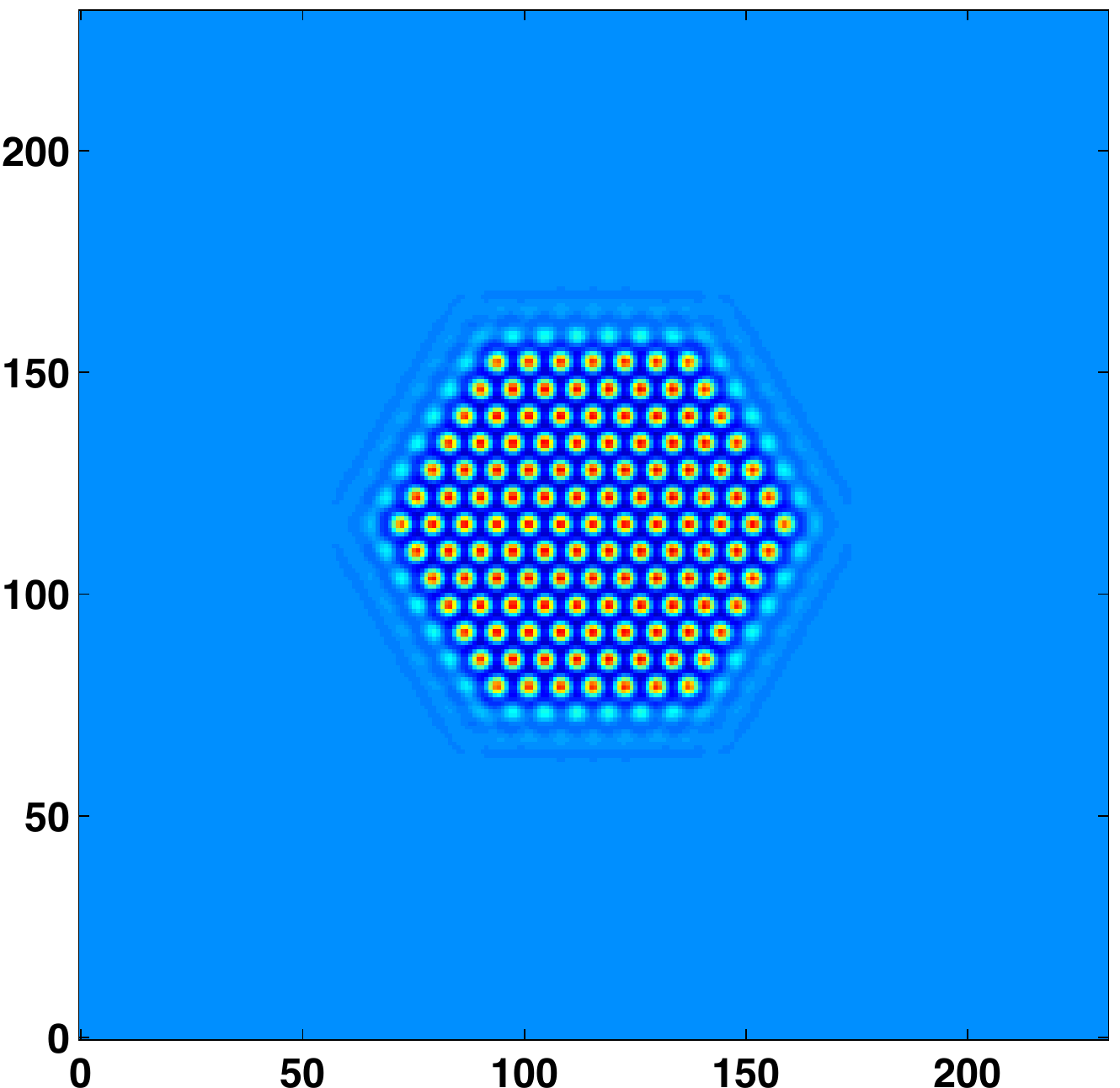}}
\put(130,390){\includegraphics[scale=0.34]{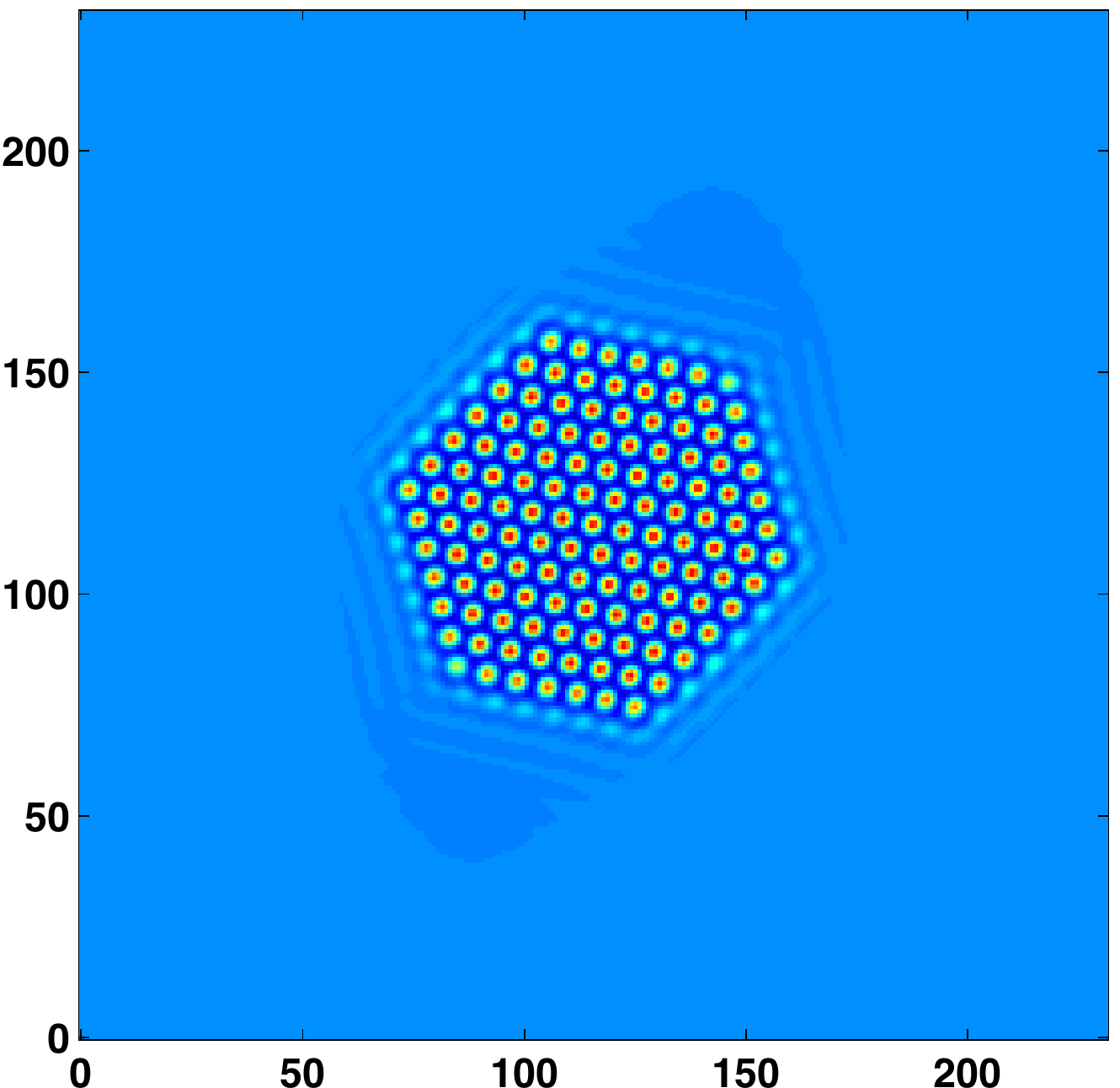}}
\put(260,390){\includegraphics[scale=0.34]{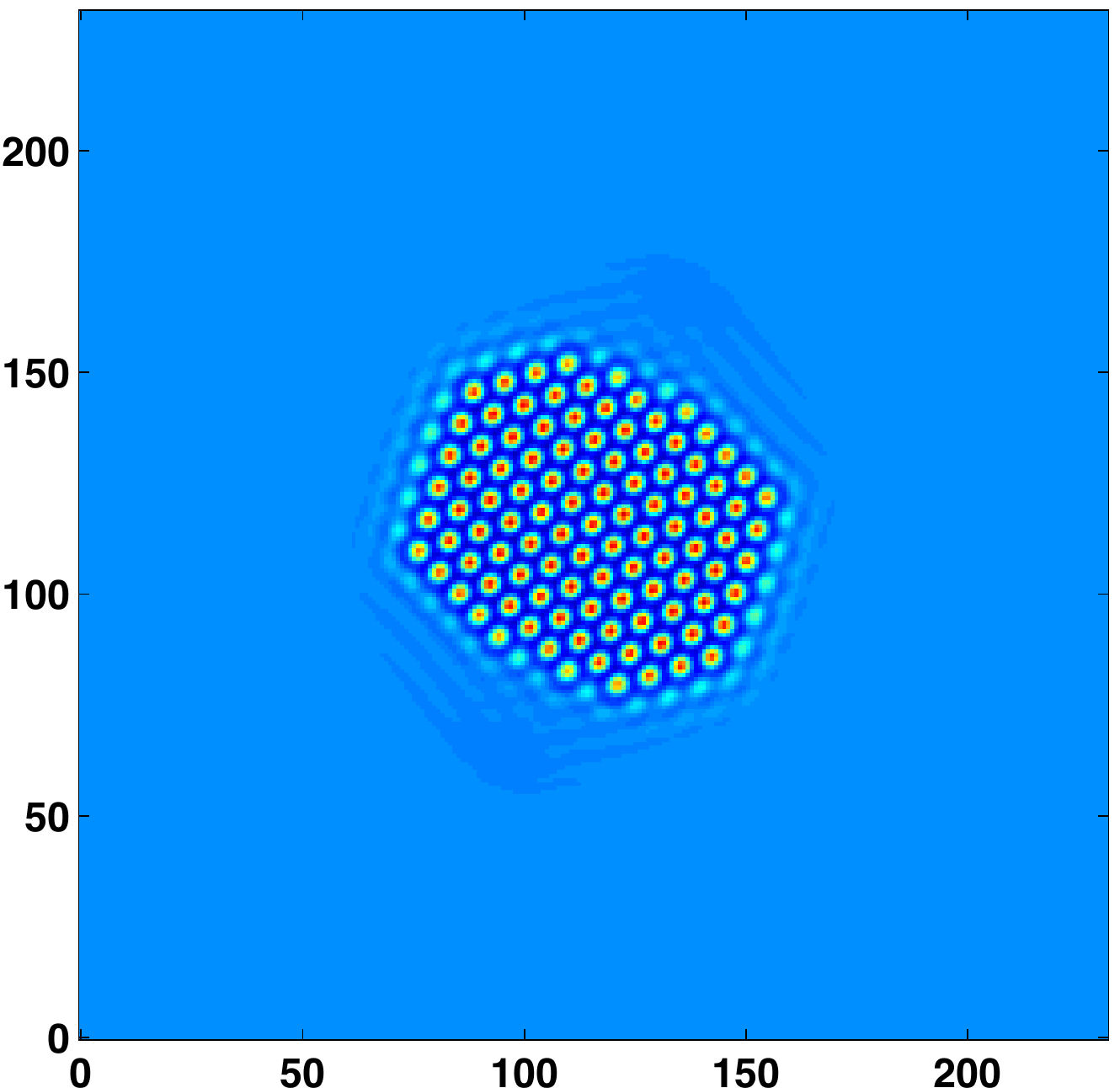}}
\put(0,260){\includegraphics[scale=0.34]{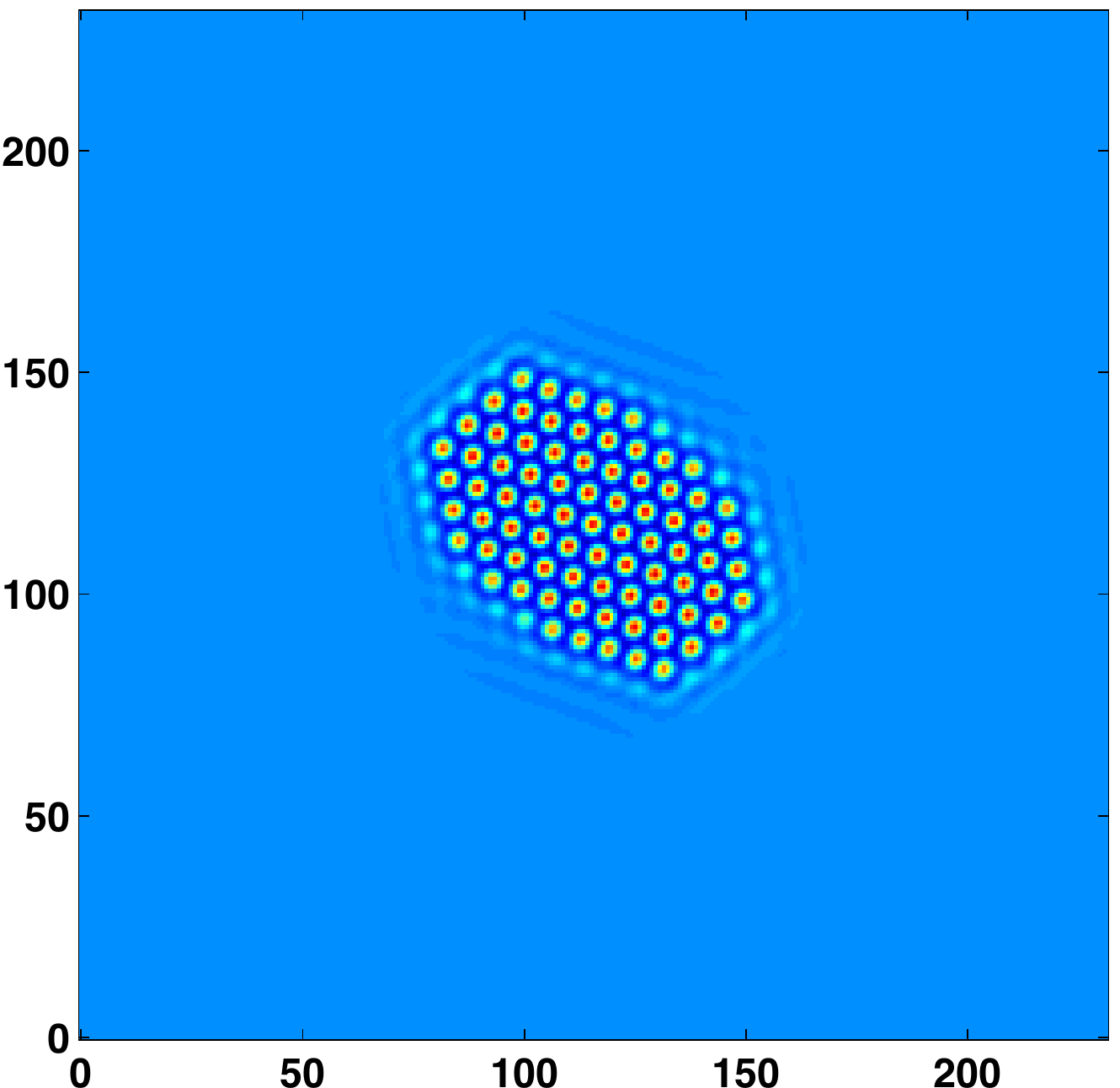}}
\put(130,260){\includegraphics[scale=0.34]{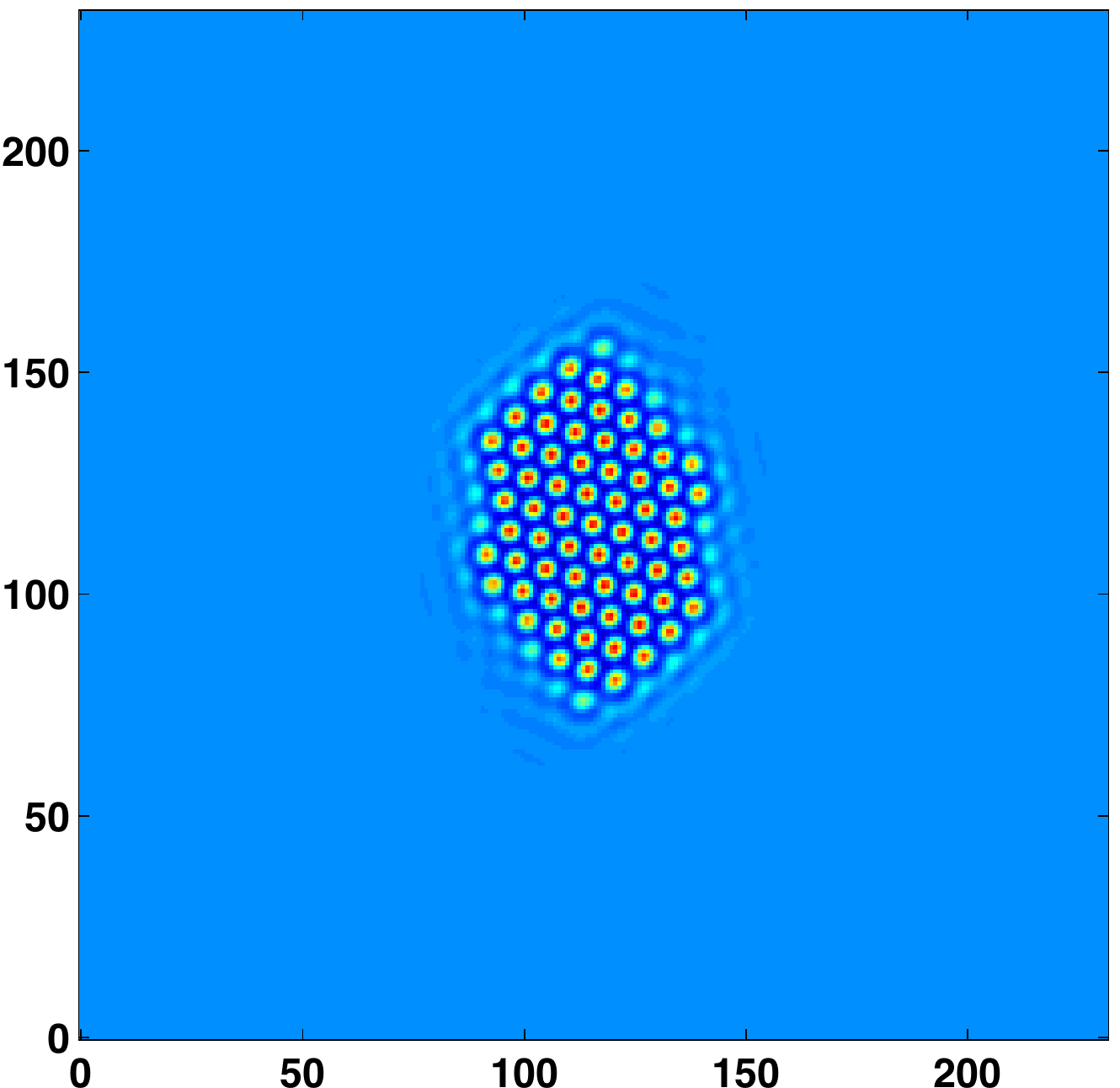}}
\put(260,260){\includegraphics[scale=0.34]{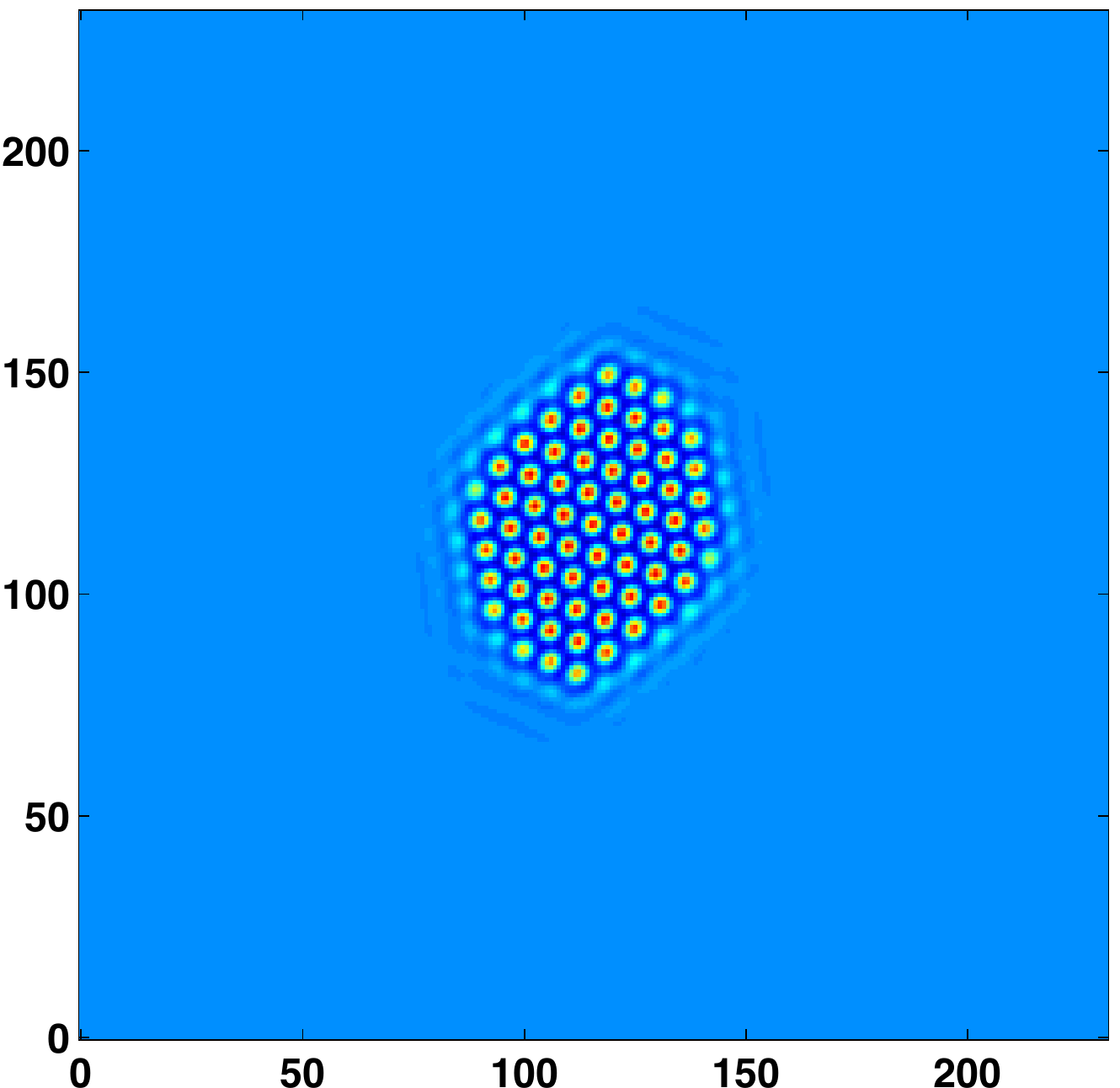}}
\put(0,130){\includegraphics[scale=0.34]{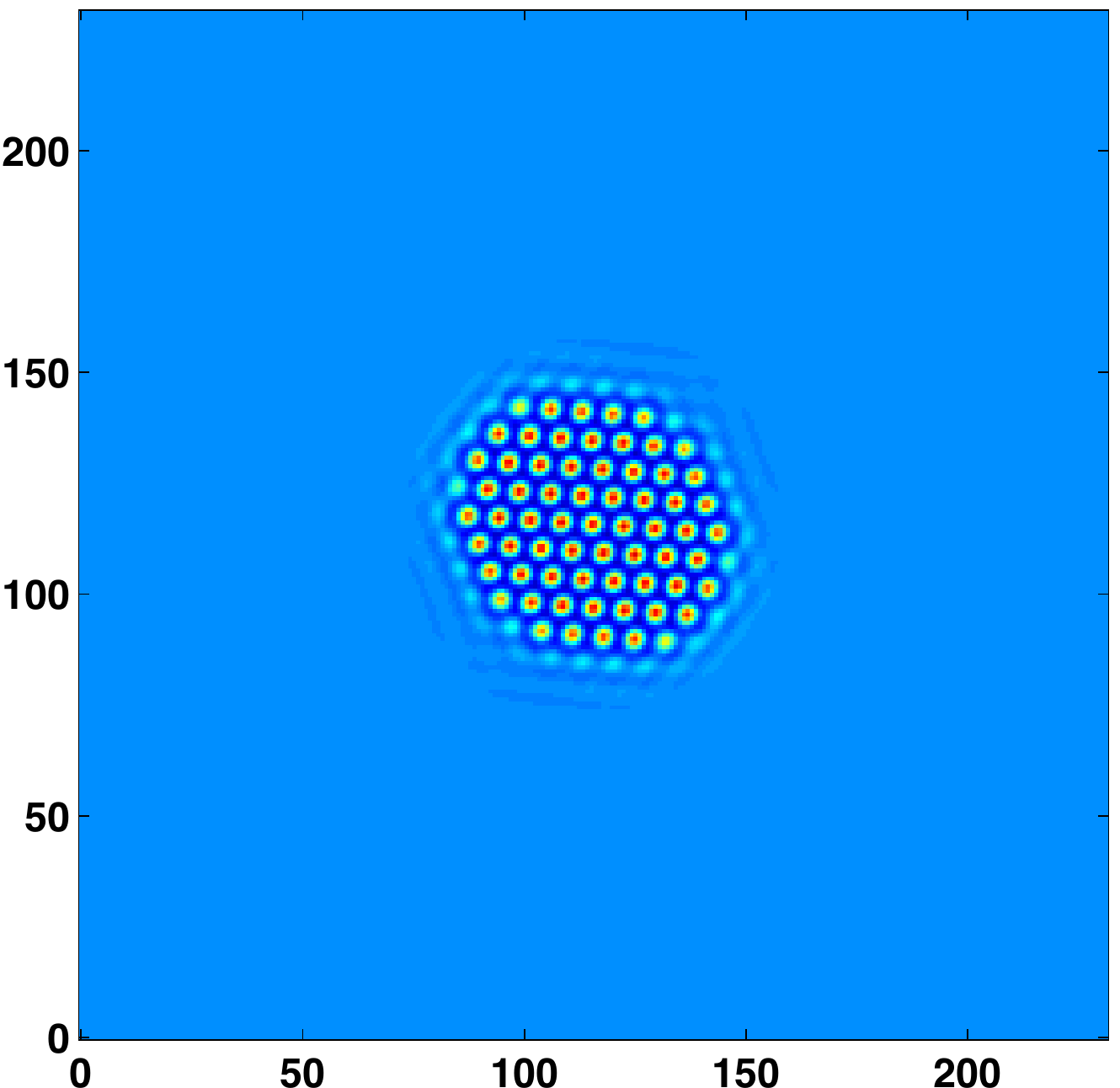}}
\put(130,130){\includegraphics[scale=0.34]{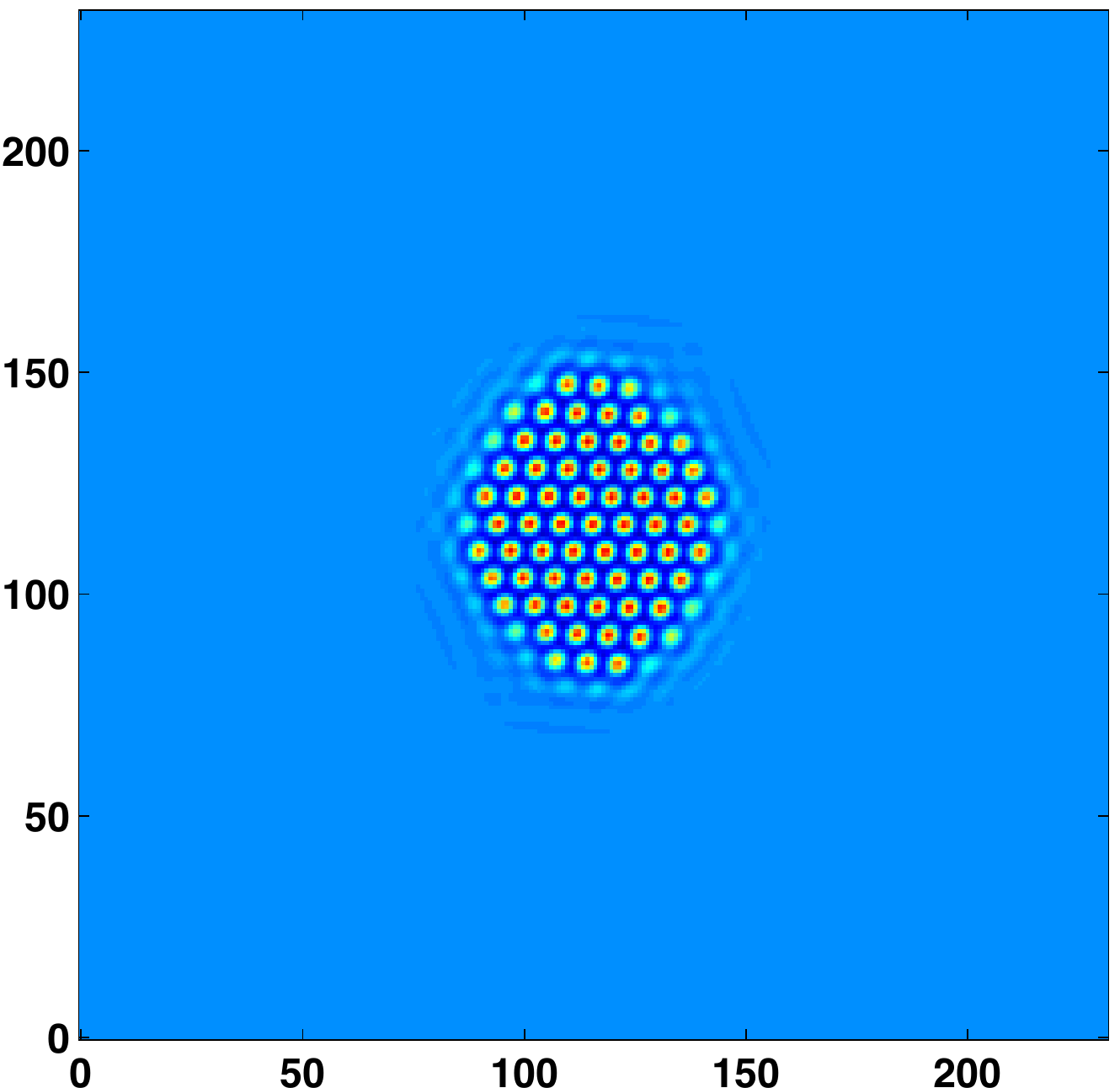}}
\put(260,130){\includegraphics[scale=0.34]{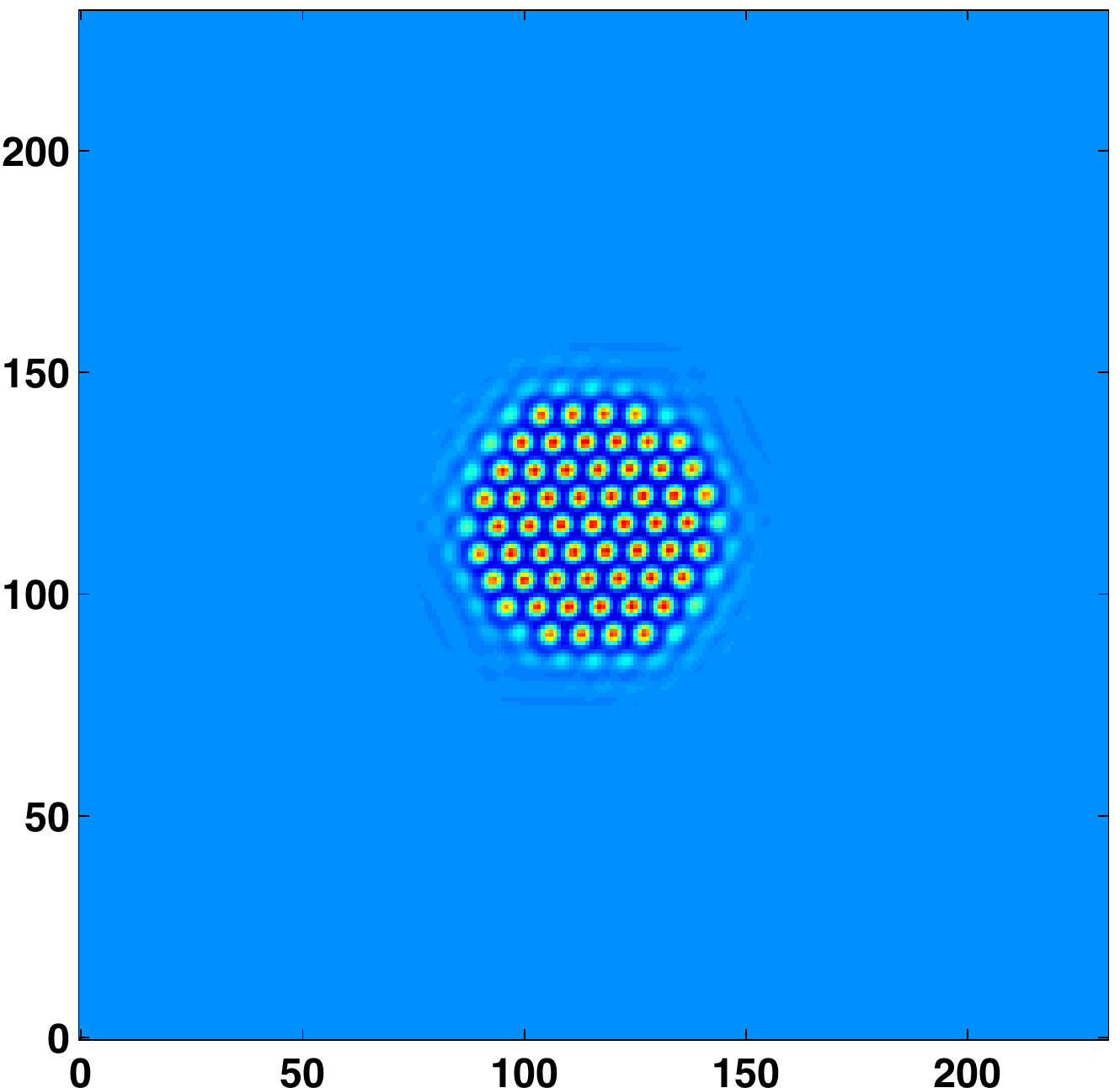}}
\put(0,0){\includegraphics[scale=0.34]{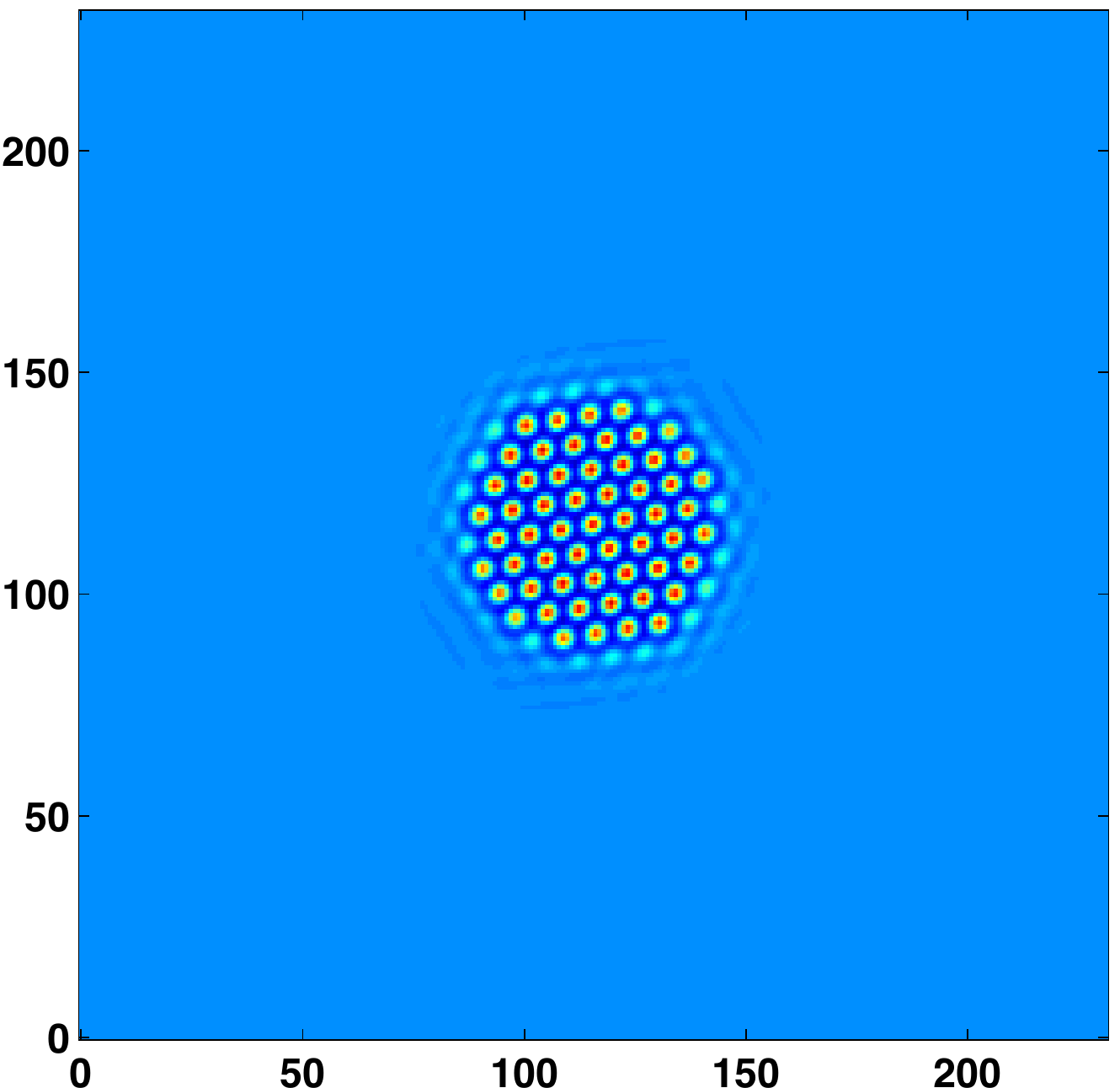}}
\put(130,0){\includegraphics[scale=0.34]{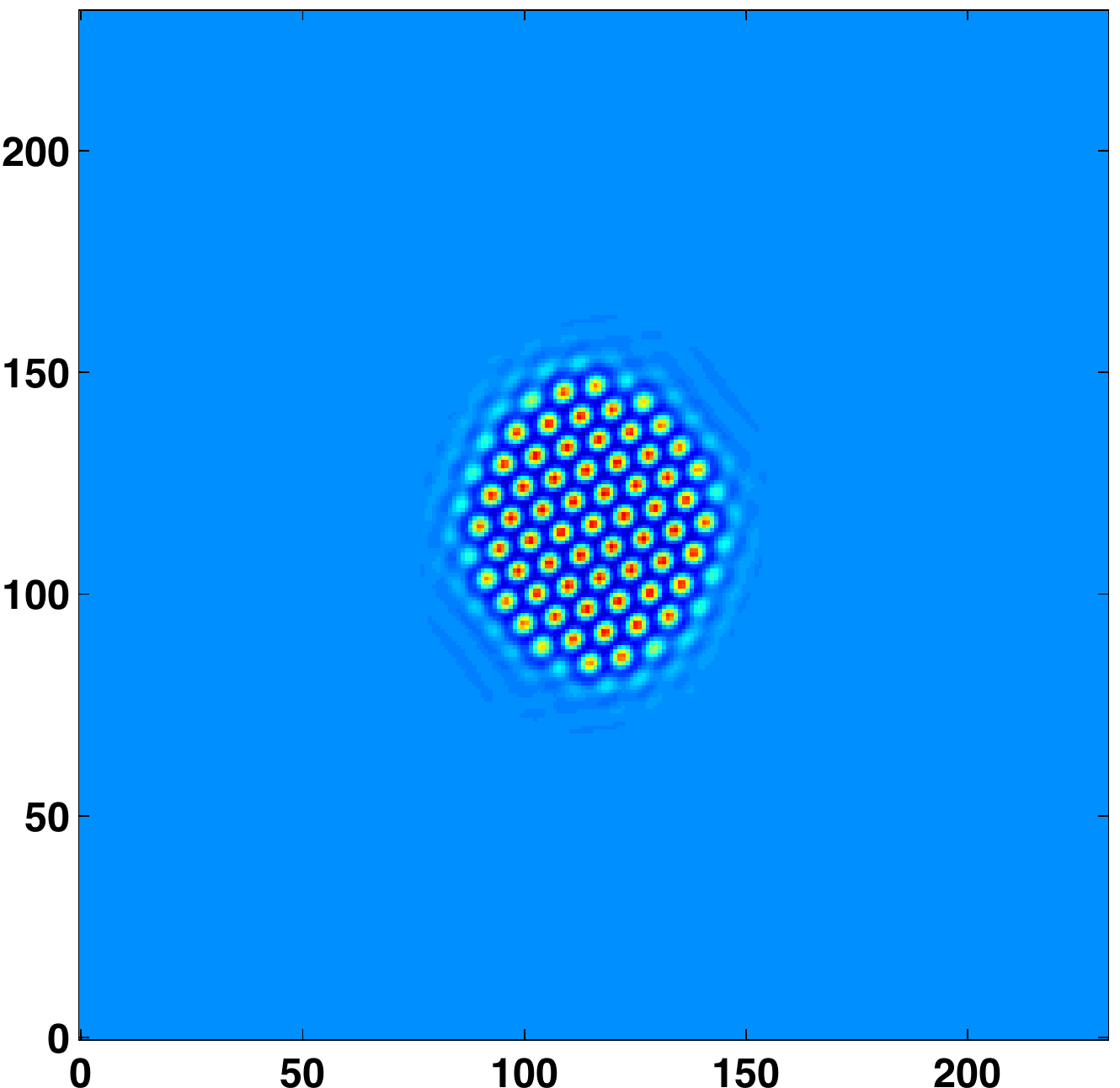}}
\put(260,0){\includegraphics[scale=0.34]{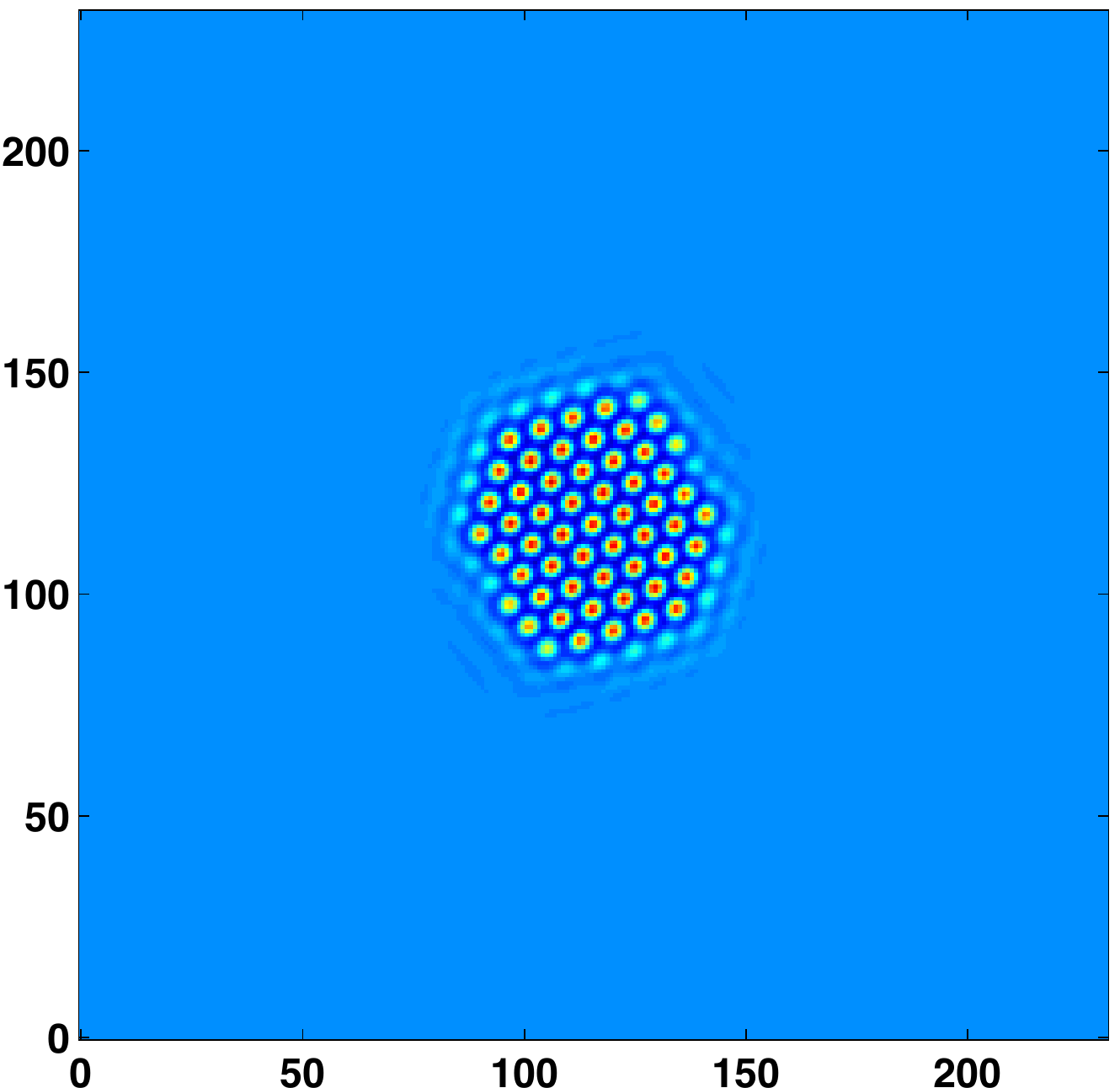}}

\put(15,400){\text{\bf t =0.0}}
\put(145,400){\text{\bf t =800.0}}
\put(275,400){\text{\bf t =1600.0}}
\put(15,270){\text{\bf t =2400.0}}
\put(145,270){\text{\bf t =3200.0}}
\put(275,270){\text{\bf t =4000.0}}
\put(15,140){\text{\bf t =4800.0}}
\put(145,140){\text{\bf t =5600.0}}
\put(275,140){\text{\bf t =6400.0}}
\put(15,10){\text{\bf t =7200.0}}
\put(145,10){\text{\bf t =8000.0}}
\put(275,10){\text{\bf t =8800.0}}

\end{picture}
\end{center}%\includegraphics{Figs/grow_15000.png}
\caption{Snapshots of the density field at different times are shown as a nanocrystal is sheared in a driven channel with wall speed $u_{wall,n/0} =\pm0.5$
as described in Section  \ref{flow_PFC}.
The initial conditions correspond to a system at equilibrium with $\bu =0$ and $u_{wall,0/n} =0$ (Fig. \ref{fig_nucleate}).
The figure shows that the crystal initially shrinks to a smaller size and then stabilizes in size as it tumbles in the channel.
 The colors (available on-line) correspond to a linear RGB scheme from 0 to 2.9 as shown in Figure \ref{fig_nucleate}.}
\label{fig_shear2}
\end{figure}

\end{document}